\documentclass[a4paper,12pt]{amsart}
\usepackage{amsmath,amsthm,amssymb}
\usepackage[dvipdfmx]{graphicx}
\usepackage[all]{xy}
\usepackage{graphicx}
\usepackage{here}
\usepackage{color}
\usepackage[top=30truemm,bottom=30truemm,left=25truemm,right=25truemm]{geometry}
\usepackage{ytableau}

\newtheorem{theo}{Theorem}[subsection]
\newtheorem{defi}[theo]{Definition}
\newtheorem{lem}[theo]{Lemma}
\newtheorem{rem}[theo]{Remark}
\newtheorem{prop}[theo]{Proposition}
\newtheorem{cor}[theo]{Corollary}

\newtheorem{ex}[theo]{Example}

\newcommand{\nc}{\newcommand}
\nc{\on}{\operatorname}
\nc{\C}{\mathbb{C}}
\nc{\R}{\mathbb{R}}
\nc{\Q}{\mathbb{Q}}
\nc{\Z}{\mathbb{Z}}
\nc{\N}{\mathbb{N}}

\nc{\bbH}{\mathbb{H}}
\nc{\bbK}{\mathbb{K}}

\nc{\bfa}{\mathbf{a}}
\nc{\bfA}{\mathbf{A}}
\nc{\bfB}{\mathbf{B}}
\nc{\bfb}{\mathbf{b}}
\nc{\bff}{\mathbf{f}}
\nc{\bfi}{\mathbf{i}}
\nc{\bfK}{\mathbf{K}}
\nc{\bfk}{\mathbf{k}}
\nc{\bfx}{\mathbf{x}}
\nc{\bfy}{\mathbf{y}}
\nc{\bfone}{\boldsymbol 1}
\nc{\bfeta}{\boldsymbol \eta}
\nc{\bfkappa}{\boldsymbol \kappa}
\nc{\bfsigma}{\boldsymbol \sigma}
\nc{\bfvarsigma}{\boldsymbol \varsigma}
\nc{\bfzeta}{\boldsymbol \zeta}

\nc{\A}{\mathbf{A}}
\nc{\U}{\mathbf{U}}

\nc{\clB}{\mathcal{B}}
\nc{\clC}{\mathcal{C}}
\nc{\clF}{\mathcal{F}}
\nc{\clI}{\mathcal{I}}
\nc{\clL}{\mathcal{L}}
\nc{\clM}{\mathcal{M}}
\nc{\clO}{\mathcal{O}}
\nc{\clT}{\mathcal{T}}
\nc{\clU}{\mathcal{U}}
\nc{\clUi}{\clU^{\imath}}
\nc{\clX}{\mathcal{X}}
\nc{\clXs}{\clX_{\mathrm{s}}}
\nc{\clY}{\mathcal{Y}}
\nc{\clYs}{\clY_{\mathrm{s}}}
\nc{\clW}{\mathcal{W}}
\nc{\clZ}{\mathcal{Z}}

\nc{\fra}{\mathfrak{a}}
\nc{\frh}{\mathfrak{h}}
\nc{\g}{\mathfrak{g}}
\nc{\frgl}{\mathfrak{gl}}
\nc{\frk}{\mathfrak{k}}
\nc{\fram}{\mathfrak{m}}
\nc{\frn}{\mathfrak{n}}
\nc{\frp}{\mathfrak{p}}
\nc{\frs}{\mathfrak{s}}
\nc{\frt}{\mathfrak{t}}
\nc{\frsl}{\mathfrak{sl}}
\nc{\frso}{\mathfrak{so}}
\nc{\frsp}{\mathfrak{sp}}
\nc{\frsu}{\mathfrak{su}}
\nc{\fru}{\mathfrak{u}}
\nc{\frz}{\mathfrak{z}}

\nc{\fin}{\mathrm{fin}}

\nc{\inv}{^{-1}}
\nc{\qu}{\quad}

\nc{\qqu}{\qquad}
\nc{\la}{\langle}
\nc{\ra}{\rangle}

\nc{\Ker}{\on{Ker}}
\nc{\im}{\on{Im}}
\nc{\Hom}{\on{Hom}}
\nc{\End}{\on{End}}
\nc{\Span}{\on{Span}}
\nc{\id}{\on{id}}
\nc{\Aut}{\on{Aut}}
\nc{\ad}{\on{ad}}
\nc{\ch}{\on{ch}}
\nc{\sgn}{\on{sgn}}
\nc{\tot}{\on{tot}}
\nc{\rk}{\on{rk}}
\nc{\rank}{\on{rank}}
\nc{\Wt}{\on{Wt}}
\nc{\diag}{\on{diag}}
\nc{\Mat}{\on{Mat}}
\nc{\tr}{\on{tr}}
\nc{\Diag}{\on{Diag}}
\nc{\GL}{\on{GL}}
\nc{\SO}{\on{SO}}
\nc{\Sp}{\on{Sp}}
\nc{\gr}{\on{gr}}
\nc{\Ind}{\on{Ind}}
\nc{\Res}{\on{Res}}
\nc{\wt}{\on{wt}}
\nc{\lt}{\on{lt}}
\nc{\lc}{\on{lc}}
\nc{\Br}{\on{Br}}
\nc{\norm}{\on{norm}}
\nc{\amp}{\on{amp}}
\nc{\Int}{\on{int}}
\nc{\Inv}{\on{inv}}
\nc{\poly}{\on{poly}}
\nc{\pr}{\on{pr}}
\nc{\sh}{\on{sh}}
\nc{\std}{\on{std}}
\nc{\triv}{\on{triv}}

\nc{\AI}{\mathrm{AI}}
\nc{\AIII}{\mathrm{AIII}}
\nc{\CR}{\mathrm{CR}}
\nc{\ev}{\mathrm{ev}}
\nc{\odd}{\mathrm{odd}}
\nc{\Par}{\mathrm{Par}}
\nc{\RS}{\mathrm{RS}}
\nc{\SST}{\mathrm{SST}}
\nc{\STab}{\mathrm{ST}}

\nc{\ol}{\overline}
\nc{\ul}{\underline}
\nc{\hf}{\frac{1}{2}}
\nc{\vphi}{\varphi}
\nc{\vrho}{\varrho}
\nc{\vpi}{\varpi}
\nc{\vep}{\varepsilon}
\nc{\eps}{\epsilon}
\nc{\lm}{\lambda}
\nc{\til}{\widetilde}
\nc{\simga}{\sigma}

\nc{\IF}{\text{ if }}
\nc{\AND}{\text{ and }}
\nc{\OR}{\text{ or }}
\nc{\OW}{\text{ otherwise}}
\nc{\lowerterms}{\text{(lower terms)}}
\nc{\higherterms}{\text{(higher terms)}}
\nc{\lex}{\text{lex}}
\nc{\sesi}{\text{ss}}
\nc{\ST}{\text{ such that }}
\nc{\Forsome}{\text{ for some }}
\nc{\Forall}{\text{ for all }}

\nc{\Atil}{\widetilde{A}}
\nc{\Btil}{\widetilde{B}}
\nc{\Etil}{\widetilde{E}}
\nc{\Ftil}{\widetilde{F}}
\nc{\Itil}{\widetilde{I}}
\nc{\ttil}{\widetilde{t}}
\nc{\vtil}{\widetilde{v}}
\nc{\Vtil}{\widetilde{V}}
\nc{\Xtil}{\widetilde{X}}
\nc{\Ytil}{\widetilde{Y}}
\nc{\lmtil}{\widetilde{\lm}}
\nc{\clBtil}{\widetilde{\clB}}

\nc{\Ui}{\U^{\imath}}
\nc{\Uidot}{\dot{\U}^\imath}
\nc{\UidotA}{\dot{\U}^\imath_{\bfA}}
\nc{\bfBidot}{\dot{\bfB}^\imath}
\nc{\clBdot}{\dot{\clB}}
\nc{\clBidot}{\dot{\clB}^\imath}
\nc{\clLidot}{\dot{\clL}^\imath}
\nc{\Bidot}{\dot{\bfB}^\imath}
\nc{\bfBdot}{\dot{\bfB}}
\nc{\taui}{\tau^{\imath}}
\nc{\psii}{\psi^{\imath}}
\nc{\wti}{\wt^\imath}
\nc{\pii}{\pi^\imath}

\nc{\Udot}{\dot{\U}}
\nc{\UdotA}{\Udot_{\bfA}}

\nc{\plim}[1][]{\mathop{\varprojlim}\limits_{#1}}
\nc{\ilim}[1][]{\mathop{\varinjlim}\limits_{#1}}

\nc{\TBA}{{\large {\bf \textcolor{red}{To Appear}}}}
\nc{\alert}{\textcolor{red}}

\nc{\Sbox}[1]{\ytableausetup{smalltableaux}
\begin{ytableau}
#1
\end{ytableau}}
\nc{\Cbox}[1]{\ytableausetup{centertableaux}
\begin{ytableau}
#1
\end{ytableau}}

\title[Crystal bases of modified $\imath$quantum groups]{Crystal bases of modified $\imath$quantum groups of certain quasi-split types}
\author[H. Watanabe]{Hideya Watanabe}
\address{(H. Watanabe) Osaka Central Advanced Mathematical Institute, Osaka Metropolitan University, Osaka, 558-8585, Japan}
\email{watanabehideya@gmail.com}
\subjclass[2010]{Primary~17B37; Secondary~17B10}
\keywords{}
\date{\today}

\begin{document}
\maketitle

\begin{abstract}
In order to see the behavior of $\imath$canonical bases at $q = \infty$, we introduce the notion of $\imath$crystals associated to an $\imath$quantum group of certain quasi-split type. The theory of $\imath$crystals clarifies why $\imath$canonical basis elements are not always preserved under natural homomorphisms. Also, we construct a projective system of $\imath$crystals whose projective limit can be thought of as the $\imath$canonical basis of the modified $\imath$quantum group at $q = \infty$.
\end{abstract}

\section{Introduction}
The aim of this paper is to introduce the notion of $\imath$crystals associated to a quantum symmetric pair $(\U,\Ui)$ of certain quasi-split type. This study is motivated by the theory of canonical bases for quantum symmetric pairs (also known as the theory of $\imath$canonical bases) initiated by Bao and Wang in \cite{BW18a} and developed in \cite{BW18,BW21}.

Let $(I,I_\bullet,\tau)$ be a Satake diagram of symmetrizable Kac-Moody type (also known as an admissible pair). Namely, $I$ is a Dynkin diagram of symmetrizable Kac-Moody type, $I_\bullet \subset I$ is a subdiagram of finite type, and $\tau \in \Aut(I)$ is a diagram automorphism of order at most two, satisfying certain axioms. From our Dynkin diagram $I$, we can construct the quantum group $\U = \U_I = U_q(I)$. It has a set of generators $E_i,F_i,K_h$, $i \in I$, $h \in Y$, where $Y$ denotes the coroot lattice. In \cite{Ko14}, Kolb defined a right coideal subalgebra $\Ui = \Ui_{\bfvarsigma,\bfkappa} \subset \U$ in terms of $I_\bullet$ and $\tau$, where $\bfvarsigma = (\varsigma_i)_{i \in I} \in (\C(q)^\times)^I$ and $\bfkappa = (\kappa_i)_{i \in I} \in \C(q)^I$ are parameters. To be more a bit precise, $\Ui$ is generated by the quantum group $\U_{I_\bullet}$ associated to the subdiagram $I_\bullet$, $K_h$ for $h \in Y^\imath := \{ h \in Y \mid w_\bullet \tau(h) = -h \}$, and distinguished elements $B_i$ for $i \in I {\setminus} I_\bullet$. The pair $(\U,\Ui)$ is called the quantum symmetric pair associated to our Satake diagram $(I,I_\bullet,\tau)$ and parameters $\bfvarsigma,\bfkappa$. The coideal subalgebra $\Ui$ itself is referred to as the $\imath$quantum group. Kolb's construction generalizes Letzter's construction \cite{Le99} for $I$ being finite type, which unifies earlier examples constructed by Koornwinder \cite{Ko90}, Gavrilik-Klimyk \cite{GK91}, Noumi \cite{N96} and others.

It has turned out that the theory of quantum symmetric pairs has many applications in numerous areas of mathematics and physics such as representation theory of Lie algebras, orthogonal polynomials, low-dimensional topology, categorifications, and integrable systems. Such applications are often based on the fact that the $\imath$quantum group $\Ui$ can be thought of as a generalization of the quantum group. To be more precise, when we take a Satake diagram of diagonal type, the embedding $\Ui \hookrightarrow \U$ can be identified with the comultiplication map $\Delta : \U_{I'} \rightarrow \U_{I'} \otimes \U_{I'}$ for some Dynkin diagram $I'$ (see Example \ref{set up for diagonal type} for details).

From this point of view, the theory of $\imath$canonical bases are thought of as a generalization of the theory of canonical bases for quantum groups initiated by Lusztig \cite{L90}. Namely, the theory of $\imath$canonical bases for diagonal types recovers the usual theory of canonical bases. Let us see this in more detail. Let $X = X_I$ denote the weight lattice, and set $X^\imath := X/\{ \lm+w_\bullet \tau(\lm) \mid \lm \in X \}$, where $w_\bullet$ denotes the longest element of the Weyl group associated to $I_\bullet$ (recall that $I_\bullet$ is of finite type). Then, the modified $\imath$quantum group is defined as
$$
\Uidot := \bigoplus_{\zeta \in X^\imath} \Ui \mathbf{1}_\zeta.
$$
Here, $\mathbf{1}_\zeta$'s are orthogonal idempotents. When our Satake diagram is of diagonal type, we can identify $X$ with $X_{I'} \oplus X_{I'}$ and $X^\imath$ with $X_{I'}$, and hence, the modified $\imath$quantum group is identified with the modified quantum group $\Udot_{I'} = \bigoplus_{\lm \in X_{I'}} \U_{I'} \mathbf{1}_\lm$ (recall that we have identified $\U = \U_{I'} \otimes \U_{I'}$ and $\Ui = \Delta(\U_{I'}) \simeq \U_{I'}$).

For each dominant integral weight $\lm \in X^+$, let $V(\lm)$ (resp., $V^{\rm{low}}(\lm)$) denote the irreducible highest weight module of highest weight $\lm$, and $v_\lm \in V(\lm)$ the highest weight vector (resp., irreducible lowest weight module of lowest weight $-\lm$, and $v^{\rm{low}}_\lm$ the lowest weight vector). Given $\lm,\mu \in X^+$, let $V^\imath(\lm,\mu)$ denote the $\U$-submodule of $V(\lm) \otimes V(\mu)$ generated by $v_{w_\bullet(\lm)} \otimes v_\mu$, where $v_{w_\bullet(\lm)} \in V(\lm)$ denotes the canonical basis element of weight $w_\bullet(\lm)$. Then, Bao and Wang \cite{BW18,BW21} proved that $V^\imath(\lm,\mu)$ has a distinguished basis, which they called the $\imath$canonical basis, of the form
$$
\bfB^\imath(\lm,\mu) = \{ (b_1 \diamond b_2)^\imath_{w_\bullet(\lm),\mu} \mid b_1 \in \bfB_{I_\bullet},\ b_2 \in \bfB \} {\setminus} \{0\}.
$$
Here, $\bfB = \bfB_I$ denotes the canonical basis of the Lusztig algebra $\bff$ associated to $I$. Also, they constructed a projective system of $\Ui$-modules
$$
V^\imath(\lm+\tau(\nu),\mu+\nu) \rightarrow V^\imath(\lm,\mu)
$$
which sends $v_{w_\bullet(\lm+\tau(\nu))} \otimes v_{\mu+\nu}$ to $v_{w_\bullet(\lm)} \otimes v_\mu$. Although it is not always true that this morphism is based (i.e., it sends an $\imath$canonical basis element to either an $\imath$canonical basis element or zero), Bao and Wang were able to prove that it is asymptotically true. Namely, given $b_1 \in \bfB_{I_\bullet}$, $b_2 \in \bfB$, if we take $\lm,\mu \in X^+$ to be sufficiently dominant, then the $\imath$canonical basis element $(b_1 \diamond b_2)^\imath_{w_\bullet(\lm+\tau(\nu)),\mu+\nu}$ is sent to $(b_1 \diamond b_2)^\imath_{w_\bullet(\lm),\mu}$ for all $\nu \in X^+$. Furthermore, they constructed the $\imath$canonical basis $\bfBidot$ of $\Uidot$ as an asymptotical limit of this projective system (see Theorem \ref{asymptotical limit} for the precise meaning).

When our Satake diagram is of diagonal type, the $\U$-module $V^\imath(\lm,\mu)$ is just the tensor product $V^{\rm{low}}(\lm') \otimes V(\mu')$ for some $\lm',\mu' \in X_{I'}^+$, and its $\imath$canonical basis coincides with the canonical basis seen as a tensor product $\U_{I'}$-module. In this case, the morphisms in the projective system are based.

Thus, we are led to investigate why and how much this property fails for quantum symmetric pairs beyond diagonal type. In this paper, we shall take a crystal theoretic approach. Namely, we focus on the projective system at $q = \infty$.

Until the end of this Introduction, assume that our Satake diagram is quasi-split (i.e., $I_\bullet = \emptyset$), and the Cartan matrix $(a_{i,j})_{i,j \in I}$ satisfies $a_{i,\tau(i)} \in \{ 2,0,-1 \}$ for all $i \in I$. Note that every Satake diagram of diagonal type satisfies this assumption. Earlier works \cite{W17, W21b, W21c} of the author suggest that there is a good combinatorial theory which has much information of the representation theory of the $\imath$quantum group associated to our Satake diagram and special parameters $\bfvarsigma,\bfkappa$. In this paper, for each $i \in I$, we define a linear operator $\Btil_i$ acting on certain $\Ui$-modules which modifies the action of $B_i$ (recall that $B_i$ is one of the generators of $\Ui$). Such an operator $\Btil_i$ for $i \in I$ with $a_{i,\tau(i)} = 2$ (resp., $a_{i,\tau(i)} = 0$) has been introduced in \cite{W21b} (resp., \cite{W17}). The $\Btil_i$ for $i \in I$ with $a_{i,\tau(i)} = -1$ is new, and it generalizes the operator introduced in \cite{W17}. As the Kashiwara operators in the representation theory of quantum groups, our operators $\Btil_i$ are defined in a way such that we can take the crystal limit, i.e., the $q \rightarrow \infty$ limit. When our Satake diagram is of diagonal type, the $\Btil_i$ coincides with a Kashiwara operator for $\U_{I'}$.

Then, we investigate how the operators $\Btil_i$ act on the tensor product of a $\Ui$-module on which $\Btil_i$'s are defined, and a $\U$-module on which Kashiwara operators are defined (recall that $\Ui$ is a right coideal of $\U$). Taking the crystal limit, we obtain a combinatorial tensor product rule for $\Btil_i$. As a special case, let us consider the tensor product of the trivial $\Ui$-module and a $\U$-module $M$. Such a $\Ui$-module is canonically identified with $M$ regarded as a $\Ui$-module by restriction. Then, the tensor product rule for $\Btil_i$ gives a $\Ui$-module structure of $M$ at $q = \infty$. When our Satake diagram is of diagonal type, the $\U$-module $M$ is of the form $L \otimes N$ for some $\U_{I'}$-modules $L,N$. Then, $\Btil_i$ on $M = L \otimes N$ at $q = \infty$ coincides with the tensor product rule for a Kashiwara operator for $\U_{I'}$.

Following the theory of crystals, we introduce the notion of $\imath$crystals which abstract the operators $\Btil_i$ at $q = \infty$. An $\imath$crystal is a set $\clB$ equipped with several structure maps including linear operators $\Btil_i \in \End_{\C}(\C \clB)$, $i \in I$. We prove that given an $\imath$crystal $\clB_1$ and a crystal $\clB_2$ satisfying certain conditions, the tensor product $\clB_1 \otimes \clB_2 := \clB_1 \times \clB_2$ of them has an $\imath$crystal structure. Taking $\clB_1$ to be ``the trivial $\imath$crystal'', we obtain a way making a crystal into an $\imath$crystal. As a result, it turns out that the $\Ui$-module structure of a $\U$-module $M$ at $q = \infty$ is described by the crystal basis of $M$ regarded as an $\imath$crystal via this method.

Let us return to the projective system $\{ V^\imath(\lm,\mu) \}_{\lm,\mu \in X^+}$. Under the assumption on our Satake diagram, the $\U$-module $V^\imath(\lm,\mu)$ is isomorphic to the irreducible highest weight module $V(\lm+\mu)$. Hence, the projective system becomes $\{ V(\lm) \}_{\lm \in X^+}$ with morphisms $V(\lm+\nu+\tau(\nu)) \rightarrow V(\nu)$ sending $v_{\lm+\nu+\tau(\nu)}$ to $v_\nu$. Now, we are able to answer our question why and how much the morphisms $V(\lm+\nu+\tau(\nu)) \rightarrow V(\lm)$ are not based. That is, when $\lm$ is not sufficiently dominant, there is no morphism $\clB(\lm+\nu+\tau(\nu)) \rightarrow \clB(\lm)$ of $\imath$crystals sending the highest weight element $b_{\lm+\nu+\tau(\nu)}$ to the highest weight element $b_\lm$, where $\clB(\lm)$ denotes the crystal basis of $V(\lm)$. Via a further observation, it turns out that there exists a dominant weight $\sigma \in X^+$ such that for each $\lm,\nu \in X^+$, there exists a morphism $\clB(\sigma+\lm+\nu+\tau(\nu);\lm+\nu+\tau(\nu)) \rightarrow \clB(\sigma+\lm;\lm)$ of $\imath$crystals sending $b_{\sigma+\lm+\nu+\tau(\nu)}$ to $b_{\sigma+\lm}$, where
$$
\clB(\lm;\mu) := \{ \Ftil_{i_1} \cdots \Ftil_{i_r} b_\lm \mid \Ftil_{i_1} \cdots \Ftil_{i_r} b_\mu \in \clB(\mu) \} {\setminus} \{0\} \subset \clB(\lm).
$$
Also, for each $\lm \in X^+$, there exists an $\imath$crystal $\clB(\lm)^\sigma$ isomorphic to $\clB(\sigma+\lm;\lm)$ whose underlying set is $\clB(\lm)$. Thus, we obtain a projective system $\{ \clB(\lm)^\sigma \}_{\lm \in X^+, \ol{\sigma+\lm} = \zeta}$ of $\imath$crystals for each $\zeta \in X^\imath$, where $\ol{\lm} \in X^\imath$ denote the image of $\lm \in X$. This projective system has the projective limit of the form
$$
\clT_{\zeta} \otimes \clB(\infty),
$$
where $\clT_\zeta$ is a certain $\imath$crystal consisting of a single element, and $\clB(\infty)$ denotes the crystal basis of the negative part $\U^-$ of $\U$.

Then, we lift these results to the representation theory of $\imath$quantum groups. Namely, for each $\lm \in X^+$, we construct a $\Ui$-module $V(\lm)^\sigma$ whose underlying set is $V(\lm)$. The $V(\lm)^\sigma$ has a distinguished basis, called the $\imath$canonical basis, whose crystal limit is the $\imath$crystal $\clB(\lm)^\sigma$. Furthermore, there exists a based $\Ui$-module homomorphism $V(\lm+\nu+\tau(\nu))^\sigma \rightarrow V(\lm)^\sigma$ which sends $v_{\lm+\nu+\tau(\nu)}$ to $v_\lm$. Thus, we obtain a projective system $\{ V(\lm)^\sigma \}_{\lm \in X^+, \ol{\sigma+\lm} = \zeta}$ of $\Ui$-modules and based homomorphisms. In particular, when we can take $\sigma$ to be $0$, this result partially proves Bao and Wang's Conjecture in \cite[Remark 6.18]{BW18}. It turns out that this projective system has the projective limit, and it is $\Uidot \mathbf{1}_\zeta$ with the $\imath$canonical basis $\bfBidot \mathbf{1}_\zeta$.

From the results above, we should call the $\imath$crystal $\bigsqcup_{\zeta \in X^\imath} \clT_\zeta \otimes \clB(\infty)$ the $\imath$crystal basis of $\Uidot$, and denote it by $\clBidot$. The description $\clBidot = \bigsqcup_{\zeta \in X^\imath} \clT_{\zeta} \otimes \clB(\infty)$ can be interpreted as the crystal limit of the description $\Uidot = \bigoplus_{\zeta \in X^\imath} \Ui \mathbf{1}_\zeta$ since $\clT_\zeta$ consists of a single element, and $\clB(\infty)$ can be seen as the crystal limit of $\U^-$, which is isomorphic to $\Ui$ as a vector space. When our Satake diagram is of diagonal type, we can take $\sigma = 0$, and can identify $\clB(\infty) = \clB(-\infty)_{I'} \otimes \clB(\infty)_{I'}$, where $\clB(-\infty)_I$ denotes the crystal basis of the positive part of $\U$. Hence, we recover the projective system $\{ V^{\rm{low}}(\lm) \otimes V(\mu) \}_{\lm,\mu \in X_{I'}^+, -\lm+\mu = \zeta}$ of $\U_{I'}$-modules and based homomorphisms, whose projective limit is $\Udot_{I'} \mathbf{1}_\zeta$, and the projective system $\{ \clB^{\rm{low}}(\lm) \otimes \clB(\mu) \}_{\lm,\mu \in X_{I'}^+, -\lm+\mu=\zeta}$ of crystals, whose projective limit is $\clB(-\infty)_{I'} \otimes \clT_\zeta \otimes \clB(\infty)_{I'}$, where $\clT_\zeta$ denotes a certain crystal consisting of a single element, and $\clB^{\rm{low}}(\lm)$ the crystal basis of $V^{\rm{low}}(\lm)$.

As explained above, our construction of the new projective system $\{ V(\lm)^\sigma \}_{\lm \in X^+}$ is motivated by an observation of the $\imath$crystal structures of various $\U$-modules. However, the construction itself may be possible without the theory of $\imath$crystals and our assumption on Satake diagrams and parameters $\bfvarsigma,\bfkappa$. We will treat this in a future work.

This paper is organized as follows. In Section \ref{Section: quantum groups and crystals}, we recall necessary knowledge concerning ordinary quantum groups and crystals. In Section \ref{Section: iquantum groups and icrystals}, we set up an $\imath$quantum group of quasi-split type, and recall Bao-Wang's construction of the projective system and the $\imath$canonical basis of the modified $\imath$quantum group. Also, we define the notion of $\imath$crystals and their morphisms. Basic examples of $\imath$crystals are given there. Section \ref{Section: modified action of Bi} is devoted to defining the operators $\Btil_i$ acting on certain $\Ui$-modules. The tensor product rule is also investigated. Based on the results obtained there, we define the tensor product of an $\imath$crystal and a crystal in Section \ref{Section: tensor product rule for icrystal}. The associativity of this tensor product is stated there, too. The proofs of the well-definedness and the associativity of the tensor product is given in Section \ref{Section: proofs} because they are lengthy and independent of later argument. In Section \ref{Section: stability}, we construct the projective system of $\imath$crystals and its projective limit mentioned earlier. Then, we finally construct our new projective system of $\Ui$-modules and based homomorphisms.

\subsection*{Acknowledgement}
This work was supported by JSPS KAKENHI Grant Numbers JP20K14286 and JP21J00013.

\subsection*{Notation}
Throughout this paper, we use the following notation:
\begin{itemize}
\item $\bbK = \C(q)$: the field of rational functions in one variable $q$.
\item $\bbK_\infty$: the subring of $\bbK$ consisting of functions regular at $q = \infty$.
\item $\bfA := \C[q,q\inv]$.
\item For $a,m,n \in \Z$, $[n]_{q^a} := \frac{q^{an}-q^{-an}}{q^a-q^{-a}}$, $[n]_{q^a}! := \prod_{k=1}^n [k]_{q^a}$, ${m+n \brack n}_{q^a} := \frac{[m+n]_{q^a}!}{[m]_{q^a}![n]_{q^a}!}$.
\item $\ol{n} \in \Z/2\Z$: the image of $n \in \Z$ under the quotient map $\Z \rightarrow \Z/2\Z$.
\item For $n \in \Z$, $\sgn(n) := \begin{cases}
1 & \IF n > 0, \\
0 & \IF n = 0, \\
-1 & \IF n < 0.
\end{cases}$
\end{itemize}

\section{Quantum groups and crystals}\label{Section: quantum groups and crystals}
In this section, we review basic results concerning the representation theory of quantum groups associated to a symmetrizable Kac-Moody algebra, abstract crystals, and canonical and crystal bases.

\subsection{Quantum groups}
Let $A = (a_{i,j})_{i,j \in I}$ be a symmetrizable generalized Cartan matrix with symmetrizing matrix $D = \diag(d_i)_{i \in I}$, i.e., $d_i$'s are pairwise coprime positive integers satisfying $d_i a_{i,j} = d_j a_{j,i}$ for all $i,j \in I$. Let $X,Y$ be free abelian groups of finite rank equipped with a perfect pairing $\la \cdot, \cdot \ra : Y \times X \rightarrow \Z$. Let $\{ \alpha_i \mid i \in I \} \subset X$ and $\{ h_i \mid i \in I \} \subset Y$ be linearly independent sets such that
$$
\la h_i,\alpha_j \ra = a_{i,j} \qu \Forall i,j \in I.
$$
When $I$ is of finite type, each $\lm \in X$ is uniquely determined by the values $\la h_i,\lm \ra \in \Z$, $i \in I$. We often identify $\lm \in X$ with $(\la h_i,\lm \ra)_{i \in I} \in \Z^I$.

The quantum group $\U$ is defined to be a unital associative algebra over $\bbK$ generated by $E_i,F_i,K_h$, $i \in I$, $h \in Y$ subject to the following relations: Let $i,j \in I$, $h,h' \in Y$.
\begin{align}
\begin{split}
&K_{0} = 1, \qu K_{h} K_{h'} = K_{h+h'}, \\
&K_h E_i = q^{\la h,\alpha_i \ra}E_i K_h, \qu K_h F_i = q^{-\la h,\alpha_i \ra}F_i K_h, \\
&E_iF_j-F_jE_i = \delta_{i,j} \frac{K_i-K_i\inv}{q_i-q_i\inv}, \\
&\sum_{r=0}^{1-a_{i,j}} (-1)^r E_i^{(r)} E_j E_i^{(1-a_{i,j}-r)} = 0 \qu \IF i \neq j, \\
&\sum_{r=0}^{1-a_{i,j}} (-1)^r F_i^{(r)} F_j F_i^{(1-a_{i,j}-r)} = 0 \qu \IF i \neq j,
\end{split} \nonumber
\end{align}
where
$$
q_i := q^{d_i}, \ K_i := K_{d_i h_i}, \ E_i^{(a)} := \frac{1}{[a]_i!} E_i^a, \ F_i^{(a)} := \frac{1}{[a]_i!} F_i^a, \ [a]_i := [a]_{q_i}, \ [a]_i! := [a]_{q_i}!.
$$

The quantum group $\U$ is equipped with a Hopf algebra structure with comultiplication $\Delta$ given by
$$
\Delta(E_i) := E_i \otimes 1 + K_i \otimes E_i, \ \Delta(F_i) := 1 \otimes F_i + F_i \otimes K_i\inv,\ \Delta(K_h) := K_h \otimes K_h.
$$

There is an anti-algebra involution $\wp$ on $\U$ such that
$$
\wp(E_i) = q_i\inv F_iK_i,\ \wp(F_i) = q_i\inv E_i K_i\inv,\ \wp(K_h) = K_h.
$$
The complex conjugate $\C \ni z \mapsto z^*$ can be extended to an $\R(q)$-algebra automorphism on $\U$ by requiring
$$
E_i^* = E_i, \ F_i^* = F_i, \ K_h^* = K_h.
$$
Set $\wp^* := \wp \circ * = * \circ \wp$.

Let $\ol{\cdot} : \bbK \rightarrow \bbK$ denote the $\C$-algebra automorphism given by $\ol{q} = q\inv$. This can be extended to a $\C$-algebra automorphism $\psi$, called the bar-involution on $\U$ by
$$
\psi(E_i) = E_i,\ \psi(F_i) = F_i,\ \psi(K_h) = K_{-h}.
$$

For each $J = \{ j_1,\ldots,j_k \} \subset I$ of finite type, let $\U_J = \U_{j_1,\ldots,j_k}$ denote the subalgebra of $\U$ generated by $E_j,F_j,K_j^{\pm 1}$, $j \in J$.

Let $\U^-$ denote the subalgebra of $\U$ generated by $F_i$, $i \in I$. Also, let $\bfB(\infty)$ and $\clB(\infty)$ denote the canonical and crystal basis of $\U^-$ with $b_\infty \in \clB(\infty)$ the highest weight element.

Let $\Udot = \bigoplus_{\lm \in X} \U \mathbf{1}_\lm$ denote the modified quantum group, and $\Udot_{\bfA}$ its $\bfA$-form. Also, let $\bfBdot$ and $\clBdot$ denote the canonical and crystal basis of $\Udot$, respectively.

For each $\lm \in X$, let $M(\lm)$ denote the Verma module of highest weight $\lm$, and $V(\lm)$ its irreducible quotient. Let $v_\lm$ denote the highest weight vector of both $M(\lm)$ and $V(\lm)$. When $\lm \in X^+$, let $\bfB(\lm)$ and $(\clL(\lm), \clB(\lm))$ denote the canonical and crystal base of $V(\lm)$. Let $b_\lm \in \clB(\lm)$ denote the highest weight element.

\subsection{Crystal}\label{subsection: cyrstal}
A crystal is a set $\clB$ equipped with the following structure
\begin{itemize}
\item $\wt : \clB \rightarrow X$: map,
\item $\vep_i,\vphi_i : \clB \rightarrow \Z \sqcup \{ -\infty \}$: maps, $i \in I$, where $-\infty$ is a formal symbol,
\item $\Etil_i,\Ftil_i : \clB \rightarrow \clB \sqcup \{0\}$: maps, $i \in I$, where $0$ is a formal symbol,
\end{itemize}
satisfying the following axioms: Let $b \in \clB$ and $i \in I$.
\begin{enumerate}
\item\label{crystal 2} If $\vphi(b) = -\infty$, then $\Etil_i b = 0 = \Ftil_i b$.
\item\label{crystal 1} $\vphi_i(b) = \vep_i(b)+\wt_i(b)$, where $\wt_i(b) := \la h_i, \wt(b) \ra$; we understand that $-\infty+a = -\infty$ for all $a \in \Z$.
\item\label{crystal 3} If $\Etil_i b \neq 0$, then $\wt(\Etil_i b) = \wt(b) + \alpha_i$, $\vep_i(\Etil_i b) = \vep_i(b)-1$, and $\Ftil_i \Etil_i b = b$.
\item\label{crystal 4} If $\Ftil_i b \neq 0$, then $\wt(\Ftil_i b) = \wt(b) - \alpha_i$, $\vphi_i(\Ftil_i b) = \vphi_i(b)-1$, and $\Etil_i \Ftil_i b = b$.
\end{enumerate}

Let $\clB$ be a crystal. The crystal graph of $\clB$ is an $I$-colored directed graph whose vertex set is $\clB$, and for each $b,b' \in \clB$ and $i \in I$, there exists an arrow from $b$ to $b'$ labeled by $i$ if and only if $\Ftil_i b = b'$.

\begin{ex}\label{crystal B(n)}\normalfont
Suppose that $I = \{ i \}$. For each $n \in \Z_{\geq 0}$, let $\clB(n) := \{ b_k \mid 0 \leq k \leq n \}$ denote the crystal given by
$$
\wt_i(b_k) = n-2k, \qu \vep_i(b_k) = k, \qu \vphi_i(b_k) = n-k, \qu \Etil_i b_k = b_{k-1}, \qu \Ftil_i b_k = b_{k+1},
$$
where $b_{-1} = b_{n+1} := 0$. The crystal graph of $\clB(n)$ is as follows:
$$
\xymatrix{
b_0 \ar[r]^-i & b_1 \ar[r]^-i & \cdots \ar[r]^-i & b_{n-1} \ar[r]^-i & b_n.
}
$$
\end{ex}

Let $\clB_1,\clB_2$ be crystals. A morphism $\mu : \clB_1 \rightarrow \clB_2$ of crystals is a map $\mu : \clB_1 \rightarrow \clB_2 \sqcup \{0\}$ satisfying the following: Let $b \in \clB_1$ and $i \in I$.
\begin{enumerate}
\item If $\mu(b) \neq 0$, then $\wt(\mu(b)) = \wt(b)$, $\vep_i(\mu(b)) = \vep_i(b)$, and $\vphi_i(\mu(b)) = \vphi_i(b)$.
\item If $\Etil_i b, \mu(b), \mu(\Etil_i b) \neq 0$, then $\mu(\Etil_i b) = \Etil_i \mu(b)$.
\item If $\Ftil_i b, \mu(b), \mu(\Ftil_i b) \neq 0$, then $\mu(\Ftil_i b) = \Ftil_i \mu(b)$.
\end{enumerate}
A crystal morphism $\mu : \clB_1 \rightarrow \clB_2$ is said to be strict if $\mu(\Etil_i b) = \Etil_i \mu(b)$ and $\mu(\Ftil_i b) = \Ftil_i \mu(b)$ for all $i \in I$, $b \in \clB_1$; here, we set $\mu(0) = 0$. A strict crystal morphism $\mu : \clB_1 \rightarrow \clB_2$ is said to be an isomorphism if the underlying map $\mu : \clB_1 \rightarrow \clB_2$ is bijective; in this case, we denote $\clB_1 \simeq \clB_2$.

Let $\clB_1,\clB_2$ be crystals. The tensor product $\clB_1 \otimes \clB_2$ of $\clB_1$ and $\clB_2$ is a crystal whose underlying set is $\clB_1 \times \clB_2$ and whose structure maps are given as follows:
\begin{align}
\begin{split}
&\wt(b_1 \otimes b_2) = \wt(b_1) + \wt(b_2), \\
&\vep_i(b_1 \otimes b_2) = \max(\vep_i(b_1) - \wt_i(b_2), \vep_i(b_2)) = \begin{cases}
\vep_i(b_1) - \wt_i(b_2) & \IF \vep_i(b_1) > \vphi_i(b_2), \\
\vep_i(b_2) & \IF \vep_i(b_1) \leq \vphi_i(b_2),
\end{cases} \\
&\vphi_i(b_1 \otimes b_2) = \max(\vphi_i(b_1), \vphi_i(b_2) + \wt_i(b_1)) = \begin{cases}
\vphi_i(b_2)+\wt_i(b_1) & \IF \vep_i(b_1) < \vphi_i(b_2), \\
\vphi_i(b_1) & \IF \vep_i(b_1) \geq \vphi_i(b_2),
\end{cases} \\
&\Etil_i(b_1 \otimes b_2) = \begin{cases}
\Etil_i b_1 \otimes b_2 \qu & \IF \vep_i(b_1) > \vphi_i(b_2), \\
b_1 \otimes \Etil_i b_2 \qu & \IF \vep_i(b_1) \leq \vphi_i(b_2),
\end{cases} \\
&\Ftil_i(b_1 \otimes b_2) = \begin{cases}
b_1 \otimes \Ftil_i b_2 \qu & \IF \vep_i(b_1) < \vphi_i(b_2), \\
\Ftil_i b_1 \otimes b_2 \qu & \IF \vep_i(b_1) \geq \vphi_i(b_2).
\end{cases}
\end{split} \nonumber
\end{align}
Here, we understand that $-\infty < a$ for all $a \in \Z$, and $-\infty \leq -\infty$.

\begin{rem}\normalfont
The tensor product for crystals is associative \cite[Proposition 2.3.2]{BS17}.
\end{rem}

\begin{ex}\label{tensor rule example}\normalfont
Suppose that $I = \{ i \}$. Recall from Example \ref{crystal B(n)} the crystal $\clB(n)$. The crystal graph $\clB(2) \otimes \clB(3)$ is described as follows:
$$
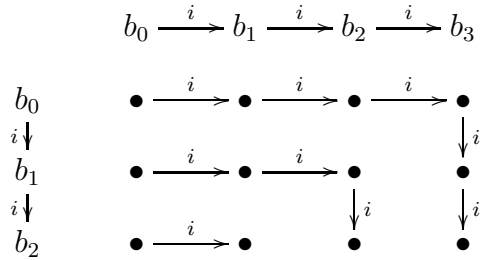
\xymatrix@R=10pt{
                  & b_0 \ar[r]^-i & b_1 \ar[r]^-i & b_2 \ar[r]^-i & b_3 \\
b_0 \ar[d]_i & \bullet \ar[r]^-i & \bullet \ar[r]^-i & \bullet \ar[r]^-i & \bullet \ar[d]^-i \\
b_1 \ar[d]_i & \bullet \ar[r]^-i & \bullet \ar[r]^-i & \bullet \ar[d]^-i & \bullet \ar[d]^-i \\
b_2 & \bullet \ar[r]^-i & \bullet & \bullet & \bullet
}
$$
\end{ex}

Let us introduce some terminologies. Let $\clB$ be a crystal.
\begin{enumerate}
\item $\clB$ is said to be seminormal if
$$
\vep_i(b) = \max\{ m \geq 0 \mid \Etil_i^m b \neq 0 \}, \qu \vphi_i(b) = \max\{ m \geq 0 \mid \Ftil_i^m b \neq 0 \}
$$
for all $b \in \clB$ and $i \in I$.
\item $\clB$ is said to be upper seminormal if
$$
\vep_i(b) = \max\{ m \geq 0 \mid \Etil_i^m b \neq 0 \}
$$
for all $b \in \clB$ and $i \in I$.
\end{enumerate}

\begin{ex}\normalfont
\ \begin{enumerate}
\item For each $\lm \in X^+$, the crystal basis $\clB(\lm)$ of $V(\lm)$ is a seminormal crystal.
\item The crystal basis $\clB(\infty)$ of $\U^-$ is an upper seminormal crystal. Also, we have $\Ftil_i b \neq 0$ for all $i \in I$ and $b \in \clB(\infty)$.
\item For each $\lm \in X$, let $\clT_\lm = \{ t_\lm \}$ denote the crystal given by
$$
\wt(t_\lm) = \lm, \qu \vep_i(t_\lm) = \vphi_i(t_\lm) = -\infty, \qu \Etil_i t_\lm = \Ftil_i t_\lm = 0, \qu \Forall i \in I.
$$
\item For each $\lm,\mu \in X$, we have $\clT_\lm \otimes \clT_\mu \simeq \clT_{\lm+\mu}$.
\item For each $\lm \in X^+$, there exists an injective crystal morphism $\iota_\lm : \clB(\lm) \rightarrow \clT_\lm \otimes \clB(\infty)$ such that $\iota_\lm(b_\lm) = t_\lm \otimes b_\infty$. This morphism is not strict.
\end{enumerate}
\end{ex}

For each $\lm \in X^+$, set
$$
\clB(\infty;\lm) := \{ b \in \clB(\infty) \mid t_\lm \otimes b \in \im \iota_\lm \}.
$$
Let $\pi_\lm : \clB(\infty) \rightarrow \clB(\lm) \sqcup \{0\}$ be a map defined by
$$
\pi_\lm(b) := \begin{cases}
b' \qu & \IF b' \in \clB(\lm) \AND \iota_\lm(b') = t_\lm \otimes b, \\
0 \qu & \IF b \notin \clB(\infty;\lm).
\end{cases}
$$
Note that for each $b \in \clB(\infty;\lm)$ and $i \in I$, we have
\begin{align}\label{crystal structure of B(infty;lm)}
\wt(b) = \wt(\pi_\lm(b)) - \lm, \qu \vep_i(b) = \vep_i(\pi_\lm(b)), \qu \vphi_i(b) = \vphi_i(\pi_\lm(b)) - \la h_i,\lm \ra.
\end{align}

\begin{lem}\label{Kashiwara operators on B(infty;lm)}
Let $\lm \in X^+$, $b \in \clB(\infty;\lm)$, and $i \in I$.
\begin{enumerate}
\item\label{Kashiwara operators on B(infty;lm) 1} We have $\Etil_i b \in \clB(\infty;\lm)$ if and only if $\vep_i(b) > 0$.
\item\label{Kashiwara operators on B(infty;lm) 2} We have $\Ftil_i b \in \clB(\infty;\lm)$ if and only if $\vphi_i(b) > -\la h_i,\lm \ra$.
\end{enumerate}
\end{lem}

\begin{proof}
Set $b' := \pi_\lm(b) \in \clB(\lm)$. Let us prove the first assertion. Since $\clB(\infty)$ is upper seminormal, the ``only if'' part is obvious. Hence, let us assume that $\vep_i(b) > 0$. Then, by identity \eqref{crystal structure of B(infty;lm)},
$$
\vep_i(b') = \vep_i(b) > 0.
$$
Since $\clB(\lm)$ is seminormal, we obtain $\Etil_i b' \neq 0$. Therefore, we have $\iota_\lm(\Etil_i b') \neq 0$, and hence,
$$
\iota_\lm(\Etil_i b') = \Etil_i \iota_\lm(b') = \Etil_i(t_\lm \otimes b) = t_\lm \otimes \Etil_i b.
$$
This implies that $\Etil_i b \in \clB(\infty;\lm)$, as desired.

Next, we prove the second assertion. Noting that $\Ftil_i(t_\lm \otimes b) = t_\lm \otimes \Ftil_i b \neq 0$, we see that we have $\Ftil_i b \in \clB(\infty;\lm)$ if and only if $\Ftil_i b' \neq 0$. Since $\clB(\lm)$ is seminormal, the latter condition is equivalent to that $\vphi_i(b') > 0$. By identity \eqref{crystal structure of B(infty;lm)}, this condition is, in turn, equivalent to that
$$
\vphi_i(b) > -\la h_i,\lm \ra.
$$
Thus, the assertion follows.
\end{proof}

\begin{lem}\label{embedding of B(mu) in B(lm)}
Let $\lm,\mu \in X^+$ be such that $\la h_i,\mu \ra \leq \la h_i,\lm \ra$ for all $i \in I$. Then, we have $\clB(\infty;\mu) \subset \clB(\infty;\lm)$.
\end{lem}

\begin{proof}
Let $b \in \clB(\infty;\mu)$. Then, we can write as
$$
b = \Ftil_{i_r} \cdots \Ftil_{i_1} b_\infty
$$
for some $i_1,\ldots,i_r \in I$. We show that $b \in \clB(\infty;\lm)$ by induction on $r \geq 0$. When $r = 0$, we have $b = b_\infty \in \clB(\infty;\lm)$. Hence, assume that $r > 0$. By induction hypothesis, we have $b' := \Etil_{i_r} b = \Ftil_{i_{r-1}} \cdots \Ftil_{i_1} b_\infty \in \clB(\infty;\lm)$. By Lemma \ref{Kashiwara operators on B(infty;lm)} \eqref{Kashiwara operators on B(infty;lm) 2}, we obtain
$$
\vphi_i(b') > -\la h_i,\mu \ra \geq -\la h_i,\lm \ra.
$$
This implies, again by Lemma \ref{Kashiwara operators on B(infty;lm)} \eqref{Kashiwara operators on B(infty;lm) 2}, that $\Ftil_{i_r} b' \in \clB(\infty,\lm)$. This completes the proof.
\end{proof}

For each $\lm,\mu \in X^+$ such that $\la h_i,\mu \ra \leq \la h_i,\lm \ra$ for all $i \in I$, set
$$
\clB(\lm;\mu) := \pi_\lm(\clB(\infty;\mu)).
$$
Note that by Lemma \ref{embedding of B(mu) in B(lm)}, $\pi_\lm(\clB(\infty;\mu))$ does not contain $0$. By identity \eqref{crystal structure of B(infty;lm)}, for each $b \in \clB(\lm;\mu)$, $i \in I$, and $b' \in \clB(\infty;\mu)$ with $\pi_\lm(b') = b$, we obtain
\begin{align}\label{crystal structure of B(lm;mu)}
\wt(b) = \wt(\pi_\mu(b')) + (\lm-\mu), \qu \vep_i(b) = \vep_i(\pi_\mu(b')), \qu \vphi_i(b) = \vphi_i(\pi_\lm(b')) + \la h_i,\lm-\mu \ra.
\end{align}

\begin{lem}\label{Kashiwara operators on B(lm;mu)}
Let $\lm,\mu \in X^+$ be such that $\la h_i,\mu \ra \leq \la h_i,\lm \ra$ for all $i \in I$, $b \in \clB(\lm;\mu)$, and $i \in I$.
\begin{enumerate}
\item We have $\Etil_i b \in \clB(\lm;\mu)$ if and only if $\vep_i(b) > 0$.
\item We have $\Ftil_i b \in \clB(\lm;\mu)$ if and only if $\vphi_i(b) > \la h_i,\lm-\mu \ra$.
\end{enumerate}
\end{lem}

\begin{proof}
The assertions follow from Lemma \ref{Kashiwara operators on B(infty;lm)} and identity \eqref{crystal structure of B(lm;mu)}.
\end{proof}

Let $\clB$ be a crystal, and $i,j \in I$ with $a_{i,j} = a_{j,i} \in \{ 0,-1 \}$. Consider the following conditions ({\it cf. } \cite[Chapter 4]{BS17}):
\begin{enumerate}
\item[{\rm (S1)}] If $b,\Etil_i b \in \clB$, then $\vep_j(\Etil_i b)-\vep_j(b) \in \{ 0,-a_{i,j} \}$.
\item[{\rm (S2)}] If $b,\Etil_j \Etil_i b \in \clB$ and $\vep_j(\Etil_i b) = \vep_j(b)$, then $\Etil_i \Etil_jb = \Etil_j \Etil_i b$ and $\vphi_i(\Etil_j b) = \vphi_i(b)$.
\item[{\rm (S3)}] If $b,\Etil_i\Etil_j b, \Etil_j \Etil_i b \in \clB$, $\vep_j(\Etil_i b) = \vep_j(b)+1$, and $\vep_i(\Etil_jb) = \vep_i(b)+1$, then $\Etil_i \Etil_jb \neq \Etil_j \Etil_i b$.
\item[{\rm (S2)'}] If $b,\Ftil_j \Ftil_i b \in \clB$ and $\vphi_j(\Ftil_i b) = \vphi_j(b)$, then $\Ftil_i \Ftil_jb = \Ftil_j \Ftil_i b$ and $\vep_i(\Ftil_j b) = \vep_i(b)$.
\item[{\rm (S3)'}] If $b,\Ftil_i \Ftil_j b, \Ftil_j \Ftil_i b \in \clB$, $\vphi_j(\Ftil_i b) = \vphi_j(b)+1$, and $\vphi_i(\Ftil_jb) = \vphi_i(b)+1$, then $\Ftil_i \Ftil_jb \neq \Ftil_j \Ftil_i b$.
\end{enumerate}

\begin{ex}
The crystals $\clB(\lm)$, $\clT_\lm$, and $\clB(\infty)$ satisfy the conditions above.
\end{ex}

\begin{lem}\label{Deduction from condition S's}
Let $\clB$ be a crystal, and $i,j \in I$ with $a_{i,j} = a_{j,i} \in \{ 0,-1 \}$. Assume that $\clB$ satisfies conditions {\rm (S1)}--{\rm (S3)'}. Then, for each $b \in \clB$, the following hold.
\begin{enumerate}
\item\label{Deduction from condition S's 1} If $\Ftil_i b \neq 0$, then $\vphi_j(\Ftil_i b) - \vphi_j(b) \in \{ 0,-a_{i,j} \}$.
\item\label{Deduction from condition S's 2} If $\Ftil_j \Ftil_i b \neq 0$ and $\vphi_j(\Ftil_i b) = \vphi_j(b)+1$, then $\vphi_i(\Ftil_j\Ftil_i b) = \vphi_i(b)-1$.
\item\label{Deduction from condition S's 3} If $\Etil_i \Etil_j b \neq 0$ and $\vphi_i(\Etil_jb) = \vphi_i(b)$, then $\vphi_j(\Etil_i \Etil_jb) = \vphi_j(b)$.
\item\label{Deduction from condition S's 4} If $\Ftil_i b, \Etil_j b \neq 0$, then we have $\vphi_i(\Etil_j b) = \vphi_i(b)-1$ if and only if $\vphi_j(\Ftil_i b) = \vphi_j(b)$.
\end{enumerate}
\end{lem}

\begin{proof}
Let us prove the first assertion. We compute as
\begin{align}
\begin{split}
\vphi_j(\Ftil_i b) - \vphi_j(b) &= (\vep_j(\Ftil_i b) + \la h_j,\wt(b)-\alpha_i \ra)-(\vep_j(b)+\la h_j,\wt(b) \ra) \\
&=-(\vep_j(\Etil_i \Ftil_i b) - \vep_j(\Ftil_i b)) - a_{i,j}.
\end{split} \nonumber
\end{align}
By condition {\rm (S1)}, applied to $\Ftil_i b$, the last line of the identity above is either $0$ or $-a_{i,j}$. This proves the assertion.

Let us prove the second assertion. Assume contrary that $\vphi_i(\Ftil_j \Ftil_i b) \neq \vphi_i(b)-1$. Since $\vphi_i(b)-1 = \vphi_i(\Ftil_i b)$, the first assertion implies that $a_{i,j} = -1$ and
\begin{align}\label{cond 1}
\vphi_i(\Ftil_j \Ftil_i b) = \vphi_i(\Ftil_i b)+1 = \vphi_i(b).
\end{align}
Then, we have
$$
\vep_i(\Ftil_j \Ftil_i b) = \vphi_i(\Ftil_j \Ftil_i b)-\wt_i(\Ftil_j \Ftil_i b) = \vphi_i(b)-(\wt_i(b)-1) = \vep_i(\Ftil_i b) = \vep_i(\Etil_j \Ftil_j \Ftil_i b).
$$
By condition {\rm (S2)}, this implies that
$$
\Etil_j \Etil_i \Ftil_j \Ftil_i b = \Etil_i \Etil_j \Ftil_j \Ftil_i b = b,
$$
which, in turn, implies that
\begin{align}\label{cond 2}
\Ftil_j \Ftil_i b = \Ftil_i \Ftil_j b.
\end{align}
Then, using identity \eqref{cond 1}, we compute as
$$
\vphi_i(\Ftil_j b) = \vphi_i(\Ftil_i \Ftil_j b)+1 = \vphi_i(\Ftil_j \Ftil_i b)+1 = \vphi_i(b)+1.
$$
By condition {\rm (S3)'}, this, together with our assumption that $\vphi_j(\Ftil_i b) = \vphi_j(b)+1$, implies that
$$
\Ftil_i \Ftil_j b \neq \Ftil_j \Ftil_i b,
$$
which contradicts identity \eqref{cond 2}. Thus, the assertion follows.

The third assertion can be proved in a similar way to the second one. The fourth assertion follows from the second and third ones.
\end{proof}

\subsection{Based modules}
A $\U$-module is said to have a bar-involution $\psi_M : M \rightarrow M$ if
$$
\psi_M(xv) = \psi(x) \psi_M(v) \qu \Forall x \in \U,\ v \in M.
$$
For example, each $V(\lm)$, $\lm \in X^+$ has a unique bar-involution $\psi_\lm$ such that $\psi_\lm(v_\lm) = v_\lm$.

A weight module is a $\U$-module $M$ which possesses a weight space decomposition
$$
M = \bigoplus_{\lm \in X} M_\lm, \qu M_\lm := \{ v \in M \mid K_h v = q^{\la h,\lm \ra} v \Forall h \in Y \}.
$$
Each weight module $M$ admits a natural $\Udot$-module structure. An $\bfA$-form of $M$ is an $\bfA$-lattice $M_{\bfA}$ such that $\Udot_{\bfA} M_{\bfA} \subset M_{\bfA}$. For example, $V(\lm)$ is a weight module, and $V(\lm)_{\bfA} := \Udot_{\bfA} v_\lm$ is an $\bfA$-form.


An integrable module is a weight module $M$ on which $E_i,F_i$ acts locally nilpotently for all $i \in I$.
For example, $V(\lm)$ is integrable. Let $M$ be an integrable $\U$-module. Then, for each $J \subset I$ of finite type, as a $\U_J$-module, $M$ decomposes into the direct sum of finite-dimensional irreducible $\U_J$-modules.

Let $M$ be a $\U$-module equipped with a $\bbK$-valued Hermitian inner product $(\cdot,\cdot)_M$, i.e., a $\bbK$-valued Hermitian form satisfying the following (cf. \cite[Definition 2.1.2]{W21b}):
\begin{itemize}
  \item For each $v \in M \setminus \{0\}$, there exist a positive real number $c$ and an integer $d$ such that $(v,v) \in c q^{2d} + q^{2d-1}\mathbb{C}[\![q^{-1}]\!]$.
  \item For each $v \in M$, we have $(v,v) = 0$ if and only if $v = 0$.
\end{itemize}
It is said to be contragredient if
$$
(x u, v)_M = (u, \wp^*(x) v)_M \qu \Forall x \in \U,\ u,v \in M.
$$
For example, $V(\lm)$ possesses a unique contragredient Hermitian inner product $(\cdot,\cdot)_\lm$ such that $(v_\lm,v_\lm)_\lm = 1$.

Given a $\bbK_\infty$-lattice $\clL_M$ of a $\bbK$-vector space $M$, set
$$
\ol{\clL}_M := \clL_M/q\inv \clL_M,
$$
and let $\ev_\infty : \clL_M \rightarrow \ol{\clL}_M$ denote the quotient map. For example, when $M$ is a $\U$-module equipped with a contragredient Hermitian inner product $(\cdot,\cdot)_M$, we can take
\begin{align}
  \clL_M := \{ v \in M \mid (v,v)_M \in \bbK_\infty \}. \nonumber
\end{align}
The inner product $(\cdot,\cdot)_M$ induces a $\C$-valued Hermitian inner product $(\cdot,\cdot)_M$ on $\ol{\clL}_M$. For example, the canonical basis $\bfB(\lm)$ of $V(\lm)$ forms an almost orthonormal basis of $V(\lm)$, and hence, the free basis of $\clL(\lm) := \clL_{V(\lm)}$. Furthermore, the crystal basis $\clB(\lm) = \ev_\infty(\bfB(\lm))$ forms an orthonormal basis of $\ol{\clL}(\lm) := \ol{\clL}_{V(\lm)}$.

Given an integrable module $M$, let $\Etil_i,\Ftil_i$, $i \in I$ denote Kashiwara operators acting on it. If $M$ possesses a contragredient Hermitian inner product, then Kashiwara operators preserve $\clL_M$, and hence induce $\C$-linear operators on $\ol{\clL}_M$. Furthermore, $\Etil_i$ and $\Ftil_i$ on $\ol{\clL}_M$ are adjoint to each other.

A crystal base of an integrable $\U$-module $M$ is a pair $(\mathcal{L}_M, \mathcal{B}_M)$ consisting of a $\mathbb{K}_\infty$-lattice $\mathcal{L}_M$ of $M$ which is compatible with the weight space decomposition of $M$ and is preserved by the Kashiwara operators, and a $\C$-basis $\clB_M$ of $\ol{\clL}_M$ which is compatible with the weight space decomposition of $M$ and which forms a seminormal crystal with respect to the Kashiwara operators.

A based $\U$-module is a $\U$-module $M$ with a bar-involution $\psi_M$, a crystal base $(\clL_M, \clB_M)$, and an $\bfA$-form $M_{\bfA}$ satisfying the following:
\begin{enumerate}
\item The quotient map $\ev_\infty : \clL_M \rightarrow \ol{\clL}_M$ restricts to an isomorphism $\clL_M \cap M_{\bfA} \cap \psi_M(\clL_M) \rightarrow \ol{\clL}_M$ of $\C$-vector spaces; let $G$ denote its inverse.
\item For each $b \in \clB_M$, it holds that $\psi_M(G(b)) = G(b)$.
\end{enumerate}
Given a based $\U$-modules $M,N$ with crystal bases $\clB_M,\clB_N$, a $\U$-module homomorphism $f : M \rightarrow N$ is said to be a based module homomorphism if $f(G(\clB_M)) \subset G(\clB_N) \sqcup \{0\}$ and $\operatorname{Ker} f$ is spanned by a subset of $G(\clB_M)$.

\begin{ex}\normalfont
Let $\lm \in X^+$. As we have seen above, the irreducible highest weight module $V(\lm)$ possesses a bar-involution $\psi_\lm$, a crystal base $(\clL(\lm), \clB(\lm))$, and an $\bfA$-form $V(\lm)_{\bfA}$. With respect to these structures, $V(\lm)$ is a based module, and we have $\bfB(\lm) = G(\clB(\lm))$.
\end{ex}

\section{$\imath$Quantum groups and $\imath$crystals of quasi-split types}\label{Section: iquantum groups and icrystals}
In this section, we recall what the $\imath$quantum group of quasi-split type is, and formulate the notion of based $\Ui$-modules in a similar way to based $\U$-modules. Also, we introduce the notion of $\imath$crystals, which is the fundamental tool in this paper.

\subsection{$\imath$Quantum groups of quasi-split types}
Let $\tau$ be a Dynkin diagram involution on $I$, i.e., $\tau$ is a permutation on $I$ such that $\tau^2 = \id$ and $a_{\tau(i),\tau(j)} = a_{i,j}$ for all $i,j \in I$. We further assume that there exist automorphisms (also denoted by $\tau$) on $X,Y$ such that $\tau(\alpha_i) = \alpha_{\tau(i)}$ and $\tau(h_i) = h_{\tau(i)}$ for all $i \in I$, and $\la \tau(h), \tau(\lm) \ra = \la h,\lm \ra$ for all $h \in Y$ and $\lm \in X$. For each $i \in I$, fix $\varsigma_i \in \C(q)^\times$ and $\kappa_i \in \C(q)$ satisfying the following conditions:
\begin{itemize}
\item $\kappa_i = 0$ unless $\tau(i) = i$ and $a_{j,i} \in 2\Z$ for all $j \in I$ with $\tau(j) = j$.
\item $\varsigma_i = \varsigma_{\tau(i)}$ if $a_{i,\tau(i)} = 0$.
\end{itemize}
Then, the associated $\imath$quantum group $\Ui = \Ui_{\bfvarsigma,\bfkappa}$ is defined to be the subalgebra of $\U$ generated by $B_i,K_h$, $i \in I$, $h \in Y^\imath$, where
$$
B_i := F_i + \varsigma_i E_{\tau(i)} K_i\inv + \kappa_i K_i\inv, \qu Y^\imath := \{ h \in Y \mid \tau(h) = -h \}.
$$

\begin{ex}\label{set up for diagonal type}\normalfont
Suppose that our Satake diagram is of diagonal type, i.e., $a_{i,\tau(i)} = 0$ for all $i \in I$. Then, one can choose $I_1,I_2 \subset I$ in a way such that $I = I_1 \sqcup I_2$, $a_{i_1,i_2} = 0$ for all $i_1 \in I_1$, $i_2 \in I_2$, and $\tau(I_1) = I_2$. According to this decomposition, we obtain $X = X_{I_1} \oplus X_{I_2}$, $Y = Y_{I_1} \oplus Y_{I_2}$. For each $\lm \in X$ and $h \in Y$, we write $\lm = \lm_1+\lm_2$, $h = h_1+h_2$, where $\lm_i \in X_{I_i}$, $h_i \in Y_{I_i}$. Let $I'$ be a copy of $I_1$, and let $I_1 \rightarrow I':\ i \mapsto i'$ denote the isomorphism. This induces isomorphisms $X_{I_1} \rightarrow X_{I'};\ \lm \mapsto \lm'$ and $Y_{I_1} \rightarrow Y_{I'};\ h \mapsto h'$. For each $i \in I'$, set $i_1 \in I_1$ to be the preimage of $i$, and $i_2 := \tau(i_1) \in I_2$. Then, there exists an isomorphism $\U \rightarrow \U_{I'} \otimes \U_{I'}$ such that
$$
E_i \mapsto \begin{cases}
F_{i'} \otimes 1 & \IF i \in I_1, \\
1 \otimes E_{\tau(i)'} & \IF i \in I_2,
\end{cases} \ F_i \mapsto \begin{cases}
E_{i'} \otimes 1 & \IF i \in I_1, \\
1 \otimes F_{\tau(i)'} & \IF i \in I_2,
\end{cases} \ K_h \mapsto K_{-h'_1} \otimes K_{\tau(h_2)'}.
$$
For each $i \in I$, set $\varsigma_i = 1$ and $\kappa_i = 0$. Then, the associated $\imath$quantum group is a subalgebra of $\U_{I'} \otimes \U_{I'}$ generated by
$$
B_{i_1} = E_i \otimes 1 + K_i \otimes E_i, \qu B_{i_2} = 1 \otimes F_i + F_i \otimes K_i\inv, \qu K_{h} \otimes K_h
$$
for $i \in I'$, $h \in Y_{I'}$. Therefore, we have
$$
\Ui = \Delta(\U_{I'}).
$$
\end{ex}

For each $J = \{ j_1,\ldots,j_k \} \subset I$ of finite type such that $\tau(J) = J$, let $\Ui_J = \Ui_{j_1,\ldots,j_k}$ denote the subalgebra of $\Ui$ generated by $B_j,K_jK_{\tau(j)}\inv$, $j \in J$.

A set of defining relations is known \cite[Theorem 3.1]{CLW18}: For $h,h' \in Y^\imath$, $i \neq j \in I$, $\ol{p} \in \Z/2\Z$,
\begin{align}\label{defining relation for Ui}
\begin{split}
&K_0 = 1, \qu K_h K_{h'} = K_{h+h'}, \\
&K_h B_i = q^{\la h,-\alpha_i \ra} B_i K_h, \\
&\sum_{n=0}^{1-a_{i,j}} (-1)^n B_i^{(n)} B_j B_i^{(1-a_{i,j}-n)} = \delta_{\tau(i),j} \frac{(-1)^{a_{i,\tau(i)}}}{q_i-q_i\inv}B_i^{(-a_{i,\tau(i)})} \\
&\qu \cdot (q_i^{a_{i,\tau(i)}}(q_i^{-2};q_i^{-2})_{-a_{i,\tau(i)}}\varsigma_{\tau(i)}k_i - (q_i^2;q_i^2)_{-a_{i,\tau(i)}}\varsigma_ik_i\inv) \qu \IF \tau(i) \neq i, \\
&\sum_{n=0}^{1-a_{i,j}} (-1)^n B_{i,\ol{a_{i,j}+p}}^{(n)} B_j B_{i,\ol{p}}^{(1-a_{i,j}-n)} = 0 \qu \IF \tau(i) = i,
\end{split}
\end{align}
where
$$
k_i := K_iK_{\tau(i)}\inv, \qu (x;x)_n := \prod_{k=1}^n (1-x^k), \qu B_i^{(n)} := \frac{1}{[n]_i!}B_i^n,
$$
and
\begin{align}
\begin{split}
&B_{i,\ol{0}}^{(n)} := \begin{cases}
\frac{1}{[2k+1]_i!} B_i \prod_{j=1}^k (B_i^2 - q_i\varsigma_i[2j]_i^2) \qu & \IF n = 2k+1, \\
\frac{1}{[2k]_i!} \prod_{j=1}^k (B_i^2 - q_i\varsigma_i[2j-2]_i^2) \qu & \IF n = 2k,
\end{cases} \\
&B_{i,\ol{1}}^{(n)} := \begin{cases}
\frac{1}{[2k+1]_i!} B_i \prod_{j=1}^k (B_i^2 - q_i\varsigma_i[2j-1]_i^2) \qu & \IF n = 2k+1, \\
\frac{1}{[2k]_i!} \prod_{j=1}^k (B_i^2 - q_i\varsigma_i[2j-1]_i^2) \qu  & \IF n = 2k.
\end{cases}
\end{split} \nonumber
\end{align}

Let us introduce a family of $1$-dimensional $\Ui$-modules which will be one of the key ingredients in later argument.

\begin{prop}\label{existence of V(0)sigma}
Let $\sigma \in X$ be such that $\la h_i-h_{\tau(i)},\sigma \ra = 0$ for all $i \in I$ with $a_{i,\tau(i)} = 0$. Then, there exists a $1$-dimensional $\Ui$-module $V(0)^\sigma = \bbK v_0^\sigma$ such that
$$
B_i v_0^\sigma = 0, \qu K_h v_0^\sigma = q^{\la h,\sigma \ra} v_0^\sigma \qu \Forall i \in I, h \in Y^\imath.
$$
\end{prop}

\begin{proof}
The assertion follows from relations \eqref{defining relation for Ui}.
\end{proof}

\begin{prop}
Let $\sigma \in X$ with $\la h_i-h_{\tau(i)},\sigma \ra = 0$ for all $i \in I$ with $a_{i,\tau(i)} = 0$, and $M$ a $\U$-module. Then, for each $\lm \in X$, $v \in M_\lm$, $i \in I$, and $h \in Y^\imath$, we have
\begin{align}
\begin{split}
&K_h (v_0^\sigma \otimes v) = q^{\la h,\sigma+\lm \ra} v_0^\sigma \otimes v, \\
&B_i(v_0^\sigma \otimes v) = v_0^\sigma \otimes (F_i + q_i^{-\la h_i-h_{\tau(i)},\sigma \ra}\varsigma_i E_{\tau(i)} K_i\inv)v.
\end{split} \nonumber
\end{align}
\end{prop}

\begin{proof}
By the definitions, we have
\begin{align}
\begin{split}
&\Delta(K_h) = K_h \otimes K_h, \\
&\Delta(B_i) = B_i \otimes K_i\inv + 1 \otimes F_i + k_i\inv \otimes (\varsigma_i E_{\tau(i)} K_i\inv).
\end{split} \nonumber
\end{align}
Then, the assertion follows.
\end{proof}

This result shows that the $\mathbf{U}^\imath$-module $V(0)^\simga \otimes M$ behaves much like $M$ viewed as a $\Ui_{\bfvarsigma',\bfkappa'}$-module with weights shifted by $\sigma$, where $\varsigma'_i = q_i^{-\la h_i-h_{\tau(i)},\sigma \ra} \varsigma_i$ and $\kappa'_i = 0$.

\subsection{Based $\Ui$-modules}\label{Subsection: based Uimodules}
Let us further assume the following conditions on the parameters $\varsigma_i,\kappa_i$:
\begin{align}\label{axiom for sigma and kappa}
\begin{split}
&\varsigma_i, \kappa_i \in \Z[q,q\inv], \\
&\ol{\kappa_i} = \kappa_i, \\
&\varsigma_{\tau(i)} = q_i^{-a_{i,\tau(i)}} \ol{\varsigma_i}.
\end{split}
\end{align}
This assumption ensures the existence of the $\imath$bar-involution $\psii$ on $\Ui$ and the $\imath$canonical basis $\bfB^\imath(\lm)$ of $V(\lm)$, $\lm \in X^+$ (see \cite{BW21}).

A $\Ui$-module $M$ is said to have an $\imath$bar-involution $\psii_M : M \rightarrow M$ if
$$
\psii_M(xv) = \psii(x) \psii_M(v) \Forall x \in \Ui,\ v \in M.
$$

\begin{ex}\normalfont
\ \begin{enumerate}
\item The irreducible highest weight module $V(\lm)$ possesses a unique $\imath$bar-involution $\psii_\lm$ such that $\psii_\lm(v_\lm) = v_\lm$.
\item The $1$-dimensional $\Ui$-module $V(0)^\sigma$ in Proposition \ref{existence of V(0)sigma} possesses a unique $\imath$bar-involution $\psi_0^{\imath, \sigma}$ such that $\psi_0^{\imath,\sigma}(v_0^\sigma) = v_0^\sigma$.
\end{enumerate}
\end{ex}

Set $X^\imath := X/\{ \lm+\tau(\lm) \mid \lm \in X \}$, and $\ol{\cdot} : X \rightarrow X^\imath$ the quotient map. The perfect pairing $\la \cdot,\cdot \ra : Y \times X \rightarrow \Z$ induces a bilinear pairing $\la \cdot,\cdot \ra : Y^\imath \times X^\imath \rightarrow \Z$. For each $\zeta \in X^\imath$ and $i \in I$ with $\tau(i) = i$, the parity of $\la h_i,\lm \ra$ is independent of $\lm \in X$ satisfying $\ol{\lm} = \zeta$. We call $\ol{\la h_i,\lm \ra} \in \Z/2\Z$ the value of $\zeta$ at $i$.

\begin{rem}\label{rem: iweight}\normalfont
When $I$ is of finite type, $\zeta \in X^\imath$ is uniquely determined by the values $\la h_i-h_{\tau(i)},\zeta \ra \in \Z$ for $i \in I$ with $\tau(i) \neq i$, and the values of $\zeta$ at $i \in I$ with $\tau(i) = i$. In this way, we often identify $\zeta$ with an element of $(\Z \sqcup (\Z/2\Z))^I$.
\end{rem}

A $\Ui$-module $M$ is said to be an $X^\imath$-weight module if it has a decomposition $M = \bigoplus_{\zeta \in X^\imath} M_\zeta$ satisfying the following:
\begin{itemize}
\item $K_h v = q^{\la h,\zeta \ra} v$ for all $h \in Y^\imath$, $\zeta \in X^\imath$, $v \in M_\zeta$.
\item $B_i M_\zeta \subset M_{\zeta-\ol{\alpha_i}}$ for all $i \in I$, $\zeta \in X^\imath$.
\end{itemize}
Such a decomposition is called an $X^\imath$-weight space decomposition.

\begin{ex}\normalfont
\ \begin{enumerate}
\item Let $M = \bigoplus_{\lm \in X} M_\lm$ be a weight $\U$-module. Then, it is an $X^\imath$-weight module with $X^\imath$-weight space decomposition
$$
M = \bigoplus_{\zeta \in X^\imath} M_\zeta, \qu M_\zeta := \bigoplus_{\substack{\lm \in X \\ \ol{\lm} = \zeta}} M_\lm.
$$
We call it the canonical $X^\imath$-weight module structure of $M$.
\item The $1$-dimensional $\Ui$-module $V(0)^\sigma$ possesses an $X^\imath$-weight module structure given by $V(0)^\sigma = (V(0)^\sigma)_{\ol{\sigma}}$.
\end{enumerate}
\end{ex}

Let $\Uidot = \bigoplus_{\zeta \in X^\imath} \Ui \mathbf{1}_\zeta$ denote the modified $\imath$quantum group, and $\Uidot_{\bfA}$ its $\bfA$-form. Let $M$ be an $X^\imath$-weight module. Then it has a natural $\Uidot$-module structure (\cite[Subsection 3.3]{W21b}). An $\bfA$-form of $M$ is an $\bfA$-lattice $M_{\bfA}$ such that $\Uidot_{\bfA} M_{\bfA} \subset M_{\bfA}$.

\begin{ex}\label{A-forms}\normalfont
\ \begin{enumerate}
\item $V(\lm)_{\bfA}$ is an $\bfA$-form of $V(\lm)$ as a $\Ui$-module with the canonical $X^\imath$-weight module structure.
\item If the $1$-dimensional $\Ui$-module $V(0)^\sigma$ has an $\bfA$-form, it must be the subspace $V(0)^\sigma_{\bfA} := \bfA v_0^\sigma$. This is the case when, for example, $\varsigma_i = q_i\inv$ and $\kappa_i = [s_i]_i$ for some $s_i \in \Z$ for all $i \in I$ with $a_{i,\tau(i)} = 2$ (in this case, an explicit generating set of $\Uidot_{\bfA}$ is known \cite{BeW18,BeW18b}).
\end{enumerate}
\end{ex}

A based $\Ui$-module is an $X^\imath$-weight module $M$ with an $\imath$bar-involution $\psii_M$, a $\bbK_\infty$-lattice $\clL_M$, a $\mathbb{C}$-basis $\clB_M$ of $\overline{\clL}_M$, and an $\bfA$-form $M_{\bfA}$ satisfying the following:
\begin{enumerate}
\item The quotient map $\ev_\infty : \clL_M \rightarrow \ol{\clL}_M$ restricts to an isomorphism $\clL_M \cap M_{\bfA} \cap \psii_M(\clL_M) \rightarrow \ol{\clL}_M$ of $\C$-vector spaces; let $G^\imath$ denote its inverse.
\item For each $b \in \clB_M$, it holds that $\psii_M(G^\imath(b)) = G^\imath(b)$.
\end{enumerate}
The Homomorphisms of based $\mathbf{U}^\imath$-modules are defined in the same way as those of based $\mathbf{U}$-modules.

\begin{ex}\normalfont
\ \begin{enumerate}
\item As we have seen so far, $V(\lm)$ possesses an $\imath$bar-involution $\psii_\lm$, a $\bbK_\infty$-lattice $\clL(\lm)$, a $\mathbb{C}$-basis $\clB(\lm)$, and an $\bfA$-form $V(\lm)_{\bfA}$. With respect to theses structures, $V(\lm)$ is a based $\Ui$-module, and we have $\bfB^\imath(\lm) = G^\imath(\clB(\lm))$.
\item If the $1$-dimensional $\Ui$-module $V(0)^\sigma$ has an $\bfA$-form, then it is a based module with respect to the $\imath$bar-involution $\psi_0^{\imath, \sigma}$, the $\bbK_\infty$-lattice $\clL(0)^\sigma := \bbK_\infty v_0^\sigma$, the $\mathbb{C}$-basis $\clB(0)^\sigma := \{ b_0^\sigma := \mathrm{ev}_\infty(v_0^\sigma) \}$, and the $\bfA$-form $V(0)^\sigma_{\bfA}$. Then, $G^\imath(b_0^\sigma) = v_0^\sigma$.
\end{enumerate}
\end{ex}

\begin{prop}\label{V(lm)sigma is based}
Let $\sigma \in X$ be such that $\la h_i-h_{\tau(i)},\sigma \ra = 0$ for all $i \in I$ with $a_{i,\tau(i)} = 0$. Let $\lm \in X^+$. Suppose that $V(0)^\sigma$ has an $\bfA$-form. Then, $V(\lm)^\sigma := V(0)^\sigma \otimes V(\lm)$ is a based $\Ui$-module with respect to an $\imath$bar-involution fixing $v_\lm^\sigma$, a $\bbK_\infty$-lattice $\clL(\lm)^\sigma := \clL(0)^\sigma \otimes \clL(\lm)$, and an $\bfA$-form $V(\lm)^\sigma_{\bfA} := V(0)^\sigma_{\bfA} \otimes V(\lm)_{\bfA}$.
\end{prop}

\begin{proof}
The assertion follows from \cite[Theorem 6.15]{BW21}.
\end{proof}

Let $\lm,\nu \in X^+$. By \cite[Proposition 7.1]{BW21}, there exists $\Ui$-module homomorphism
$$
\pi = \pi_{\lm,\nu} : V(\lm+\nu+\tau(\nu)) \rightarrow V(\lm)
$$
such that
$$
\pi(v_{\lm+\nu+\tau(\nu)}) = v_\lm.
$$
For each $\zeta \in X^\imath$, these homomorphisms form a projective system $\{ V(\lm) \}_{\lm \in X^+, \ol{\lm} = \zeta}$ which is asymptotically stable in the following sense:

\begin{theo}[{\cite[Theorem 7.2]{BW21}}]\label{asymptotical limit}
Let $\zeta \in X^\imath$ and $b \in \clB(\infty)$. Then, there exists a unique $G^\imath_\zeta(b) \in \Uidot \mathbf{1}_\zeta$ such that
$$
G^\imath_\zeta(b) v_\lm = G^\imath(\pi_\lm(b))
$$
for all $\lm \gg 0$ with $\ol{\lm} = \zeta$. Here, $\lm \gg 0$ means $\la h_i,\lm \ra$ is sufficiently large for all $i \in I$. Moreover, $\Bidot := \{ G^\imath_\zeta(b) \mid \zeta \in X^\imath,\ b \in \clB(\infty) \}$ forms a basis of $\Uidot$.
\end{theo}

The basis $\Bidot$ is called the $\imath$canonical basis of $\Uidot$.

Although each $V(\lm)$ is a based $\Ui$-module, the homomorphisms in the projective system above are not necessarily based.

\subsection{$\imath$Crystal}
From now on, we assume the following:
\begin{itemize}
\item $a_{i,\tau(i)} \in \{ 2,0,-1 \}$ for all $i \in I$.
\item $\varsigma_i \in \{ q_i^a \mid a \in \Z \}$ for all $i \in I$.
\item $\kappa_i \in \{ [a]_i \mid a \in \Z \}$ for all $i \in I$.
\end{itemize}
Note that the first condition is satisfied for all $I$ of finite or affine type, except of type $A_1^{(1)}$ with nontrivial $\tau$. The second condition, together with axiom \eqref{axiom for sigma and kappa} in the beginning of Subsection \ref{Subsection: based Uimodules}, forces $\varsigma_i$ to satisfy the following:
\begin{itemize}
\item $\varsigma_i = q_i\inv$ if $a_{i,\tau(i)} = 2$.
\item $\varsigma_i = 1$ if $a_{i,\tau(i)} = 0$.
\item $\varsigma_i \varsigma_{\tau(i)} = q_i$ if $a_{i,\tau(i)} = -1$.
\end{itemize}
Therefore, we have
$$
B_i = \begin{cases}
F_i + q_i\inv E_iK_i\inv + [s_i]_i \qu & \IF a_{i,\tau(i)} = 2, \\
F_i + E_{\tau(i)}K_i\inv \qu & \IF a_{i,\tau(i)} = 0, \\
F_i + q_i^{s_i} E_{\tau(i)} K_i\inv \qu & \IF a_{i,\tau(i)} = -1
\end{cases}
$$
for some $s_i \in \Z$ such that $s_i + s_{\tau(i)} = 1$ for all $i \in I$ with $a_{i,\tau(i)} = -1$.

\begin{rem}\label{wp preserves Ui}\normalfont
Our assumption on the parameters $\varsigma_i,\kappa_i$ ensures that $\wp^*(\Ui) = \Ui$ by \cite[Proposition 4.6]{BW18}. In particular, we can talk about contragredient Hermitian inner product on $\Ui$-modules.
\end{rem}

\begin{defi}\label{Def: icrystal}\normalfont
An $\imath$crystal is a set $\clB$ equipped with the following structure:
\begin{itemize}
\item $\wti : \clB \rightarrow X^\imath$: map.
\item $\beta_i : \clB \rightarrow \Z \sqcup \{ -\infty,-\infty_\ev,-\infty_\odd \}$: map, $i \in I$, where $-\infty,-\infty_\ev$, and $-\infty_\odd$ are formal symbols.
\item $\Btil_i \in \End_{\C}(\ol{\clL})$, $i \in I$, where $\ol{\clL} := \C \clB$.
\item $(\cdot,\cdot)$ : Hermitian inner product on $\ol{\clL}$ making $\clB$ an orthonormal basis.
\end{itemize}
satisfying the following axioms: Let $b,b' \in \clB$, $i \in I$.
\begin{enumerate}
\item\label{Def: icrystal 1} If $\beta_i(b) \notin \Z$, then $\Btil_i b = 0$.
\item\label{Def: icrystal 2} If $(\Btil_i b,b') \neq 0$, then $\wti(b') = \wti(b) - \ol{\alpha_i}$.
\item\label{Def: icrystal 2.5} If $(\Btil_i b,b') \neq 0$, then $(\Btil_i b,b') = (b, \Btil_{\tau(i)} b')$.
\item\label{Def: icrystal 2.6} If $\Btil_i b \in \clB$, then $\Btil_{\tau(i)} \Btil_i b = b$.
\item\label{Def: icrystal 3} If $a_{i,\tau(i)} = 2$, then
\begin{enumerate}
\item\label{Def: icrystal 3a} $\beta_i(b) \in \Z \sqcup \{ -\infty_\ev, -\infty_\odd \}$.
\item\label{Def: icrystal 3b} $\ol{\beta_i(b)+s_i} = \wti_i(b)$, where $\wti_i(b)$ denotes the value of $\wti(b)$ at $i$ (see before Remark \ref{rem: iweight}). We understand that
$$
-\infty_\ev + a = \begin{cases}
-\infty_\ev & \IF \ol{a} = \ol{0}, \\
-\infty_\odd & \IF \ol{a} = \ol{1},
 \end{cases} \ -\infty_\odd + a = \begin{cases}
-\infty_\odd & \IF \ol{a} = \ol{0}, \\
-\infty_\ev & \IF \ol{a} = \ol{1}
 \end{cases}
$$
for all $a \in \Z$, and $\ol{-\infty_\ev} = \ol{0}$, $\ol{-\infty_\odd} = \ol{1}$.
\item\label{Def: icrystal 3c} If $(\Btil_i b,b') \neq 0$, then $\beta_i(b') = \beta_i(b)$.
\end{enumerate}
\item\label{Def: icrystal 4} If $a_{i,\tau(i)} = 0$, then
\begin{enumerate}
\item\label{Def: icrystal 4a} $\beta_i(b) \in \Z \sqcup \{ -\infty \}$.
\item\label{Def: icrystal 4b} $\beta_i(b) = \beta_{\tau(i)}(b) + \wti_i(b)$, where $\wti_i(b) := \la h_i-h_{\tau(i)}, \wti(b) \ra$.
\item\label{Def: icrystal 4c} If $(\Btil_i b,b') \neq 0$, then $b' = \Btil_i b$ and $\beta_i(b') = \beta_i(b)-1$.
\end{enumerate}
\item\label{Def: icrystal 5} If $a_{i,\tau(i)} = -1$, then
\begin{enumerate}
\item\label{Def: icrystal 5a} $\beta_i(b) \in \Z \sqcup \{ -\infty \}$.
\item\label{Def: icrystal 5b} $\beta_i(b) = \beta_{\tau(i)}(b) + \wti_i(b)-s_i$ or $\beta_i(b) = \beta_{\tau(i)}(b) + \wti_i(b)-s_i+1$, where $\wti_i(b) := \la h_i-h_{\tau(i)}, \wti(b) \ra$.
\item\label{Def: icrystal 5c} If $\beta_i(b) \neq \beta_{\tau(i)}(b) + \wti_i(b)-s_i$ and $(\Btil_i b,b') \neq 0$, then $b' = \Btil_i b$ and $\beta_i(b') \neq \beta_{\tau(i)}(b')+\wti_i(b')-s_i$.
\item\label{Def: icrystal 5d} If $(\Btil_i b, b') \neq 0$ and $\beta_i(b') \neq \beta_{\tau(i)}(b') + \wti_i(b')-s_i$, then $\beta_i(b') = \beta_i(b)-1$.
\end{enumerate}
\end{enumerate}
\end{defi}

Let us fix a complete set $I_\tau$ of representatives for the $\tau$-orbits on $I$.

\begin{defi}\normalfont
Let $\clB$ be an $\imath$crystal. The crystal graph of $\clB$ is an $(I_\tau \times \C^\times)$-colored directed graph whose vertex set is $\clB$, and for each $b,b' \in \clB$, $i \in I_\tau$, and $z \in \C^\times$, there exists an arrow from $b$ to $b'$ labeled by $(i,z)$ if and only if $(\Btil_i b,b') = z$. We often omit the label $z$ when $z = 1$.
\end{defi}

\begin{ex}\label{icrystal of diagonal type is crystal}\normalfont
Suppose that our Satake diagram is of diagonal type. We retain the notation in Example \ref{set up for diagonal type}. If we set $I_\tau := I_2$, then an $\imath$crystal and its crystal graph are nothing but a crystal and its crystal graph associated to the Dynkin diagram $I'$. Under this identification, $\Btil_{i_1},\Btil_{i_2}, \beta_{i_1}, \beta_{i_2}$ correspond to $\Etil_i,\Ftil_i, \vep_i,\vphi_i$, respectively for each $i \in I'$.
\end{ex}

\begin{rem}\label{Btili is recovered from graph}\normalfont
Since $\clB$ is an orthonormal basis, we have
$$
\Btil_i b = \sum_{b' \in \clB} (\Btil_i b,b')b'
$$
for all $b \in \clB$. Hence, the maps $\Btil_i$, $i \in I_\tau$ can be recovered from the crystal graph of $\clB$.
\end{rem}

\begin{lem}\label{Btil is Hermite}
Let $\clB$ be an $\imath$crystal, $b,b' \in \clB$, and $i \in I$. Then, we have
$$
(\Btil_i b,b') = (b, \Btil_{\tau(i)} b').
$$
\end{lem}

\begin{proof}
By Definition \ref{Def: icrystal} \eqref{Def: icrystal 2.5}, we have $(\Btil_i b,b') = (b, \Btil_{\tau(i)}b')$ if $(\Btil_i b,b') \neq 0$. Now, suppose that $(\Btil_i b,b') = 0$. Assume contrary that $(b,\Btil_{\tau(i)} b') \neq 0$. Then, by Definition \ref{Def: icrystal} \eqref{Def: icrystal 2.5} again, we obtain
$$
(b',\Btil_i b) = (\Btil_{\tau(i)}b',b) \neq 0,
$$
which is a contradiction. Thus, the proof completes.
\end{proof}

\begin{rem}\normalfont
By Lemma \ref{Btil is Hermite} (see also Remark \ref{Btili is recovered from graph}), the maps $\Btil_{\tau(i)}$, $i \in I_\tau$ can be recovered from the crystal graph.
\end{rem}

\begin{defi}\label{Def: icrystal morphism}\normalfont
Let $\clB_1,\clB_2$ be $\imath$crystals. A morphism $\mu : \clB_1 \rightarrow \clB_2$ of $\imath$crystals is a linear map $\mu : \ol{\clL}_1 \rightarrow \ol{\clL}_2$, where $\ol{\clL_j} := \C\clB_j$, satisfying the following: Let $b_1 \in \clB_1$, $b_2 \in \clB_2$, and $i \in I$.
\begin{enumerate}
\item\label{Def: icrystal morphism 1} If $(\mu(b_1),b_2) \neq 0$, then $\wti(b_2) = \wti(b_1)$ and $\beta_i(b_2) = \beta_i(b_1)$.
\item\label{Def: icrystal morphism 2} If $\Btil_i b_1 \in \clB_1$ and $\mu(b_1), \mu(\Btil_i b_1) \neq 0$, then $\mu(\Btil_i b_1) = \Btil_i \mu(b_1)$.
\end{enumerate}
An $\imath$crystal morphism $\mu : \clB_1 \rightarrow \clB_2$ is said to be strict if $\mu(\Btil_i b) = \Btil_i \mu(b)$ for all $i \in I$, $b \in \clB_1$. A strict $\imath$crystal morphism is said to be an equivalence if the underlying map $\mu : \ol{\clL}_1 \rightarrow \ol{\clL}_2$ is a linear isomorphism; in this case, we write $\clB_1 \sim \clB_2$. A strict $\imath$crystal morphism is said to be very strict if $\mu(\clB_1) \subset \clB_2 \sqcup \{0\}$. An $\imath$crystal equivalence is said to be an isomorphism if it is very strict; in this case, we write $\clB_1 \simeq \clB_2$.
\end{defi}

\begin{rem}\normalfont
Although an $\imath$crystal isomorphism induces an isomorphism of crystal graphs, it can happen that two equivalent $\imath$crystals have non-isomorphic crystal graphs.
\end{rem}

\begin{ex}\label{examples of icrystals}\normalfont
\ \begin{enumerate}
\item\label{examples of icrystals 1} Consider the crystal basis $\clB(0) = \{ b_0 \}$ of the trivial module $V(0)$. It has an $\imath$crystal structure given as follows:
\begin{enumerate}
\item $\wti(b_0) = \ol{0}$.
\item $\beta_i(b_0) = \begin{cases}
|s_i| \qu & \IF a_{i,\tau(i)} = 2, \\
0 \qu & \IF a_{i,\tau(i)} = 0, \\
\max(-s_i,0) \qu & \IF a_{i,\tau(i)} = -1.
\end{cases}$
\item $\Btil_i b_0 = \begin{cases}
\sgn(s_i)b_0 \qu & \IF a_{i,\tau(i)} = 2, \\
0 \qu & \OW.
\end{cases}$
\end{enumerate}

\item For each $\zeta \in X^\imath$, let $\clT_\zeta = \{ t_\zeta \}$ denote the $\imath$crystal given by
$$
\wti(t_\zeta) = \zeta, \qu \beta_i(t_\zeta) = \begin{cases}
-\infty_\ev & \IF a_{i,\tau(i)} = 2 \AND \zeta_i = \ol{s_i}, \\
-\infty_\odd & \IF a_{i,\tau(i)} = 2 \AND \zeta_i \neq \ol{s_i}, \\
-\infty & \IF a_{i,\tau(i)} \neq 2,
\end{cases} \qu \Btil_i t_\zeta = 0,
$$
where $\zeta_i \in \Z/2\Z$ denotes the value of $\zeta$ at $i$.

\item\label{examples of icrystals 3} Suppose that $I = \{i\}$. For each $n \in \Z$, let $\clB^\imath(n) = \{ b \}$ denote the $\imath$crystal given by
$$
\wti_i(b) = \ol{n+s_i}, \qu \beta_i(b) = |n|, \qu \Btil_i b = \sgn(n) b.
$$
The crystal graph of $\clB^\imath(n)$ is as follows:
$$
\xymatrix{
b \ar@(ur,dr)^-{(i,\sgn(n))}
}
$$

\item Suppose that $I = \{ i \}$. For each $n \in \Z_{> 0}$, let $\clB^\imath(n;-n) = \{ b_+,b_- \}$ denote the $\imath$crystal given by
$$
\wti_i(b_\pm) = \ol{n+s_i}, \qu \beta_i(b_\pm) = n, \qu \Btil_i b_\pm = b_\mp
$$
The crystal graph of $\clB^\imath(n;-n)$ is as follows:
$$
\xymatrix{
b_+ \ar@<0.5ex>[r]^-i & b_- \ar@<0.5ex>[l]^-i
}
$$
There exists an $\imath$crystal equivalence $\clB^\imath(n) \sqcup \clB^\imath(-n) \rightarrow \clB^\imath(n;-n)$ which sends $b \in \clB^\imath(\pm n)$ to $\frac{1}{\sqrt{2}}(b_+ \pm b_-)$. Note that the crystal graphs of $\clB^\imath(n) \sqcup \clB^\imath(-n)$ and $\clB^\imath(n;-n)$ are not isomorphic.

\item\label{examples of icrystals 5} Suppose that $I = \{ i,\tau(i) \}$ and $a_{i,\tau(i)} = 0$. For each $n \in \Z_{\geq 0}$, let $\clB^\imath(n) = \{ b_k \mid 0 \leq k \leq n \}$ denote the $\imath$crystal given by
\begin{align}
\begin{split}
&\wti_i(b_k) = n-2k, \qu \beta_i(b_k) = n-k, \qu \Btil_i b_k = b_{k+1}, \\
&\wti_{\tau(i)}(b_k) = -n+2k, \qu \beta_{\tau(i)}(b_k) = k, \qu \Btil_{\tau(i)} b_k = b_{k-1},
\end{split} \nonumber
\end{align}
where $b_{-1} = b_{n+1} = 0$. The crystal graph of $\clB^\imath(n)$ (with $I_\tau = \{i\}$) is as follows:
$$
\xymatrix{
b_0 \ar[r]^-i & b_1 \ar[r]^-i & \cdots \ar[r]^-i & b_{n-1} \ar[r]^-i & b_n
}
$$

\item\label{examples of icrystals 6} Suppose that $I = \{ i,\tau(i) \}$ and $a_{i,\tau(i)} = -1$. For each $n_- \in \Z_{\geq 0}$ and $n_+ \in \Z$, let $\clB^\imath(n_-,n_+) = \{ b_k \mid 0 \leq k \leq n_- \}$ denote the $\imath$crystal given by
\begin{align}
\begin{split}
&\wti_i(b_k) = n_-+n_+-3k, \\
&\beta_i(b_k) = n_--k+\max(n_+-s_i-k,0) = \begin{cases}
n_-+n_+-s_i-2k & \IF 0 \leq k \leq n_+-s_i, \\
n_--k & \IF n_+-s_i < k \leq n_-,
\end{cases} \\
&\Btil_i b_k = b_{k+1}, \\
&\wti_{\tau(i)}(b_k) = -n_--n_++3k, \\
&\beta_{\tau(i)}(b_k) = k+\max(-n_+-s_{\tau(i)}+k,0) = \begin{cases}
k & \IF 0 \leq k \leq n_+-s_i, \\
-n_+-s_{\tau(i)}+2k & \IF n_+-s_i < k \leq n_-,
\end{cases} \\
&\Btil_{\tau(i)} b_k = b_{k-1},
\end{split} \nonumber
\end{align}
where $b_{-1} = b_{n_-+1} = 0$. The crystal graph of $\clB^\imath(n_-,n_+)$ (with $I_\tau = \{i\}$) is as follows:
$$
\xymatrix{
b_0 \ar[r]^-i & b_1 \ar[r]^-i & \cdots \ar[r]^-i & b_{n_--1} \ar[r]^-i & b_{n_-}
}
$$

\item\label{examples of icrystals 7} Suppose that $I = \{ i,\tau(i) \}$ and $a_{i,\tau(i)} = -1$. For each $n_- \in \Z_{> 0}$ and $n_+ \in \Z$ with $-1 < n_+-s_i < n_-$, let $\clB^\imath(n_-,n_+;\vee) = \{ b_{k,\pm} \mid 0 \leq k \leq n_+-s_i \} \sqcup \{ b_{k} \mid n_+-s_i < k \leq n_- \}$ denote the $\imath$crystal given by
\begin{align}
\begin{split}
&\wti_i(b_{k,\pm}) = n_-+n_+-3k, \qu \wti_i(b_k) = n_-+n_+-3k, \\
&\beta_i(b_{k,\pm}) = n_-+n_+-s_i-2k, \qu \beta_i(b_k) = n_--k, \\
&\Btil_i b_{k,\pm} = \begin{cases}
b_{k+1,\pm} & \IF k \neq n_+-s_i, \\
\frac{1}{\sqrt{2}} b_{k+1} & \IF k = n_+-s_i,
\end{cases} \qu \Btil_i b_{k} = b_{k+1}, \\
&\wti_{\tau(i)}(b_{k,\pm}) = -n_--n_++3k, \qu \wti_{\tau(i)}(b_k) = -n_--n_++3k, \\
&\beta_{\tau(i)}(b_{k,\pm}) = k, \qu \beta_{\tau(i)}(b_k) = -n_+-s_{\tau(i)}+2k, \\
&\Btil_{\tau(i)} b_{k,\pm} = b_{k-1,\pm}, \qu \Btil_{\tau(i)}b_k = \begin{cases}
b_{k-1} & \IF k \neq n_+-s_i+1, \\
\frac{1}{\sqrt{2}} (b_{k-1,+}+b_{k-1,-}) & \IF k = n_+-s_i+1,
\end{cases}
\end{split} \nonumber
\end{align}
where $b_{-1,\pm} = b_{n_-+1,+} = 0$. The crystal graph of $\clB^\imath(n_-,n_+;\vee)$ (with $I_\tau = \{i\}$) is as follows:
$$
\xymatrix{
b_{0,+} \ar[r]^-i & b_{1,+} \ar[r]^-i & \cdots \ar[r]^-i & b_{n_+-s_i,+} \ar[dr]^-{(i,\frac{1}{\sqrt{2}})} \\
&&&&b_{n_+-s_i+1} \ar[r]^-i & \cdots \ar[r]^-i &  b_{n_-} \\
b_{0,-} \ar[r]_-i & b_{1,-} \ar[r]_-i & \cdots \ar[r]_-i & b_{n_+-s_i,-} \ar[ur]_-{(i,\frac{1}{\sqrt{2}})}
}
$$
There exists an $\imath$crystal equivalence $\clB^\imath(n_-,n_+) \sqcup \clB^\imath(n_+-s_i,n_-+s_i) \rightarrow \clB^\imath(n_-,n_+;\vee)$ which sends $b_k \in \clB^\imath(n_+-s_i,n_-+s_i)$ to $\frac{1}{\sqrt{2}}(b_{k,+} - b_{k,-})$ and $b_k \in \clB^\imath(n_-,n_+)$ to $\frac{1}{\sqrt{2}}(b_{k,+} + b_{k,-})$ when $k \leq n_+-s_i$, while $b_{k}$ when $k > n_+-s_i$.

\item\label{examples of icrystals 8} Suppose that $I = \{ i,\tau(i) \}$ and $a_{i,\tau(i)} = -1$. For each $n_- \in \Z_{> 0}$ and $n_+ \in \Z$ with $-1 < n_+-s_i < n_-$, let $\clB^\imath(n_-,n_+;\wedge) = \{ b_{k,\pm} \mid n_++s_{\tau(i)} \leq k \leq n_- \} \sqcup \{ b_{k} \mid 0 \leq k < n_++s_{\tau(i)} \}$ denote the $\imath$crystal given by
\begin{align}
\begin{split}
&\wti_i(b_{k,\pm}) = n_-+n_+-3k, \qu \wti_i(b_k) = n_-+n_+-3k, \\
&\beta_i(b_{k,\pm}) = n_--k, \qu \beta_i(b_k) = n_-+n_+-s_i-2k, \\
&\Btil_i b_{k,\pm} = b_{k+1,\pm}, \qu \Btil_i b_{k} = \begin{cases}
b_{k+1} & \IF k \neq n_++s_{\tau(i)}-1, \\
\frac{1}{\sqrt{2}}(b_{k+1,+} + b_{k+1,-}) & \IF k = n_++s_{\tau(i)}-1,
\end{cases} \\
&\wti_{\tau(i)}(b_{k,\pm}) = -n_--n_++3k, \qu \wti_{\tau(i)}(b_k) = -n_--n_++3k, \\
&\beta_{\tau(i)}(b_{k,\pm}) = -n_+-s_{\tau(i)}+2k, \qu \beta_{\tau(i)}(b_k) = k, \\
&\Btil_{\tau(i)} b_{k,\pm} = \begin{cases}
b_{k-1,\pm} & \IF k \neq n_++s_{\tau(i)}, \\
\frac{1}{\sqrt{2}} b_{k-1} & \IF k = n_++s_{\tau(i)},
\end{cases} \qu \Btil_{\tau(i)} b_k = b_{k-1},
\end{split} \nonumber
\end{align}
where $b_{-1,+} = b_{n_-+1,\pm} = 0$. The crystal graph of $\clB^\imath(n_-,n_+;\wedge)$ is as follows:
$$
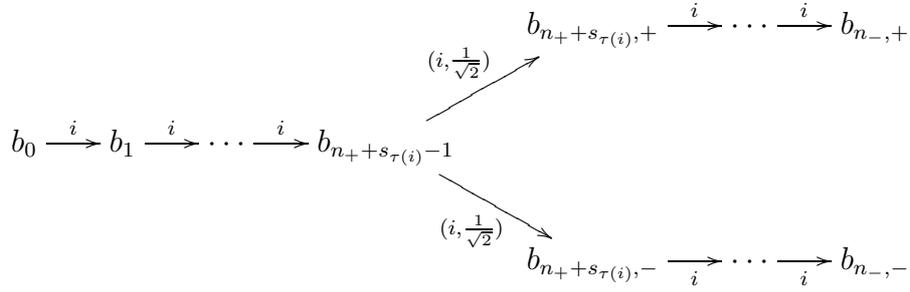
\xymatrix@C=20pt{
&&&& b_{n_++s_{\tau(i)},+} \ar[r]^-i & \cdots \ar[r]^-i & b_{n_-,+}\\
b_{0} \ar[r]^-i & b_{1} \ar[r]^-i & \cdots \ar[r]^-i &  b_{n_++s_{\tau(i)}-1} \ar[dr]_-{(i,\frac{1}{\sqrt{2}})} \ar[ur]^-{(i,\frac{1}{\sqrt{2}})} \\
&&&& b_{n_++s_{\tau(i)},-} \ar[r]_-i & \cdots \ar[r]_-i & b_{n_-,-}
}
$$
There exists an $\imath$crystal equivalence $\clB^\imath(n_-,n_+) \sqcup \clB^\imath(n_--n_+-s_{\tau(i)},-n_+-2s_{\tau(i)}) \rightarrow \clB^\imath(n_-,n_+;\wedge)$ which sends $b_k \in \clB^\imath(n_--n_+-s_{\tau(i)},-n_+-2s_{\tau(i)})$ to $\frac{1}{\sqrt{2}}(b_{k+n_++s_{\tau(i)},+} - b_{k+n_++s_{\tau(i)},-})$, and $b_k \in \clB^\imath(n_-,n_+)$ to $b_k$ when $k < n_++s_{\tau(i)}$, while $\frac{1}{\sqrt{2}}(b_{k,+} + b_{k,-})$ when $k \geq n_++s_{\tau(i)}$.
\end{enumerate}
\end{ex}

\begin{prop}\label{property for a = 0}
Let $\clB$ be an $\imath$crystal, $i \in I$ with $a_{i,\tau(i)} = 0$, and $b \in \clB$ with $\Btil_i b \neq 0$. Then, we have $\wti_i(\Btil_i b) = \wti_i(b)-2$ and $\beta_{\tau(i)}(\Btil_i b) = \beta_{\tau(i)}(b)+1$.
\end{prop}

\begin{proof}
Set $b' := \Btil_i b \in \clB$. Then, by Definition \ref{Def: icrystal} \eqref{Def: icrystal 2}, we have
$$
\wti_i(b') = \la h_i-h_{\tau(i)}, \wti(b)-\ol{\alpha_i} \ra = \wti_i(b)-2.
$$
This proves the first assertion. Next, by Definition \ref{Def: icrystal} \eqref{Def: icrystal 4c}, we have $\beta_i(b') = \beta_i(b)-1$, and hence,
$$
\beta_{\tau(i)}(b') = \beta_i(b') - \wti_i(b') = \beta_i(b)-\wti_i(b)+1 = \beta_{\tau(i)}(b)+1.
$$
This proves the second assertion.
\end{proof}

\begin{prop}\label{basic property for a = -1}
Let $\clB$ be an $\imath$crystal, $i \in I$ with $a_{i,\tau(i)} = -1$, and $b,b',b'' \in \clB$ with $(\Btil_i b,b'), (\Btil_{\tau(i)}b,b'') \neq 0$. Then, the following hold:
\begin{enumerate}
\item $\wti_i(b') = \wti_i(b)-3$.
\item If $\beta_i(b) \neq \beta_{\tau(i)}(b)+\wti_i(b)-s_i$, then $b' = \Btil_i b$, $\beta_i(b') = \beta_i(b)-1$, and $\beta_{\tau(i)}(b') = \beta_{\tau(i)}(b)+2$.
\item If $\beta_i(b) \neq \beta_{\tau(i)}(b)+\wti_i(b)-s_i$ and $\beta_i(b'') \neq \beta_{\tau(i)}(b'')+\wti_i(b'')-s_i$, then $b'' = \Btil_{\tau(i)} b$, $\beta_i(b'') = \beta_i(b)+1$, and $\beta_{\tau(i)}(b'') = \beta_{\tau(i)}(b)-2$.
\item If $\beta_i(b) \neq \beta_{\tau(i)}(b)+\wti_i(b)-s_i$ and $\beta_i(b'') = \beta_{\tau(i)}(b'')+\wti_i(b'')-s_i$, then $\beta_i(b'') = \beta_i(b)+1$ and $\beta_{\tau(i)}(b'') = \beta_{\tau(i)}(b)-1$.
\item $\wti_i(b'') = \wti_i(b)+3$.
\item If $\beta_i(b) = \beta_{\tau(i)}(b)+\wti_i(b)-s_i$, then $b'' = \Btil_{\tau(i)} b$, $\beta_{\tau(i)}(b'') = \beta_{\tau(i)}(b)-1$, and $\beta_i(b'') = \beta_i(b)+2$.
\item If $\beta_i(b) = \beta_{\tau(i)}(b)+\wti_i(b)-s_i$ and $\beta_i(b') = \beta_{\tau(i)}(b')+\wti_i(b')-s_i$, then $b' = \Btil_i b$, $\beta_{\tau(i)}(b') = \beta_{\tau(i)}(b)+1$, and $\beta_i(b') = \beta_i(b)-2$.
\item If $\beta_i(b) = \beta_{\tau(i)}(b)+\wti_i(b)-s_i$ and $\beta_i(b') \neq \beta_{\tau(i)}(b')+\wti_i(b')-s_i$, then $\beta_{\tau(i)}(b') = \beta_{\tau(i)}(b)+1$ and $\beta_i(b') = \beta_i(b)-1$.
\end{enumerate}
\end{prop}

\begin{proof}
The first assertion follows from Definition \ref{Def: icrystal} \eqref{Def: icrystal 2} as Proposition \ref{property for a = 0}.

The second assertion follows from Definition \ref{Def: icrystal} \eqref{Def: icrystal 5b} -- \eqref{Def: icrystal 5d} and the first assertion of the proposition.

Let us prove the third assertion. By Lemma \ref{Btil is Hermite}, we have $(\Btil_i b'',b) = (b'', \Btil_{\tau(i)}b) \neq 0$. Then, the second assertion of the proposition implies that $b = \Btil_i b''$, $\beta_i(b) = \beta_i(b'')-1$, and $\beta_{\tau(i)}(b) = \beta_{\tau(i)}(b'')+2$. Now, by Definition \ref{Def: icrystal} \eqref{Def: icrystal 2.6}, we obtain
$$
\Btil_{\tau(i)} b = \Btil_{\tau(i)} \Btil_i b'' = b''.
$$
Thus, the third assertion follows.

Let us prove the fourth assertion. Again, we have $(\Btil_i b'', b) \neq 0$. Then, the first assertion of the proposition and Definition \ref{Def: icrystal} \eqref{Def: icrystal 5d} imply that $\wti_i(b) = \wti_i(b'')-3$ and $\beta_i(b) = \beta_i(b'')-1$, respectively. Now, we compute as
\begin{align}
\begin{split}
\beta_{\tau(i)}(b'') &= \beta_i(b'')-\wti_i(b'')-s_i \\
&= (\beta_i(b)+1)-(\wti_i(b)+3)-s_i \\
&= (\beta_i(b)-\wti_i(b)-s_i-1)-1 \\
&= \beta_i(b)-1.
\end{split} \nonumber
\end{align}
Thus, the fourth assertion follows.

The remaining assertions follow from the first four assertions by interchanging the roles of $i$ and $\tau(i)$.
\end{proof}

\section{Modified action of $B_i$}\label{Section: modified action of Bi}
For each $i \in I$, let $\mathbf{U}^\imath_i$ denote the subalgebra of $\mathbf{U}^\imath$ generated by $B_i,B_{\tau(i)}, (K_iK_{\tau(i)}^{-1})^{\pm 1}$.

In this section, we shall define linear operators $\Btil_i$, $i \in I$ acting on $\Ui$-modules $M$ satisfying the following.
\begin{enumerate}
  \item[(C1)]\label{conditions on Ui modules} For each $i \in I$ with $a_{i,\tau(i)} = 2$, as a $\mathbf{U}^\imath_i$-module, the $M$ is isomorphic to a direct sum of various $V^\imath(n)$ to be defined in Subsection \ref{subsect: Btil for a:2}.
  \item[(C2)] For each $i \in I$ with $a_{i,\tau(i)} = 0$, as a $\mathbf{U}^\imath_{i}$-module, the $M$ is isomorphic to a direct sum of various $V^\imath(n)$ to be defined in Subsection \ref{subsect: Btil for a:0}.
  \item[(C3)] For each $i \in I$ with $a_{i,\tau(i)} = -1$, as a $\mathbf{U}^\imath_{i}$-module, the $M$ is isomorphic to a direct sum of various $V^\imath(n_-,n_+)$ to be defined in Subsection \ref{subsect: Btil for a:-1}.
\end{enumerate}
Also, for each $i \in I$, we give an decomposition $M = \bigoplus_{k \in \mathbb{Z}} M_{i,k}$ as a vector space, and define $\beta_i(m)$ to be $k$ for all $m \in M_{i,k}$.

\begin{defi}\normalfont\label{def: icrystal base}
Let $M$ be an $X^\imath$-weight $\Ui$-module satisfying conditions (C1)--(C3) above. We say that a pair $(\clL_M, \clB_M)$ is an $\imath$crystal base of $M$ if $\clL_M$ is a $\bbK_\infty$-lattice of $M$ and $\clB_M$ is a $\C$-basis of $\ol{\clL}_M$ satisfying the following:
\begin{itemize}
\item $\clL_M = \bigoplus_{\zeta \in X^\imath} \clL_{M,\zeta}$, where $\clL_{M,\zeta} := \clL_M \cap M_\zeta$.
\item $\clB_M = \bigsqcup_{\zeta \in X^\imath} \clB_{M,\zeta}$, where $\clB_{M,\zeta} := \clB_M \cap \ev_\infty(\clL_{M,\zeta})$; this enables us to define a map $\wti : \clB_M \rightarrow X^\imath$.
\item $\clL_M$ is stable under $\Btil_i$ for all $i \in I$; this induces a linear map $\Btil_i$ on $\ol{\clL}_M$.
\item $\clL_M = \bigoplus_{k \in \Z} (\clL_M \cap M_{i,k})$, and $\clB_M = \bigsqcup_{n \in \Z} (\clB_M \cap \ev_\infty(\clL_M \cap M_{i,k}))$ for all $i \in I$, where $M_{i,k}$ denotes the subspace of $M$ appearing in the decomposition above. This enables us to define a map $\beta_i : \mathcal{B}_M \rightarrow \mathbb{Z}$.
\item $\clB_M$ forms an $\imath$crystal with respect to the structure maps above and the Hermitian inner product making $\clB_M$ an orthonormal basis.
\end{itemize}
Also, we say that the $\clB_M$ is an $\imath$crystal basis of $M$.
\end{defi}

Since the operators $\Btil_i$ and maps $\beta_i$ are defined in terms of $\Ui_{i}$-modules, we assume, until the end of this section, that $I = \{ i,\tau(i) \}$ for some $i \in I_\tau$.

\subsection{The $a_{i,\tau(i)} = 2$ case}\label{subsect: Btil for a:2}
Suppose that $a_{i,\tau(i)} = 2$. In this case, we have $\U = U_q(\frsl_2)$, and $\Ui = \bbK[B_i]$. Hence, for each $n \in \Z$, there exists a $1$-dimensional irreducible $\Ui$-module $V^\imath(n) = \bbK v$ such that
$$
B_i v = [n]_i v.
$$
This $\Ui$-module has an $X^\imath$-weight module structure such that $V^\imath(n) = V^\imath(n)_{\ol{n}}$; note that $X^\imath = X/2X \simeq \Z/2\Z$.

Set
\[
  \mathcal{L}^\imath(n) := \mathbb{K}_\infty v, \quad b := \mathrm{ev}_\infty(v),
\]
and
\[
  \mathcal{B}^\imath(n) := \{ b \}.
\]

Let $M$ be an $X^\imath$-weight module isomorphic to a direct sum of $V^\imath(n)$'s. For each $n \in \Z$, let $M[n]$ denote the isotypic component of $M$ of type $V^\imath(n)$. Then, we have
$$
M = \bigoplus_{n \in \Z} M[n].
$$
For each $v \in M[n]$, we set
$$
\beta_i(v) := |n|, \qu \Btil_i v := \sgn(n)v.
$$
Then, one can equip an $\imath$crystal structure on $\mathcal{B}^\imath(n)$.
Note that this $\imath$crystal is the same as the one in Example \ref{examples of icrystals} \eqref{examples of icrystals 3}.

\begin{lem}\label{icrystal basis of the trivial module}
For each $n \in \Z$, the pair $(\mathcal{L}^\imath(n), \mathcal{B}^\imath(n))$ is an $\imath$crystal base of $V^\imath(n)$. In particular, the crystal base $(\mathcal{L}(0), \mathcal{B}(0))$ of the trivial $\U$-module $V(0)$ is an $\imath$crystal base such that the $\imath$crystal structure of $\mathcal{B}(0)$ is isomorphic to the $\imath$crystal $\clB^\imath(s_i)$.
\end{lem}

\begin{proof}
  The first assertion is straightforwardly verified. The second assertion follows from the fact that $V(0) \simeq V^\imath(s_i)$ as $\mathbf{U}^\imath$-modules.
\end{proof}

\begin{prop}\label{deg and Btil on tensor of crystal bases}
Let $M$ be a $\Ui$-module isomorphic to a direct sum of various $V^\imath(n)$ with an $\imath$crystal base $(\clL_M,\clB_M)$, and $N$ an integrable $\U$-module with a crystal base $(\clL_N,\clB_N)$. Then, the tensor product $M \otimes N$ is isomorphic to a direct sum of various $V^\imath(n)$. Moreover, for each $b_1 \in \clB_M$, $b_2 \in \clB_N$, the value $\beta_i(b_1 \otimes b_2)$ and the vector $\Btil_i(b_1 \otimes b_2)$ are well-defined; we have
\begin{align}
\begin{split}
&\beta_i(b_1 \otimes b_2) = \begin{cases}
\beta_i(b_1) - \wt_i(b_2) \qu & \IF \beta_i(b_1) > \vphi_i(b_2), \\
\vep_i(b_2) \qu & \IF \beta_i(b_1) \leq \vphi_i(b_2) \AND \ol{\beta_i(b_1)} = \ol{\vphi_i(b_2)}, \\
\vep_i(b_2)+1 \qu & \IF \beta_i(b_1) \leq \vphi_i(b_2) \AND \ol{\beta_i(b_1)} \neq \ol{\vphi_i(b_2)},
\end{cases} \\
&\Btil_i(b_1 \otimes b_2) = \begin{cases}
\Btil_i b_1 \otimes b_2 \qu & \IF \beta_i(b_1) > \vphi_i(b_2), \\
b_1 \otimes \Etil_i b_2 \qu & \IF \beta_i(b_1) \leq \vphi_i(b_2) \AND \ol{\beta_i(b_1)} = \ol{\vphi_i(b_2)}, \\
b_1 \otimes \Ftil_i b_2 \qu & \IF \beta_i(b_1) \leq \vphi_i(b_2) \AND \ol{\beta_i(b_1)} \neq \ol{\vphi_i(b_2)}.
\end{cases}
\end{split} \nonumber
\end{align}
\end{prop}

\begin{proof}
The assertion can be proved essentially in the same way as \cite[Proposition 5.1.4]{W21b}.
\end{proof}

\begin{cor}\label{icrystal structure of a crystal; a = 2}
Let $M$ be an integrable $\U$-module with a crystal base $(\clL_N,\clB_N)$. Then, for each $b \in \clB_M$, the value $\beta_i(b)$ and the vector $\Btil_i(b)$ are well-defined; we have
\begin{align}
\begin{split}
&\beta_i(b) = \begin{cases}
|s_i| - \wt_i(b) \qu & \IF |s_i| > \vphi_i(b), \\
\vep_i(b) \qu & \IF |s_i| \leq \vphi_i(b) \AND \ol{s_i} = \ol{\vphi_i(b)}, \\
\vep_i(b)+1 \qu & \IF |s_i| \leq \vphi_i(b) \AND \ol{s_i} \neq \ol{\vphi_i(b)},
\end{cases} \\
&\Btil_i(b) = \begin{cases}
\sgn(s_i) b \qu & \IF |s_i| > \vphi_i(b), \\
\Etil_i(b) \qu & \IF |s_i| \leq \vphi_i(b) \AND \ol{s_i} = \ol{\vphi_i(b)}, \\
\Ftil_i(b) \qu & \IF |s_i| \leq \vphi_i(b) \AND \ol{s_i} \neq \ol{\vphi_i(b)}.
\end{cases}
\end{split} \nonumber
\end{align}
\end{cor}

\begin{proof}
The assertion follows from Lemma \ref{icrystal basis of the trivial module} and Proposition \ref{deg and Btil on tensor of crystal bases} by identifying $\clB_M$ with $\clB^\imath(s_i) \otimes \clB_M$.
\end{proof}

\begin{ex}\label{icrystal structure of V(lm) for AI}\normalfont
Let $n \in \Z_{\geq 0}$. Then, the $(n+1)$-dimensional irreducible $\U$-module $V(n)$ has $\clB(n)$ in Example \ref{crystal B(n)} as its crystal basis. Let us illustrate $\Btil_i$ and $\beta_i$ on it. In the following, the $\Btil_i$ is described in the same way as the crystal graph (the label ``$i$'' is omitted since we have $I_\tau = \{i\}$):
\begin{enumerate}
\item When $n < |s_i|$.
$$
\xymatrix@R=5pt{
& b_0 \ar@(ul,ur)^-{\sgn(s_i)} & b_1 \ar@(ul,ur)^-{\sgn(s_i)} &  \cdots & b_n \ar@(ul,ur)^-{\sgn(s_i)} \\
\beta_i :& |s_i|-n & |s_i|-n+2 & \cdots & |s_i|+n
}
$$
\item When $n \geq |s_i|$ and $\ol{n} = \ol{s_i}$.
$$
\xymatrix@C=10pt@R=5pt{
&b_0 & b_1 \ar@<0.5ex>[r] & b_2 \ar@<0.5ex>[l] & \cdots & b_{n-|s_i|-1} \ar@<0.5ex>[r] & b_{n-|s_i|} \ar@<0.5ex>[l] & b_{n-|s_i|+1} \ar@(ul,ur)^-{\sgn(s_i)} & \cdots & b_n \ar@(ul,ur)^-{\sgn(s_i)} \\
\beta_i : & 0 & 2 & 2 & \cdots & n-|s_i| & n-|s_i| & n-|s_i|+2 & \cdots & |s_i|+n
}
$$
\item When $n \geq |s_i|$ and $\ol{n} \neq \ol{s_i}$.
$$
\xymatrix@C=10pt@R=5pt{
& b_0 \ar@<0.5ex>[r] & b_1 \ar@<0.5ex>[l] & \cdots & b_{n-|s_i|-1} \ar@<0.5ex>[r] & b_{n-|s_i|} \ar@<0.5ex>[l] & b_{n-|s_i|+1} \ar@(ul,ur)^-{\sgn(s_i)} & \cdots & b_n \ar@(ul,ur)^-{\sgn(s_i)} \\
\beta_i :& 1 & 1 & \cdots & n-|s_i| & n-|s_i| & n-|s_i|+2 & \cdots & |s_i|+n
}
$$
\end{enumerate}
\end{ex}

\subsection{The $a_{i,\tau(i)} = 0$ case}\label{subsect: Btil for a:0}
Suppose that $a_{i,\tau(i)} = 0$. In this case, we have $\U \simeq U_q(\frsl_2) \otimes U_q(\frsl_2)$, and there exists an algebra isomorphism $U_q(\frsl_2) \rightarrow \Ui$ which sends $E,F,K^{\pm 1}$ to $B_{\tau(i)},B_i, k_i^{\pm 1}$. Hence, for each $n \in \Z_{\geq 0}$, there exists an $(n+1)$-dimensional irreducible $\Ui$-module $V^\imath(n) = \bigoplus_{k=0}^n \bbK v_k$ such that
$$
B_{\tau(i)} v_0 = 0, \qu B_i^{(k)} v_0 = v_k, \qu k_i v_0 = q_i^n v_0.
$$
Here, we understand that $v_{-1} = v_{n+1} = 0$. This $\Ui$-module has an $X^\imath$-weight module structure such that
$$
V^\imath(n) = \bigoplus_{k=0}^n V^\imath(n)_{\zeta_k}, \qu V^\imath(n)_{\zeta_k} = \bbK v_k,
$$
where $\zeta_k \in X^\imath$ is such that $\la h_i-h_{\tau(i)}, \zeta_k \ra = n-2k$.

Set
\[
  \mathcal{L}^\imath(n) := \bigoplus_{k=0}^n \mathbb{K}_\infty v_k, \quad b_k := \mathrm{ev}_\infty(v_k),
\]
and
\[
  \mathcal{B}^\imath(n) := \{ b_k \mid 0 \leq k \leq n \}.
\]

Let $M$ be an $X^\imath$-weight module isomorphic to a direct sum of $V^\imath(n)$'s. For each $n \in \Z_{\geq 0}$, let $M[n]$ denote the isotypic component of $M$ of type $V^\imath(n)$. For each $n \in \Z_{\geq 0}$ and $0 \leq k \leq n$, set
$$
M[n;k] := B_i^{k}(M[n] \cap \Ker B_{\tau(i)}).
$$
Then, we have
$$
M = \bigoplus_{0 \leq k \leq n} M[n;k].
$$
For each $v \in M[n,k]$, we set
$$
\beta_i(v) := n-k, \qu \beta_{\tau(i)}(v) := k, \qu \Btil_i v := \frac{1}{[k+1]_i} B_i v, \qu \Btil_{\tau(i)} v := \frac{1}{[n-k+1]_i} B_{\tau(i)} v.
$$
Note that $M[n;k] \subset M_{\zeta_k}$, where $\zeta_k \in X^\imath$ is as before, and $\Btil_i,\Btil_{\tau(i)}$ define a $\bbK$-linear endomorphism on $M$.

Then, one can equip an $\imath$crystal structure on $\mathcal{B}^\imath(n)$.
Note that this $\imath$crystal is the same as the one in Example \ref{examples of icrystals} \eqref{examples of icrystals 5}.

\begin{rem}\normalfont
If we interchange $i$ and $\tau(i)$, then $V^\imath(n)$ becomes $V^\imath(n)$ with $v_k$ being replaced by $v_{n-k}$. Hence, our definition of $\beta_j,\Btil_j$ for $j \in \{ i,\tau(i) \}$ is independent of the choice of $I_\tau$.
\end{rem}

\begin{lem}\label{icrystal basis of the trivial module a=0}
  For each $n \geq 0$, the pair $(\mathcal{L}^\imath(n), \mathcal{B}^\imath(n))$ is an $\imath$crystal base of $V^\imath(n)$. In particular, the crystal base $(\mathcal{L}(0), \mathcal{B}(0))$ of the trivial $\U$-module $V(0)$ is an $\imath$crystal base such that the $\imath$crystal structure of $\mathcal{B}(0)$ is isomorphic to the $\imath$crystal $\clB^\imath(0)$.
\end{lem}

\begin{proof}
  The first assertion is straightforwardly verified. The second assertion follows from the easily verified fact that $V(0) \simeq V^\imath(0)$ as $\mathbf{U}^\imath$-modules.
\end{proof}

\begin{prop}\label{Btil on tensor of crystal bases for 0}
Let $M$ be a $\Ui$-module isomorphic to a direct sum of various $V^\imath(n)$ with an $\imath$crystal base $(\clL_M,\clB_M)$, and $N$ an integrable $\U$-module with a crystal base $(\clL_N,\clB_N)$. Then, the tensor product $M \otimes N$ is isomorphic to a direct sum of various $V^\imath(n)$. Moreover, for each $b_1 \in \clB_M$, $b_2 \in \clB_N$ and $j \in \{ i,\tau(i) \}$, the value $\beta_j(b_1 \otimes b_2)$ and the vector $\Btil_j(b_1 \otimes b_2)$ are well-defined; we have
\begin{align}
\begin{split}
\beta_j(b_1 \otimes b_2) &= \max(\vphi_j(b_2)+\wti_j(b_1)-\wt_{\tau(j)}(b_2), \beta_j(b_1)-\wt_{\tau(j)}(b_2), \vep_{\tau(j)}(b_2)), \\
\Btil_j(b_1 \otimes b_2) &= \begin{cases}
b_1 \otimes \Ftil_j b_2 \qu & \IF \vphi_j(b_2) > \beta_{\tau(j)}(b_1), \vphi_{\tau(j)}(b_2)-\wti_j(b_1), \\
\Btil_j b_1 \otimes b_2 \qu & \IF \vphi_j(b_2) \leq \beta_{\tau(j)}(b_1) > \vphi_{\tau(j)}(b_2)-\wti_j(b_1), \\
b_1 \otimes \Etil_{\tau(j)} b_2 \qu & \IF \vphi_j(b_2), \beta_{\tau(j)}(b_1) \leq \vphi_{\tau(j)}(b_2)-\wti_j(b_1).
\end{cases}
\end{split} \nonumber
\end{align}
\end{prop}

\begin{proof}
The assertion can be proved essentially in the same way as \cite[Theorem 6.3.5]{W17}.
\end{proof}

\begin{cor}\label{icrystal structure of a crystal; a = 0}
Let $M$ be an integrable $\U$-module with a crystal base $(\clL_M,\clB_M)$. Then, for each $b \in \clB_M$ and $j \in \{ i,\tau(i) \}$, the value $\beta_j(b_1 \otimes b_2)$ and the vector $\Btil_j(b_1 \otimes b_2)$ are well-defined; we have
\begin{align}
\begin{split}
&\beta_j(b) = \max(\vphi_j(b)-\wt_{\tau(j)}(b),\vep_{\tau(j)}(b)), \\
&\Btil_j b = \begin{cases}
\Ftil_j b \qu & \IF \vphi_j(b) > \vphi_{\tau(j)}(b), \\
\Etil_{\tau(j)}b \qu & \IF \vphi_j(b) \leq \vphi_{\tau(j)}(b).
\end{cases}
\end{split} \nonumber
\end{align}
\end{cor}

\begin{proof}
The assertion follows from Lemma \ref{icrystal basis of the trivial module a=0} and Proposition \ref{Btil on tensor of crystal bases for 0} by identifying $\clB_M$ with $\clB^\imath(0) \otimes \clB_M$.
\end{proof}

\begin{ex}\normalfont
Let $m,n \in \Z_{\geq 0}$, and consider the $(m+1)(n+1)$-dimensional irreducible $\U$-module $V(m,n) = \bigoplus_{\substack{0 \leq k \leq m \\ 0 \leq l \leq n}} \bbK v_{k,l}$ given by
$$
E_i v_{0,0} = E_{\tau(i)} v_{0,0} = 0, \qu F_{\tau(i)}^{(k)} F_i^{(l)} v_{0,0} = v_{k,l}, \qu K_i v_{0,0} = q_i^n v_{0,0}, \qu K_{\tau(i)} v_{0,0} = q_i^m v_{0,0}.
$$
It has a crystal basis $\clB(m,n) := \{ b_{k,l} \mid 0 \leq k \leq m,\ 0 \leq l \leq n \}$ given by
\begin{align}
\begin{split}
&\wt_i(b_{k,l}) = n-2l, \ \vep_i(b_{k,l}) = l, \ \vphi_i(b_{k,l}) = n-l, \ \Etil_i b_{k,l} = b_{k,l-1}, \ \Ftil_i b_{k,l} = b_{k,l+1}, \\
&\wt_{\tau(i)}(b_{k,l}) = m-2k, \ \vep_{\tau(i)}(b_{k,l}) = k, \ \vphi_{\tau(i)}(b_{k,l}) = m-k, \ \Etil_{\tau(i)} b_{k,l} = b_{k-1,l}, \ \Ftil_{\tau(i)} b_{k,l} = b_{k+1,l}.
\end{split} \nonumber
\end{align}
Hence, we have
$$
\Btil_i b_{k,l} = \begin{cases}
b_{k,l+1} & \IF n-l > m-k, \\
b_{k-1,l} & \IF n-l \leq m-k.
\end{cases}
$$
For example, when $(m,n) = (2,3)$, the $\Btil_i$ on $\clB(m,n)$ is described as follows (the label ``$i$'' is omitted since we have $I_\tau = \{i\}$):
$$
\xymatrix@R=8pt{
b_{2,0} \ar[r] & b_{2,1} \ar[r] & b_{2,2} \ar[r] & b_{2,3} \ar[d] \\
b_{1,0} \ar[r] & b_{1,1} \ar[r] & b_{1,2} \ar[d] & b_{1,3} \ar[d] \\
b_{0,0} \ar[r] & b_{0,1} & b_{0,2} & b_{0,3}
}
$$
From above, we see that the $\Btil_i$ (resp., $\Btil_{\tau(i)}$) on a crystal basis of an integrable $\U$-module coincides with the $\Ftil_{\tau(i)'}$ (resp., $\Etil_{\tau(i)'}$) on the tensor product of crystal bases of two integrable $\U_{I'}$-modules (see Examples \ref{tensor rule example}, \ref{set up for diagonal type}, and \ref{icrystal of diagonal type is crystal}).
\end{ex}

\subsection{The $a_{i,\tau(i)} = -1$ case}\label{subsect: Btil for a:-1}
Suppose that $a_{i,\tau(i)} = -1$. In this case, we have $\U \simeq U_q(\frsl_3)$, and
\begin{align}
\begin{split}
&k_i B_i = q_i^{-3} B_i k_i, \qu k_i B_{\tau(i)} = q_i^3 B_{\tau(i)} k_i, \\
&B_i^2B_{\tau(i)} - [2]_i B_iB_{\tau(i)}B_i + B_{\tau(i)}B_i^2 = -[2]_i B_i \{ k_i;-1-s_i \}_i, \\
&B_{\tau(i)}^2B_i - [2]_i B_{\tau(i)}B_iB_{\tau(i)} + B_iB_{\tau(i)}^2 = -[2]_i \{ k_i;-1-s_i \}_i B_{\tau(i)},
\end{split} \nonumber
\end{align}
where
$$
\{ k_i;a \}_i := q_i^ak_i + q_i^{-a}k_i\inv
$$
for each $a \in \Z$.

Set
$$
t := B_{\tau(i)}B_i - q_iB_iB_{\tau(i)} - [k_i;-s_i]_i,
$$
where
$$
[k_i;a]_i := \frac{q_i^a k_i - q_i^{-a}k_i\inv}{q_i-q_i\inv}.
$$
Then, for each $k \in \Z_{\geq 0}$, we have
\begin{align}\label{Btaui Bik}
\begin{split}
B_{\tau(i)} B_i^{(k)} = B_i^{(k-1)}(t + [k_i;-s_i-2(k-1)]_i) + q_i^k B_i^{(k)} B_{\tau(i)}.
\end{split}
\end{align}

For each $n_- \in \Z_{\geq 0}$ and $n_+ \in \Z$, let $V^\imath(n_-,n_+) = \bigoplus_{k=0}^{n_-} \bbK v_n$ be a $\Ui$-module given by
$$
B_{\tau(i)} v_0 = 0, \qu B_i^{(k)} v_0 = v_k, \qu k_i v_0 = q_i^{n_-+n_+} v_0, \qu t v_0 = [n_--n_++s_i]_i v_0,
$$
where, $v_{-1}=v_{n_-+1} = 0$. Then, $V^\imath(n_-,n_+)$ is an irreducible (see \cite[Theorem 4.4.7]{W21a}) $X^\imath$-weight module such that
$$
V^\imath(n_-,n_+) = \bigoplus_{k=0}^{n_-} V^\imath(n_-,n_+)_{\zeta_k}, \qu V^\imath(n_-,n_+)_{\zeta_k} = \bbK v_k,
$$
where $\zeta_k \in X^\imath$ is such that $\la h_i-h_{\tau(i)}, \zeta_k \ra = n_-+n_+-3k$. Furthermore, $V^\imath(n_-,n_+)$ admits a contragredient Hermitian inner product $(\cdot,\cdot)$ such that $(v_0,v_0) = 1$.

\begin{rem}\label{effect of changing i taui}\normalfont
Note that we have
$$
B_i v_{n_-} = 0, \qu B_{\tau(i)}^{(k)} v_{n_-}  \in \bbK^\times v_{n_--k}, \qu k_{\tau(i)} v_{n_-} = q_i^{2n_--n_+}, \qu t' v_{n_-} = [n_++s_{\tau(i)}] v_{n_-},
$$
where
$$
t' := B_iB_{\tau(i)} - q_iB_{\tau(i)}B_i - [k_{\tau(i)};-s_{\tau(i)}]_i.
$$
This shows that if we exchange the roles of $i$ and $\tau(i)$, then $n_+,k$ are replaced by $n_--n_+,n_--k$, respectively.
\end{rem}

\begin{lem}
For each $0 \leq k \leq n_-$, we have
$$
(v_k,v_k) = q_i^{k(-n_--n_++ \frac{3}{2}k +s_i - \frac{1}{2})} { n_- \brack k }_i \prod_{l=1}^k \{ n_+-s_i-l+1 \}_i \in \Z[q,q\inv],
$$
where
$$
\{ a \}_i := q_i^a + q_i^{-a}
$$
for each $a \in \Z$. Consequently, if we write $(v_k,v_k) = \lt(v_k)^2 + \text{lower terms}$ for some $\lt(v_k) \in \R_{> 0} q^{\Z}$, then we obtain
\begin{align}
\begin{split}
\lt(v_k) = \begin{cases}
1 & \IF n_+-s_i \geq n_- \OR, \\
 & \ -1 < n_+-s_i < n_- \AND k < n_+-s_i+1, \\
\sqrt{2} q_i^{\frac{1}{2}(k-n_++s_i-1)(k-n_++s_i)} & \IF -1 < n_+-s_i < n_- \AND k \geq n_+-s_i+1, \\
q_i^{\frac{1}{2}k(k-2n_++2s_i-1)} & \IF n_+-s_i \leq -1.
 \end{cases}
\end{split} \nonumber
\end{align}
\end{lem}

\begin{proof}
Let us see that
$$
\wp^*(B_i) = q_i\inv E_iK_i\inv + q_i^{s_i} K_i\inv q_i\inv F_{\tau(i)} K_{\tau(i)} = q_i^{s_i-2} B_{\tau(i)} k_i\inv.
$$
Also, by identity \eqref{Btaui Bik}, we have
\begin{align}
\begin{split}
B_{\tau(i)} v_k &= B_{\tau(i)} B_i^{(k)} v_0 \\
&= ([n_--n_++s_i]_i + [n_-+n_+-s_i-2k+2]_i) v_{k-1} \\
&= [n_--k+1]_i\{ n_+-s_i-k+1 \}_i v_{k-1}.
\end{split} \nonumber
\end{align}
Now, we compute as
\begin{align}
\begin{split}
(v_k,v_k) &= \frac{1}{[k]_i} (v_{k-1}, q_i^{s_i-2} B_{\tau(i)} k_i\inv v_k) \\
&= \frac{q_i^{-n_--n_++3k+s_i-2}[n_--k+1]_i\{ n_+-s_i-k+1 \}_i }{[k]_i} (v_{k-1}, v_{k-1}) \\
&= \prod_{l=1}^k \frac{q_i^{-n_--n_++3l+s_i-2}[n_--l+1]_i\{ n_+-s_i-l+1 \}_i }{[l]_i} (v_0,v_0) \\
&= q_i^{k(-n_--n_++\frac{3}{2}k+s_i-\frac{1}{2})} { n_- \brack k }_i \prod_{l=1}^k \{ n_+-s_i-l+1 \}_i.
\end{split} \nonumber
\end{align}
This proves the first half of the assertion. The remaining assertion follows from this identity by noting that the leading term of $\{ a \}_i$ is $2^{\delta_{a,0}} q_i^{|a|}$.
\end{proof}

For each $k$, set
$$
\vtil_k := \lt(v_k)\inv v_k.
$$
Such $\vtil_k$ is characterized up to a multiple of $1+q\inv \bbK_\infty$ by the conditions that $\vtil_k \in \bbK^\times v_k$ and $(\vtil_k,\vtil_k) \in 1 + q\inv \bbK_\infty$. Now, we define linear operators $\Btil_i,\Btil_{\tau(i)}$ by
\begin{align}
\begin{split}
&\Btil_i \vtil_k = \begin{cases}
\vtil_{k+1} & \IF 0 \leq k < n_-, \\
0 & \IF k = n_-,
\end{cases} \\
&\Btil_{\tau(i)} \vtil_k = \begin{cases}
\vtil_k & \IF 0 < k \leq n_-, \\
0 & \IF k = 0.
\end{cases}
\end{split} \nonumber
\end{align}
Note that $\Btil_i,\Btil_{\tau(i)}$ are independent of the choice of $I_\tau$ up to $1+q\inv \bbK_\infty$ (see Remark \ref{effect of changing i taui}). In particular, they are actually independent at $q = \infty$.

Set
\[
  \mathcal{L}^\imath(n_-,n_+) := \bigoplus_{k=0}^{n_-} \mathbb{K}_\infty \vtil_k, \quad b_k := \mathrm{ev}_\infty(\vtil_k)
\]
and
\[
  \mathcal{B}^\imath(n_-,n_+) := \{ b_k \mid 0 \leq k \leq n_- \}.
\]

Let $M$ be an $X^\imath$-weight module isomorphic to a direct sum of $V^\imath(n_-,n_+)$'s. The linear operators $\Btil_i,\Btil_{\tau(i)}$ on $V^\imath(n_-,n_+)$'s can be extended to $M$. For each $n_-,n_+$, let $M[n_-,n_+]$ denote the isotypic component of $M$ of type $V^\imath(n_-,n_+)$. For each $0 \leq k \leq n_-$, set
$$
M[n_-,n_+;k] := B_i^{k}(M[n_-,n_+] \cap \Ker B_{\tau(i)}).
$$
Then, we have
$$
M = \bigoplus_{\substack{0 \leq k \leq n_- \\ n_+ \in \Z}} M[n_-,n_+;k].
$$
For each $v \in M[n_-,n_+;k]$, we set
$$
\beta_i(v) = n_--k+\max(n_+-s_i-k,0), \qu \beta_{\tau(i)}(v) = k+\max(-n_+-s_{\tau(i)}+k,0).
$$
Again, note that $\beta_i,\beta_{\tau(i)}$ are independent of the choice of $I_\tau$ (see Remark \ref{effect of changing i taui}).

Then, one can equip an $\imath$crystal structure on $\mathcal{B}^\imath(n_-,n_+)$.
Note that this $\imath$crystal is the same as the one in Example \ref{examples of icrystals} \eqref{examples of icrystals 6}.

\begin{lem}\label{icrystal basis of the trivial module a=-1}
For each $n_- \in \Z_{\geq 0}$ and $n_+ \in \Z$, the pair $(\mathcal{L}^\imath(n_-,n_+), \mathcal{B}^\imath(n_-,n_+))$ is an $\imath$crystal base of $V^\imath(n_-,n_+)$. In particular, the crystal base $(\mathcal{L}(0), \mathcal{B}(0))$ of the trivial $\U$-module $V(0)$ is an $\imath$crystal base such that the $\imath$crystal structure of $\mathcal{B}(0)$ is isomorphic to the $\imath$crystal $\clB^\imath(0,0)$.
\end{lem}

\begin{proof}
  The first assertion is straightforwardly verified. The second assertion follows from the easily verified fact that $V(0) \simeq V^\imath(0,0)$ as $\mathbf{U}^\imath$-modules.
\end{proof}

Below, we aim to describe the tensor product rule for $\Btil_i$. Recall that $\U \simeq U_q(\frsl_3)$. Let us take this isomorphism in a way such that the natural representation $V_\natural = \bbK u_{-1} \oplus \bbK u_0 \oplus \bbK u_1$ of $U_q(\frsl_3)$ has the following $\U$-module structure:
$$
E_i u_{-1} = E_{\tau(i)} u_{-1} = 0,\ K_i u_{-1} = q_i u_{-1},\ K_{\tau(i)} u_{-1} = u_{-1},\ F_i u_{-1} = u_0,\ F_{\tau(i)}u_0 = u_1.
$$
Let $(\cdot,\cdot)$ denote the contragredient Hermitian inner product on the $\mathbf{U}$-module $V_\natural$ such that $(u_{-1},u_{-1}) = 1$. Also, set $\clL_\natural := \clL_{V_\natural}$, $\ol{\clL}_\natural := \ol{\clL}_{V_\natural}$, $\clB_\natural := \{ b_j := \ev_\infty(u_j) \mid j = -1,0,1 \}$. Then, $(\clL_\natural, \clB_\natural)$ is the crystal base of $V_\natural$.

Set $V := V^\imath(n_-,n_+) \otimes V_\natural$, $\clL := \clL_{V^\imath(n_-,n_+)} \otimes \clL_\natural$, and
$$
\clB := \clB^\imath(n_-,n_+) \otimes \clB_\natural = \{ b \otimes b' \mid (b,b') \in \clB^\imath(n_-,n_+) \times \clB_\natural \}.
$$
The following propositions describe $\Btil_i$ on $\clL$ modulo $q\inv \clL$. The proofs are based on straightforward calculation with no technical argument needed, hence we omit them.

\begin{prop}
If $n_- = 0$, then we have
$$
V \simeq V^\imath(1,n_+) \oplus V^\imath(0,n_++1),
$$
with highest weight vectors
\begin{align}
\begin{split}
&v_{+0} := v_0 \otimes u_{-1} + q_i^{-(n_+-s_i)} v_0 \otimes u_1, \\
&v_{0+} := v_0 \otimes u_{-1} - q_i^{+(n_+-s_i)} v_0 \otimes u_1.
\end{split} \nonumber
\end{align}
Furthermore, we have
$$
B_i v_{+0} = q_i^{-(n_+-s_i)} \{ n_+-s_i \}_i v_{0} \otimes u_0
$$
By normalizing the three vectors above, we see that $(\clL, \clB)$ forms an $\imath$crystal base of $V$ such that the crystal graph of $\clB$ is described as follows (the label ``$i$'' is omitted since we have $I_\tau = \{i\}$. The same applies to the following propositions).
\begin{enumerate}
\item When $n_+-s_i > 0$, we have $\clB \simeq \clB^\imath(1,n_+) \sqcup \clB^\imath(0,n_++1)$:
$$
\xymatrix@R=12pt{
 & b_{-1} \ar[r]^i & b_0 \ar[r]^{\tau(i)} & b_1 \\
b_0 & \bullet \ar[r] & \bullet & \bullet
}
$$
\item When $n_+-s_i = 0$, we have $\clB \simeq \clB^\imath(1,n_+;\vee)$:
$$
\xymatrix@R=12pt{
 & b_{-1} \ar[r]^i & b_0 \ar[r]^{\tau(i)} & b_1 \\
b_0 & \bullet \ar[r]^-{\frac{1}{\sqrt{2}}} & \bullet & \bullet \ar[l]_-{\frac{1}{\sqrt{2}}}
}
$$
\item When $n_+-s_i < 0$, we have $\clB \simeq \clB^\imath(1,n_+) \sqcup \clB^\imath(0,n_++1)$:
$$
\xymatrix@R=12pt{
 & b_{-1} \ar[r]^i & b_0 \ar[r]^{\tau(i)} & b_1 \\
b_0 & \bullet & \bullet & \bullet \ar[l]
}
$$
\end{enumerate}
\end{prop}

\begin{prop}
If $n_- > 0$, then, we have
$$
V \simeq V^\imath(n_-+1,n_+) \oplus V^\imath(n_-,n_++1) \oplus V^\imath(n_--1,n_+-1),
$$
with highest weight vectors
\begin{align}
\begin{split}
&v_{+0} := v_0 \otimes u_{-1} + q_i^{n_--(n_+-s_i)} v_0 \otimes u_1, \\
&v_{0+} := v_0 \otimes u_{-1} - q_i^{-n_-+(n_+-s_i)} v_0 \otimes u_1, \\
&v_{--} := v_1 \otimes u_{-1} + q_i^{-n_--(n_+-s_i)-1} v_1 \otimes u_1 - q_i^{-n_--(n_+-s_i)}[n_-]_i\{ n_+-s_i \}_i v_0 \otimes u_0.
\end{split} \nonumber
\end{align}
Furthermore, for each $k \geq 0$, we have
\begin{align}
\begin{split}
B_i^{(k)} v_{+0} &= q_i^{-k} v_k \otimes u_{-1} + q_i^{n_--(n_+-s_i)} v_k \otimes u_1 + q_i^{-(n_+-s_i)+k-1} \{ (n_+-s_i)-k+1 \}_i v_{k-1} \otimes u_0, \\
B_i^{(k)} v_{0+} &= q_i^{-k} v_k \otimes u_{-1} - q_i^{-n_-+(n_+-s_i)} v_k \otimes u_1 + (1-q_i^{-2(n_--k+1)}) v_{k-1} \otimes u_0, \\
B_i^{(k)} v_{--} &= q_i^{-k}[k+1]_i v_{k+1} \otimes u_{-1} + q_i^{-n_--(n_+-s_i)-1}[k+1]_i v_{k+1} \otimes u_1 \\
&\qu- q_i^{-n_--(n_+-s_i)+k}[n_--k]_i \{ (n_+-s_i)-k \}_i v_k \otimes u_0.
\end{split} \nonumber
\end{align}
\end{prop}

\begin{prop}\label{Bi(n-,n+) otimes Bnatural 1}
Suppose that $0 < n_- < n_+-s_i$. Set
$$
\vtil_{+0} := v_{+0}, \qu \vtil_{0+} := -q_i^{n_--(n_+-s_i)} v_{0+}, \qu \vtil_{--} := v_{--}.
$$
Then, modulo $q\inv \clL$, we have
\begin{align}
\begin{split}
&\Btil_i^k \vtil_{+0} \equiv \begin{cases}
\vtil_0 \otimes u_{-1} \qu & \IF k = 0, \\
\vtil_{k-1} \otimes u_0 \qu & \IF 1 \leq k \leq n_-+1,
\end{cases} \\
&\Btil_i^k \vtil_{0+} \equiv \vtil_k \otimes u_1 \qu \IF 0 \leq k \leq n_-, \\
&\Btil_i^k \vtil_{--} \equiv \vtil_{k+1} \otimes u_{-1} \qu \IF 0 \leq k \leq n_--1.
\end{split} \nonumber
\end{align}
Consequently, $(\clL, \clB)$ forms an $\imath$crystal base of $V$ such that $\clB$ is isomorphic to $\clB^\imath(n_-+1,n_+) \sqcup \clB^\imath(n_-,n_++1) \sqcup \clB^\imath(n_--1,n_+-1)$. The crystal graph of $\clB$ is described as follows:
$$
\xymatrix@R=12pt{
 & b_{-1} \ar[r]^i & b_0 \ar[r]^{\tau(i)} & b_1 \\
b_0 \ar[d] & \bullet \ar[r] & \bullet \ar[d] & \bullet \ar[d] \\
b_1 \ar[d] & \bullet \ar[d] & \bullet \ar[d] & \bullet \ar[d] \\
\vdots \ar[d] & \vdots \ar[d] & \vdots \ar[d] & \vdots \ar[d] \\
b_{n_--1} \ar[d] & \bullet \ar[d] & \bullet \ar[d] & \bullet \ar[d] \\
b_{n_-} & \bullet  & \bullet  & \bullet
}
$$
\end{prop}

\begin{prop}\label{Bi(n-,n+) otimes Bnatural 2}
Suppose that $0 < n_- = n_+-s_i$. Set
$$
\vtil_{+0} := \frac{1}{\sqrt{2}}v_{+0}, \qu \vtil_{0+} := \frac{1}{\sqrt{2}} v_{0+}, \qu \vtil_{--} := v_{--}.
$$
Then, modulo $q\inv \clL$, we have
\begin{align}
\begin{split}
&\Btil_i^k \vtil_{+0} \equiv \begin{cases}
\frac{1}{\sqrt{2}}(\vtil_0 \otimes u_{-1}+\vtil_0 \otimes u_1) \qu & \IF k = 0, \\
\frac{1}{\sqrt{2}}(\vtil_{k-1} \otimes u_0 + \vtil_k \otimes u_1) \qu & \IF 1 \leq k \leq n_-, \\
\vtil_{n_-} \otimes u_0 \qu & \IF k = n_-+1,
\end{cases} \\
&\Btil_i^k \vtil_{0+} \equiv \begin{cases}
\frac{1}{\sqrt{2}}(\vtil_0 \otimes u_{-1} - \vtil_0 \otimes u_1) \qu & \IF k = 0, \\
\frac{1}{\sqrt{2}}(\vtil_{k-1} \otimes u_0 - \vtil_k \otimes u_1) \qu & \IF 1 \leq k \leq n_-,
\end{cases} \\
&\Btil_i^k \vtil_{--} \equiv \vtil_{k+1} \otimes u_{-1} \qu \IF 0 \leq k \leq n_--1,
\end{split} \nonumber
\end{align}
Consequently, $(\clL, \clB)$ forms an $\imath$crystal basis of $V$ such that $\clB$ is isomorphic to $\clB^\imath(n_-+1,n_+;\vee) \sqcup \clB^\imath(n_--1,n_+-1)$. The crystal graph of $\clB$ is described as follows:
$$
\xymatrix@R=12pt{
 & b_{-1} \ar[r]^i & b_0 \ar[r]^{\tau(i)} & b_1 \\
b_0 \ar[d] & \bullet \ar[r] & \bullet \ar[d] & \bullet \ar[d] \\
b_1 \ar[d] & \bullet \ar[d] & \bullet \ar[d] & \bullet \ar[d] \\
\vdots \ar[d] & \vdots \ar[d] & \vdots \ar[d] & \vdots \ar[d] \\
b_{n_--1} \ar[d] & \bullet \ar[d] & \bullet \ar[d]_-{\frac{1}{\sqrt{2}}} & \bullet \ar[d] \\
b_{n_-} & \bullet  & \bullet  & \bullet \ar[l]^-{\frac{1}{\sqrt{2}}}
}
$$
\end{prop}

\begin{prop}\label{Bi(n-,n+) otimes Bnatural 3}
Suppose that $-1 < n_+-s_i < n_-$. Set
$$
\vtil_{+0} := q_i^{-n_-+(n_+-s_i)}v_{+0}, \qu \vtil_{0+} := v_{0+}, \qu \vtil_{--} := v_{--}.
$$
Then, modulo $q\inv \clL$, we have
\begin{align}
\begin{split}
&\Btil_i^k \vtil_{+0} \equiv \begin{cases}
\vtil_k \otimes u_{1} \qu & \IF 0 \leq k \leq n_-, \\
\vtil_{n_-} \otimes u_0 \qu & \IF k = n_-+1,
\end{cases} \\
&\Btil_i^k \vtil_{0+} \equiv \begin{cases}
\vtil_0 \otimes u_{-1} \qu & \IF k = 0, \\
\vtil_{k-1} \otimes u_0 \qu & \IF 1 \leq k \leq n_-,
\end{cases} \\
&\Btil_i^k \vtil_{--} \equiv \vtil_{k+1} \otimes u_{-1} \qu \IF 0 \leq k \leq n_--1.
\end{split} \nonumber
\end{align}
Consequently, $(\clL, \clB)$ forms an $\imath$crystal basis of $V$ such that $\clB$ is isomorphic to $\clB^\imath(n_-+1,n_+) \sqcup \clB^\imath(n_-,n_++1) \sqcup \clB^\imath(n_--1,n_+-1)$. The crystal graph of $\clB$ is described as follows:
$$
\xymatrix@R=12pt{
 & b_{-1} \ar[r]^i & b_0 \ar[r]^{\tau(i)} & b_1 \\
b_0 \ar[d] & \bullet \ar[r] & \bullet \ar[d] & \bullet \ar[d] \\
b_1 \ar[d] & \bullet \ar[d] & \bullet \ar[d] & \bullet \ar[d] \\
\vdots \ar[d] & \vdots \ar[d] & \vdots \ar[d] & \vdots \ar[d] \\
b_{n_--1} \ar[d] & \bullet \ar[d] & \bullet & \bullet \ar[d] \\
b_{n_-} & \bullet  & \bullet  & \bullet \ar[l]
}
$$
\end{prop}

\begin{prop}\label{Bi(n-,n+) otimes Bnatural 4}
Suppose that $n_- > 0$ and $n_+-s_i = -1$. Set
$$
\vtil_{+0} := q_i^{-n_--1} v_{+0}, \qu \vtil_{0+} := v_{0+}, \qu \vtil_{--} := v_{--}.
$$
Then, modulo $q\inv \clL$, we have
\begin{align}
\begin{split}
&\Btil_i^k \vtil_{+0} \equiv \begin{cases}
\vtil_k \otimes u_{1} \qu & \IF 0 \leq k \leq n_-, \\
\vtil_{n_-} \otimes u_0 \qu & \IF k = n_-+1,
\end{cases} \\
&\Btil_i^k \vtil_{0+} \equiv \begin{cases}
\vtil_0 \otimes u_{-1} \qu & \IF k = 0, \\
\frac{1}{\sqrt{2}}(\vtil_k \otimes u_{-1} + \vtil_{k-1} \otimes u_0) \qu & \IF 1 \leq k \leq n_-,
\end{cases} \\
&\Btil_i^k \vtil_{--} \equiv \frac{1}{\sqrt{2}}(\vtil_{k+1} \otimes u_{-1} - \vtil_k \otimes u_0) \qu \IF 0 \leq k \leq n_--1.
\end{split} \nonumber
\end{align}
Consequently, $(\clL, \clB)$ forms an $\imath$crystal basis of $V$ such that $\clB$ is isomorphic to $\clB^\imath(n_-+1,n_+) \sqcup \clB^\imath(n_-,n_++1;\wedge)$. The crystal graph of $\clB$ is described as follows:
$$
\xymatrix@R=12pt{
 & b_{-1} \ar[r]^i & b_0 \ar[r]^{\tau(i)} & b_1 \\
b_0 \ar[d] & \bullet \ar[r]^-{\frac{1}{\sqrt{2}}} \ar[d]_-{\frac{1}{\sqrt{2}}} & \bullet \ar[d] & \bullet \ar[d] \\
b_1 \ar[d] & \bullet \ar[d] & \bullet \ar[d] & \bullet \ar[d] \\
\vdots \ar[d] & \vdots \ar[d] & \vdots \ar[d] & \vdots \ar[d] \\
b_{n_--1} \ar[d] & \bullet \ar[d] & \bullet & \bullet \ar[d] \\
b_{n_-} & \bullet  & \bullet  & \bullet \ar[l]
}
$$
\end{prop}

\begin{prop}\label{Bi(n-,n+) otimes Bnatural 5}
Suppose that $n_- > 0$ and  $n_+-s_i < -1$. Set
$$
\vtil_{+0} := q_i^{-n_-+(n_+-s_i)}v_{+0}, \qu \vtil_{0+} := v_{0+}, \qu \vtil_{--} := -v_{--}.
$$
Then, modulo $q\inv \clL$, we have
\begin{align}
\begin{split}
&\Btil_i^k \vtil_{+0} \equiv \begin{cases}
\vtil_k \otimes u_{1} \qu & \IF 0 \leq k \leq n_-, \\
\vtil_{n_-} \otimes u_0 \qu & \IF k = n_-+1,
\end{cases} \\
&\Btil_i^k \vtil_{0+} \equiv \vtil_{k} \otimes u_{-1} \qu \IF 0 \leq k \leq n_-, \\
&\Btil_i^k \vtil_{--} \equiv \vtil_{k} \otimes u_{0} \qu \IF 0 \leq k \leq n_--1.
\end{split} \nonumber
\end{align}
Consequently, $(\clL, \clB)$ forms an $\imath$crystal basis of $V$ such that $\mathcal{B}$ is isomorphic to $\clB^\imath(n_-+1,n_+) \sqcup \clB^\imath(n_-,n_++1) \sqcup \clB^\imath(n_--1,n_+-1)$. The crystal graph of $\mathcal{B}$ is described as follows:
$$
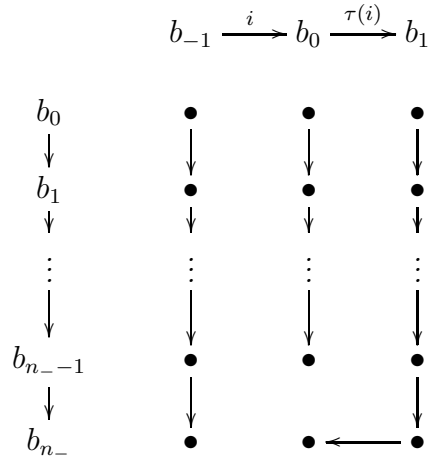
\xymatrix@R=12pt{
 & b_{-1} \ar[r]^i & b_0 \ar[r]^{\tau(i)} & b_1 \\
b_0 \ar[d] & \bullet \ar[d] & \bullet \ar[d] & \bullet \ar[d] \\
b_1 \ar[d] & \bullet \ar[d] & \bullet \ar[d] & \bullet \ar[d] \\
\vdots \ar[d] & \vdots \ar[d] & \vdots \ar[d] & \vdots \ar[d] \\
b_{n_--1} \ar[d] & \bullet \ar[d] & \bullet & \bullet \ar[d] \\
b_{n_-} & \bullet  & \bullet  & \bullet \ar[l]
}
$$
\end{prop}

Since $\beta_i(b), \beta_{\tau(i)}(b)$ for $b \in \clB^\imath(n_-,n_+)$, $\clB^\imath(n_-,n_+;\vee)$, $\clB^\imath(n_-,n_+;\wedge)$ can be read from the crystal graphs, Propositions \ref{Bi(n-,n+) otimes Bnatural 1} -- \ref{Bi(n-,n+) otimes Bnatural 5} can be reformulated as follows:

\begin{prop}\label{tensor rule in terms of icrystal data}
Let $M$ be a $\Ui$-module isomorphic to a direct sum of various $V^\imath(n_-,n_+)$ with an $\imath$crystal base $(\clL_M, \clB_M)$. Then, $M \otimes V_\natural$ is isomorphic to a direct sum of various $V^\imath(n_-,n_+)$. Moreover, for each $b \in \clB_M$, $j \in \{ i,\tau(i) \}$, and $k \in \{ -1, 0, 1 \}$, the value $\beta_j(b \otimes b_k)$ and the vector $\Btil_j(b \otimes b_k)$ are well-defined; if we set $\wti_j := \wti_j(b)$ and $\beta_j := \beta_j(b)$, then we have the following:
\begin{align}
\begin{split}
&\beta_i(b \otimes b_{-1}) = \begin{cases}
\beta_i+1 & \IF \beta_{\tau(i)} = 0 \AND \beta_i = \beta_{\tau(i)}+\wti_i-s_i, \\
\beta_i & \OW, 
\end{cases} \\
&\beta_{\tau(i)}(b \otimes b_{-1}) = \begin{cases}
0 & \IF \beta_{\tau(i)} = 0, \\
\beta_{\tau(i)}-1 & \IF \beta_{\tau(i)} > 0,
\end{cases} \\
&\Btil_i(b \otimes b_{-1}) = \begin{cases}
b \otimes b_0 & \IF \beta_{\tau(i)} = 0 < \beta_i = \beta_{\tau(i)}+\wti_i-s_i, \\
\frac{1}{\sqrt{2}} b \otimes b_0 & \IF \beta_{\tau(i)} = 0 = \beta_i = \beta_{\tau(i)}+\wti_i-s_i, \\
\frac{1}{\sqrt{2}}(\Btil_i b \otimes b_{-1} + b \otimes b_0) & \IF \beta_{\tau(i)} = 0 < \beta_i \neq \beta_{\tau(i)}+\wti_i-s_i, \\
\Btil_i b \otimes b_{-1} & \OW,
\end{cases} \\
&\Btil_{\tau(i)}(b \otimes b_{-1}) = \begin{cases}
0 & \IF \beta_{\tau(i)} \leq 1, \\
\frac{1}{\sqrt{2}} \Btil_{\tau(i)} b \otimes b_{-1} & \IF \beta_{\tau(i)} = 2 \AND \beta_{\tau(i)}(\Btil_{\tau(i)} b) = 0, \\
\Btil_{\tau(i)} b \otimes b_{-1} & \OW,
\end{cases}
\end{split} \nonumber
\end{align}
\begin{align}
\begin{split}
&\beta_i(b \otimes b_0) = \begin{cases}
0 & \IF \beta_i = 0, \\
\beta_i-1 & \IF \beta_i > 0,
\end{cases} \\
&\beta_{\tau(i)}(b \otimes b_0) = \begin{cases}
\beta_{\tau(i)}+2 & \IF \beta_i = 0 \AND \beta_i \neq \beta_{\tau(i)}+\wti_i-s_i, \\
\beta_{\tau(i)}+1 & \OW,
\end{cases} \\
&\Btil_i(b \otimes b_0) = \begin{cases}
0 & \IF \beta_i \leq 1, \\
\frac{1}{\sqrt{2}} \Btil_i b \otimes b_0 & \IF \beta_i = 2 \AND \beta_i(\Btil_i b) = 0, \\
\Btil_i b \otimes b_0 & \OW,
\end{cases} \\
&\Btil_{\tau(i)}(b \otimes b_0) = \begin{cases}
\frac{1}{\sqrt{2}}(b \otimes b_{-1} + b \otimes b_1) & \IF \beta_i = 0 = \beta_{\tau(i)} \AND \beta_i=\beta_{\tau(i)}+\wti_i-s_i, \\
\frac{1}{\sqrt{2}}(\Btil_{\tau(i)}b \otimes b_{0} + b \otimes b_1) & \IF \beta_i = 0 < \beta_{\tau(i)} \AND \beta_i=\beta_{\tau(i)}+\wti_i-s_i, \\
b \otimes b_1 & \IF \beta_i = 0 \AND \beta_i \neq \beta_{\tau(i)}+\wti_i-s_i, \\
\frac{1}{\sqrt{2}} b \otimes b_{-1} & \IF \beta_{\tau(i)} = 0 < \beta_i \neq \beta_{\tau(i)}+\wti_i-s_i, \\
b \otimes b_{-1} & \IF \beta_{\tau(i)} = 0 < \beta_i=\beta_{\tau(i)}+\wti_i-s_i, \\
\Btil_{\tau(i)} b \otimes b_0 & \OW,
\end{cases}
\end{split} \nonumber
\end{align}
\begin{align}
\begin{split}
&\beta_i(b \otimes b_1) = \beta_i+1, \\
&\beta_{\tau(i)}(b \otimes b_1) = \beta_{\tau(i)}, \\
&\Btil_i(b \otimes b_1) = \begin{cases}
\frac{1}{\sqrt{2}} b \otimes b_0 & \IF \beta_i = 0 \AND \beta_i=\beta_{\tau(i)}+\wti_i-s_i, \\
b \otimes b_0 & \IF \beta_i = 0 \AND \beta_i \neq \beta_{\tau(i)}+\wti_i-s_i, \\
\Btil_i b \otimes b_1 & \OW,
\end{cases} \\
&\Btil_{\tau(i)}(b \otimes b_1) = \Btil_{\tau(i)}b \otimes b_1.
\end{split} \nonumber
\end{align}
\end{prop}

\begin{cor}\label{icrystal structure of a crystal; a = -1}
Let $M$ be an integrable $\U$-module with a crystal base $(\clL_M, \clB_M)$. Then, the values $\beta_i(b), \beta_{\tau(i)}(b)$ are defined for all $b \in \clB_M$. Also, $\Btil_i,\Btil_{\tau(i)}$ preserve $\clL_M$.
\end{cor}

\begin{proof}
Since $M$ can be embedded into a direct sum of $V_{\natural}^{\otimes N}$ for various $N \geq 0$, the assertion follows inductively from Lemma \ref{icrystal basis of the trivial module a=-1} and Proposition \ref{tensor rule in terms of icrystal data} by identifying $\clB(0) \otimes \clB_M$ with $\clB_M$.
\end{proof}

\begin{rem}\normalfont
Let $M$ be as before. The $\beta_j(b),\Btil_j b$ for $j \in \{ i,\tau(i) \}$, $b \in \clB_M$ can be inductively calculated by Proposition \ref{tensor rule in terms of icrystal data}. Explicit formulas for them will be given later.
\end{rem}

\section{Tensor product rule for $\imath$crystal}\label{Section: tensor product rule for icrystal}
In this section, we shall construct an $\imath$crystal from an $\imath$crystal $\clB_1$ and a crystal $\clB_2$ satisfying certain conditions, whose underlying set is $\clB_1 \times \clB_2$. The construction is motivated by the representation theoretic results in the previous section.

\subsection{Statements}
Given an $\imath$crystal $\clB_1$ and a crystal $\clB_2$, consider the direct product $\clB := \clB_1 \times \clB_2$, and identify $\ol{\clL} := \C\clB$ with $\ol{\clL_1} \otimes_{\C} \ol{\clL_2}$, where $\ol{\clL_i} := \C\clB_i$ with $i = 1,2$. Then, $\{ b_1 \otimes b_2 \mid b_1 \in \clB_1,\ b_2 \in \clB_2 \}$ forms an orthonormal basis of $\ol{\clL}$ with respect to the Hermitian inner product induced by those on $\ol{\clL_1}$ and $\ol{\clL_2}$ making $\clB_1$ and $\clB_2$ orthonormal bases;
$$
(b_1 \otimes b_2, b'_1 \otimes b'_2) := (b_1,b'_1) (b_2,b'_2).
$$
We identify $(b_1,b_2) \in \clB$ with $b_1 \otimes b_2$, and write $\clB = \clB_1 \otimes \clB_2$.

For $i \in I$ and $b = b_1 \otimes b_2 \in \clB$, set
\begin{align}
\begin{split}
&F_i(b) := \begin{cases}
\vphi_i(b_2) + \delta_{\ol{\beta_i(b_1)+1},\ol{\vphi_i(b_2)}} & \IF a_{i,\tau(i)} = 2, \\
\vphi_i(b_2) & \IF a_{i,\tau(i)} \neq 2,
\end{cases} \\
&B_i(b) := \begin{cases}
\beta_i(b_1) & \IF a_{i,\tau(i)} = 2, \\
\beta_i(b_1)-\wti_i(b_1)+s_i & \IF a_{i,\tau(i)} \neq 2,
\end{cases} \\
&E_i(b) := \begin{cases}
\vphi_i(b_2) & \IF a_{i,\tau(i)} = 2, \\
\vphi_{\tau(i)}(b_2)-\wti_i(b_1)+s_i & \IF a_{i,\tau(i)} \neq 2,
\end{cases} \\
\end{split} \nonumber
\end{align}
where we set $s_i = 0$ for all $i \in I$ with $a_{i,\tau(i)} = 0$. Note that we have $B_i(b) = \beta_{\tau(i)}(b_1)$ if $a_{i,\tau(i)} = 0$, and
$$
B_i(b) = \begin{cases}
\beta_{\tau(i)}(b_1) & \IF \beta_i(b_1) = \beta_{\tau(i)}(b_1)+\wti_i(b_1)-s_i, \\
\beta_{\tau(i)}(b_1)+1 & \IF \beta_i(b_1) \neq \beta_{\tau(i)}(b_1)+\wti_i(b_1)-s_i
 \end{cases}
$$
if $a_{i,\tau(i)} = -1$. Also, note that when $a_{i,\tau(i)} = 2$, we have
$$
F_i(b) = \begin{cases}
E_i(b) & \IF \ol{\beta_i(b_1)} = \ol{\vphi_i(b_2)}, \\
E_i(b)+1 & \IF \ol{\beta_i(b_1)} \neq \ol{\vphi_i(b_2)}.
\end{cases}
$$

\begin{prop}\label{tensor product of icrystal and crystal}\normalfont
Let $\clB_1$ be an $\imath$crystal, $\clB_2$ a crystal. Assume that $\clB_2$ satisfies conditions {\rm (S1)}--{\rm (S3)'} in Subsection \ref{subsection: cyrstal} for all $i,\tau(i) \in I$ with $a_{i,\tau(i)} \neq 2$. For each $b_1 \in \clB_1$, $b_2 \in \clB_2$, and $i \in I$, set $b := b_1 \otimes b_2$, $\beta_i := \beta_i(b_1)$, $\wti := \wti(b_1)$, $\wti_i := \wti_i(b_1)$, $\vep_i := \vep_i(b_2)$, $\vphi_i := \vphi_i(b_2)$, $\wt := \wt(b_2)$, $\wt_i := \wt_i(b_2)$. Then, $\clB := \clB_1 \otimes \clB_2$ is equipped with an $\imath$crystal structure as follows: Let $b_1 \in \clB_1$, $b_2 \in \clB_2$, and $i \in I$.
\begin{enumerate}
\item $\wti(b) = \wti + \ol{\wt}$.
\item If $a_{i,\tau(i)} = 2$, then
\begin{align}
\begin{split}
\beta_i(b) &= \max(F_i(b),B_i(b),E_i(b))-\wt_i \\
&= \begin{cases}
\vep_i+1 & \IF F_i(b) > B_i(b), E_i(b), \\
\beta_i-\wt_i & \IF F_i(b) \leq B_i(b) > E_i(b), \\
\vep_i & \IF F_i(b), B_i(b) \leq E_i(b),
\end{cases} \\
\Btil_i b &= \begin{cases}
b_1 \otimes \Ftil_i b_2 & \IF F_i(b) > B_i(b),E_i(b), \\
\Btil_i b_1 \otimes b_2 & \IF F_i(b) \leq B_i(b) > E_i(b), \\
b_1 \otimes \Etil_i b_2 & \IF F_i(b),B_i(b) \leq E_i(b). \\
\end{cases}
\end{split} \nonumber
\end{align}
Here, we understand that $-\infty < -\infty_\ev,-\infty_\odd < a$ for all $a \in \Z$.
\item If $a_{i,\tau(i)} = 0$, then
\begin{align}
\begin{split}
\beta_i(b) &= \max(F_i(b),B_i(b),E_i(b))+\wti_i-\wt_{\tau(i)} \\
&= \begin{cases}
\vphi_i+\wti_i-\wt_{\tau(i)} & \IF F_i(b) > B_i(b),E_i(b), \\
\beta_i-\wt_{\tau(i)} & \IF F_i(b) \leq B_i(b) > E_i(b), \\
\vep_{\tau(i)} & \IF F_i(b),B_i(b) \leq E_i(b),
\end{cases} \\
\Btil_i b &= \begin{cases}
b_1 \otimes \Ftil_i b_2 & \IF F_i(b) > B_i(b),E_i(b), \\
\Btil_i b_1 \otimes b_2 & \IF F_i(b) \leq B_i(b) > E_i(b), \\
b_1 \otimes \Etil_{\tau(i)}b_2 & \IF F_i(b),B_i(b) \leq E_i(b).
\end{cases}
\end{split} \nonumber
\end{align}
\item If $a_{i,\tau(i)} = -1$, then we have
\begin{align}
\begin{split}
\beta_i(b) &= \max(F_i(b),B_i(b),E_i(b))+\wti_i-s_i-\wt_{\tau(i)}, \\
&=\begin{cases}
\vphi_i+\wti_i-s_i-\wt_{\tau(i)} & \IF F_i(b) > B_i(b), E_i(b), \\
\beta_i-\wt_{\tau(i)} & \IF F_i(b) \leq B_i(b) > E_i(b), \\
\vep_{\tau(i)} & \IF F_i(b), B_i(b) \leq E_i(b).
\end{cases}
\end{split} \nonumber
\end{align}
Also, the following hold:
\begin{enumerate}
\item When $F_i(b) > B_i(b), E_i(b)$,
$$
\Btil_i b = \begin{cases}
\frac{1}{\sqrt{2}} b_1 \otimes \Ftil_i b_2 & \IF F_i(b)=E_i(b)+1 \AND \vphi_{\tau(i)}(\Ftil_i b_2) = \vphi_{\tau(i)}+1, \\
b_1 \otimes \Ftil_i b_2 & \OW.
\end{cases}
$$
Here and below, whenever talking about $\vphi_i(b_2)$ for $b_2 \in \clB_2 \sqcup \{0\}$, we further impose that $b_2 \in \clB_2$.
\item When $F_i(b) \leq B_i(b) > E_i(b)$,
\begin{align}
\begin{split}
&\Btil_i b = \begin{cases}
\frac{1}{\sqrt{2}} \Btil_i b_1 \otimes b_2 & \IF B_i(b) = E_i(b)+1,\AND \beta_i(\Btil_i b_1) = \beta_i-2, \\
\frac{1}{\sqrt{2}}(\Btil_i b_1 \otimes b_2 + b_1 \otimes \Ftil_i b_2) & \IF F_i(b) = B_i(b) \neq \beta_{\tau(i)}, \\
\Btil_i b_1 \otimes b_2 & \OW.
\end{cases}
\end{split} \nonumber
\end{align}
Here and below, whenever talking about $\beta_i(b_1)$ for $b_1 \in \clL_1$, we further impose that $b_1 \in \clB_1$.
\item When $F_i(b), B_i(b) \leq E_i(b)$,
\begin{align}
\begin{split}
&\Btil_i b = \begin{cases}
\frac{1}{\sqrt{2}} b_1 \otimes \Etil_{\tau(i)}b_2 & \IF E_i(b) = F_i(b) \AND \vphi_i(\Etil_{\tau(i)}b_2) = \vphi_i, \\
&\OR E_i(b) = B_i(b) = \beta_{\tau(i)} \AND \vphi_i(\Etil_{\tau(i)}b_2) < E_i(b), \\
\frac{1}{\sqrt{2}}(b_1 \otimes \Etil_{\tau(i)}b_2 + b_1 \otimes \Ftil_i b_2)& \IF E_i(b) = F_i(b) > \beta_{\tau(i)} \AND \vphi_i(\Etil_{\tau(i)} b_2) = \vphi_i-1, \\
b_1 \otimes \Etil_{\tau(i)}b_2 & \OW.
\end{cases}
\end{split} \nonumber
\end{align}
\end{enumerate}
\end{enumerate}
\end{prop}

Since the proof of this proposition is lengthy and independent of the later argument, we put it in Subsection \ref{Subsection: proof of tensor product rule}.

\begin{rem}\label{coincidence at Vnatural}\normalfont
The tensor product rule above for $I = \{ i,\tau(i) \}$ with $a_{i,\tau(i)} \neq -1$ coincides with the one given in Propositions \ref{deg and Btil on tensor of crystal bases} and \ref{Btil on tensor of crystal bases for 0}. When $a_{i,\tau(i)} = -1$ and $\clB_2 = \clB_\natural$, it coincides with the one given in Proposition \ref{tensor rule in terms of icrystal data}.
\end{rem}

\begin{prop}\label{associativity for icrystal}
The tensor product of an $\imath$crystal and a crystal is associative. Namely, let $\clB_1$ be an $\imath$crystal, and $\clB_2,\clB_3$ crystals such that $\clB_2,\clB_3,\clB_2 \otimes \clB_3$ satisfy conditions {\rm (S1)}--{\rm (S3)'} for all $i,\tau(i) \in I$ with $a_{i,\tau(i)} \neq 2$. Then, the canonical map
$$
\clB_1 \otimes (\clB_2 \otimes \clB_3) \rightarrow (\clB_1 \otimes \clB_2) \otimes \clB_3;\ b_1 \otimes (b_2 \otimes b_3) \mapsto (b_1 \otimes b_2) \otimes b_3
$$
gives rise to an isomorphism of $\imath$crystals.
\end{prop}

Since the proof of this proposition is lengthy and independent of the later argument, we put it in Subsection \ref{Subsection: proof of associativity}.

\begin{rem}\label{icrystal on tensor power}\normalfont
Suppose that $I = \{ i,\tau(i) \}$ for some $i \in I$. Let $V_\natural$ denote the natural representation of $\U$, and $\clB_\natural$ its crystal basis. By Remark \ref{coincidence at Vnatural} and Proposition \ref{associativity for icrystal}, the $\beta_i$ and $\Btil_i$ on $\clB_\natural^{\otimes N} = \clB(0) \otimes \clB_\natural^{\otimes N}$ combinatorially defined in this section coincide with the ones representation theoretically defined in the previous section.
\end{rem}

\subsection{Consequences}
In this subsection, we list some consequences of Propositions \ref{tensor product of icrystal and crystal} and \ref{associativity for icrystal}.

\begin{cor}\label{icrystal structure on crystal 1}
Let $\clB$ be a crystal satisfying conditions {\rm (S1)}--{\rm (S3)'} for all $i,\tau(i) \in I$ with $a_{i,\tau(i)} \neq 2$. For each $b \in \clB$ and $i \in I$, set $\vphi_i := \vphi_i(b)$, $\vep_i := \vep_i(b)$, $\wt := \wt(b)$, and $\wt_i := \wt_i(b)$. Then, it has an $\imath$crystal structure as follows: Let $b \in \clB$ and $i \in I$.
\begin{enumerate}
\item $\wti(b) = \ol{\wt}$.
\item If $a_{i,\tau(i)} = 2$, then
\begin{align}
\begin{split}
&\beta_i(b) = \begin{cases}
\vep_i+1 & \IF |s_i| \leq \vphi_i \AND \ol{s_i} \neq \ol{\vphi_i}, \\
|s_i|-\wt_i & \IF |s_i| > \vphi_i, \\
\vep_i & \IF |s_i| \leq \vphi_i \AND \ol{s_i} = \ol{\vphi_i},
\end{cases} \\
&\Btil_i b = \begin{cases}
\Ftil_i b & \IF |s_i| \leq \vphi_i \AND \ol{s_i} \neq \ol{\vphi_i}, \\
\sgn(s_i) b & \IF |s_i| > \vphi_i, \\
\Etil_i b & \IF |s_i| \leq \vphi_i \AND \ol{s_i} = \ol{\vphi_i},
\end{cases}
\end{split} \nonumber
\end{align}
\item If $a_{i,\tau(i)} = 0$, then
\begin{align}
\begin{split}
\beta_i(b) &= \max(\vphi_i, 0, \vphi_{\tau(i)})-\wt_{\tau(i)} \\
&= \begin{cases}
\vphi_i-\wt_{\tau(i)} & \IF \vphi_i > 0,\vphi_{\tau(i)}, \\
-\wt_{\tau(i)} & \IF \vphi_i \leq 0 > \vphi_{\tau(i)}, \\
\vep_{\tau(i)} & \IF \vphi_i,0 \leq \vphi_{\tau(i)},
\end{cases} \\
\Btil_i b &= \begin{cases}
\Ftil_i b & \IF \vphi_i > 0,\vphi_{\tau(i)}, \\
0 & \IF \vphi_i \leq 0 > \vphi_{\tau(i)}, \\
\Etil_{\tau(i)}b & \IF \vphi_i,0 \leq \vphi_{\tau(i)}.
\end{cases}
\end{split} \nonumber
\end{align}
\item If $a_{i,\tau(i)} = -1$, then
\begin{align}
\begin{split}
\beta_i(b) &= \max(\vphi_i,\max(0,s_i),\vphi_{\tau(i)}+s_i)-s_i-\wt_{\tau(i)}, \\
&= \begin{cases}
\vphi_i-s_i-\wt_{\tau(i)} & \IF \vphi_i > \max(s_i,0),\vphi_{\tau(i)}+s_i, \\
\max(-s_i,0)-\wt_{\tau(i)} & \IF \vphi_i \leq \max(s_i,0) > \vphi_{\tau(i)}+s_i, \\
\vep_{\tau(i)} & \IF \vphi_i, \max(s_i,0) \leq \vphi_{\tau(i)}+s_i.
\end{cases}
\end{split} \nonumber
\end{align}
\begin{enumerate}
\item When $\vphi_i > \max(s_i,0), \vphi_{\tau(i)}+s_i$,
$$
\Btil_i(b) = \begin{cases}
\frac{1}{\sqrt{2}} \Ftil_i b & \IF \vphi_i = \vphi_{\tau(i)}+s_i+1 \AND \vphi_{\tau(i)}(\Ftil_i b) = \vphi_{\tau(i)}+1, \\
\Ftil_i b & \OW.
\end{cases}
$$
\item When $\vphi_i \leq \max(s_i,0) > \vphi_{\tau(i)}+s_i$,
$$
\Btil_i b = \begin{cases}
\frac{1}{\sqrt{2}} \Ftil_i b & \IF \vphi_i = s_i > 0, \\
0 & \OW.
\end{cases}
$$
\item When $\vphi_i,\max(s_i,0) \leq \vphi_{\tau(i)}+s_i$,
$$
\Btil_i b = \begin{cases}
\frac{1}{\sqrt{2}} \Etil_{\tau(i)} b & \IF \vphi_i = \vphi_{\tau(i)}+s_i \AND \vphi_i(\Etil_{\tau(i)} b) = \vphi_i, \\
& \OR \vphi_{\tau(i)}+s_i = 0 \geq s_i \AND \vphi_i(\Etil_{\tau(i)} b) < 0, \\
\frac{1}{\sqrt{2}}(\Etil_{\tau(i)} b + \Ftil_i b) & \IF \vphi_i = \vphi_{\tau(i)}+s_i > \max(0,-s_{\tau(i)}) \AND \vphi_i(\Etil_{\tau(i)} b) = \vphi_i-1, \\
\Etil_{\tau(i)} b & \OW.
\end{cases}
$$
\end{enumerate}
\end{enumerate}
\end{cor}

\begin{proof}
Recall from Example \ref{examples of icrystals} \eqref{examples of icrystals 1} that the crystal $\clB(0) = \{ b_0 \}$ has an $\imath$crystal structure. By Proposition \ref{tensor product of icrystal and crystal}, the tensor product $\clB(0) \otimes \clB$ has an $\imath$crystal structure. For each $b \in \clB$, we have
\begin{align}
\begin{split}
F_i(b_0 \otimes b) &= \begin{cases}
\vphi_i+\delta_{\ol{s_i+1},\vphi_i} & \IF a_{i,\tau(i)} = 2, \\
\vphi_i & \IF a_{i,\tau(i)} \neq 2,
\end{cases} \\
B_i(b_0 \otimes b) &= \begin{cases}
|s_i| & \IF a_{i,\tau(i)} = 2, \\
\max(0,s_i) & \IF a_{i,\tau(i)} \neq 2,
\end{cases} \\
E_i(b_0 \otimes b) &= \begin{cases}
\vphi_i & \IF a_{i,\tau(i)} = 2, \\
\vphi_{\tau(i)}+s_i & \IF a_{i,\tau(i)} \neq 2.
\end{cases}
\end{split} \nonumber
\end{align}
Now, the assertion follows from Proposition \ref{tensor product of icrystal and crystal} by identifying $b_0 \otimes b$ with $b$. 
\end{proof}

\begin{cor}\label{icrystal structure on crystal}
Let $\clB$ be a seminormal crystal satisfying conditions {\rm (S1)}--{\rm (S3)'} for all $i,\tau(i) \in I$ with $a_{i,\tau(i)} \neq 2$. For each $b \in \clB$ and $i \in I$, set $\vphi_i := \vphi_i(b)$, $\vep_i := \vep_i(b)$, $\wt := \wt(b)$, and $\wt_i := \wt_i(b)$. Then, it has an $\imath$crystal structure as follows: Let $b \in \clB$ and $i \in I$.
\begin{enumerate}
\item $\wti(b) = \ol{\wt}$.
\item If $a_{i,\tau(i)} = 2$, then
\begin{align}
\begin{split}
&\beta_i(b) = \begin{cases}
\vep_i+1 & \IF |s_i| \leq \vphi_i \AND \ol{s_i} \neq \ol{\vphi_i}, \\
|s_i|-\wt_i & \IF |s_i| > \vphi_i, \\
\vep_i & \IF |s_i| \leq \vphi_i \AND \ol{s_i} = \ol{\vphi_i},
\end{cases} \\
&\Btil_i b = \begin{cases}
\Ftil_i b & \IF |s_i| \leq \vphi_i \AND \ol{s_i} \neq \ol{\vphi_i}, \\
\sgn(s_i) b & \IF |s_i| > \vphi_i, \\
\Etil_i b & \IF |s_i| \leq \vphi_i \AND \ol{s_i} = \ol{\vphi_i},
\end{cases}
\end{split} \nonumber
\end{align}
\item If $a_{i,\tau(i)} = 0$, then
\begin{align}
\begin{split}
\beta_i(b) &= \max(\vphi_i, \vphi_{\tau(i)})-\wt_{\tau(i)} \\
&= \begin{cases}
\vphi_i-\wt_{\tau(i)} & \IF \vphi_i > \vphi_{\tau(i)}, \\
\vep_{\tau(i)} & \IF \vphi_i \leq \vphi_{\tau(i)},
\end{cases} \\
\Btil_i b &= \begin{cases}
\Ftil_i b & \IF \vphi_i > \vphi_{\tau(i)}, \\
\Etil_{\tau(i)}b & \IF \vphi_i \leq \vphi_{\tau(i)}.
\end{cases}
\end{split} \nonumber
\end{align}
\item If $a_{i,\tau(i)} = -1$, then
\begin{align}
\begin{split}
\beta_i(b) &= \max(\vphi_i,\vphi_{\tau(i)}+s_i)-s_i-\wt_{\tau(i)}, \\
&= \begin{cases}
\vphi_i-s_i-\wt_{\tau(i)} & \IF \vphi_i > \vphi_{\tau(i)}+s_i, \\
\vep_{\tau(i)} & \IF \vphi_i \leq \vphi_{\tau(i)}+s_i.
\end{cases}
\end{split} \nonumber
\end{align}
\begin{enumerate}
\item When $\vphi_i > \vphi_{\tau(i)}+s_i$,
$$
\Btil_i(b) = \begin{cases}
\frac{1}{\sqrt{2}} \Ftil_i b & \IF \vphi_i = \vphi_{\tau(i)}+s_i+1 \AND \vphi_{\tau(i)}(\Ftil_i b) = \vphi_{\tau(i)}+1, \\
\Ftil_i b & \OW.
\end{cases}
$$
\item When $\vphi_i \leq \vphi_{\tau(i)}+s_i$,
$$
\Btil_i b = \begin{cases}
\frac{1}{\sqrt{2}} \Etil_{\tau(i)} b & \IF \vphi_i = \vphi_{\tau(i)}+s_i \AND \vphi_i(\Etil_{\tau(i)} b) = \vphi_i, \\
\frac{1}{\sqrt{2}}(\Etil_{\tau(i)} b + \Ftil_i b) & \IF \vphi_i = \vphi_{\tau(i)}+s_i > \max(0,-s_{\tau(i)}) \AND \vphi_i(\Etil_{\tau(i)} b) = \vphi_i-1, \\
\Etil_{\tau(i)} b & \OW.
\end{cases}
$$
\end{enumerate}
\end{enumerate}
\end{cor}

\begin{proof}
Noting that we have $\vphi_i(b) \geq 0$ for all $i \in I$, the assertion is immediate from Corollary \ref{icrystal structure on crystal 1}.
\end{proof}

\begin{theo}\label{thm: crystal base is icrystal base}
Let $M$ be an integrable $\U$-module with a crystal base $(\clL_M,\clB_M)$. Then, as a $\mathbf{U}^\imath$-module, it satisfies the conditions $(C1)$--$(C3)$ in the beginning of Section \ref{Section: iquantum groups and icrystals}, and the pair $(\clL_M, \clB_M)$ is an $\imath$crystal base of $M$ such that the $\imath$crystal structure of $\clB_M$ is given by Corollary \ref{icrystal structure on crystal}.
\end{theo}

\begin{proof}
By the definition of $\imath$crystal base (Definition \ref{def: icrystal base}), it suffices to show that for each $i \in I$, the pair $(\clL_M, \clB_M)$ is an $\imath$crystal base of $M$, regarded as a $\mathbf{U}^\imath_i$-module.

First, we prove the assertion for the trivial module $V(0) = \mathbb{K} v_0$ with crystal base $(\mathcal{L}(0) = \mathbb{K}_\infty v_0, \mathcal{B} = \{ b_0 \})$, where $b_0 := \mathrm{ev}_\infty(v_0)$. By Lemmas \ref{icrystal basis of the trivial module}, \ref{icrystal basis of the trivial module a=0}, and \ref{icrystal basis of the trivial module a=-1}, we have
$$
\beta_i(b_0) = \begin{cases}
|s_i| & \IF a_{i,\tau(i)} = 2, \\
\max(-s_i,0) & \IF a_{i,\tau(i)} \neq 2,
\end{cases} \qu \Btil_i b_0 = \begin{cases}
\sgn(s_i) b_0 & \IF a_{i,\tau(i)} = 2, \\
0 & \IF a_{i,\tau(i)} \neq 2.
\end{cases}
$$
This shows that $\clB(0)$ is an $\imath$crystal basis of $V(0)$ whose $\imath$crystal structure is as in Example \ref{examples of icrystals} \eqref{examples of icrystals 1}.

Next, let us investigate how $\beta_i$ and $\Btil_i$ act on $M \simeq V(0) \otimes M$ for each $i \in I$. Since $M$ is integrable, as a $\U_{i,\tau(i)}$-module, it can be embedded into a direct sum of tensor powers $V_\natural^{\otimes N}$ of the natural representation of $\U_{i,\tau(i)}$ for various $N \geq 0$. Therefore, we may assume that $M \simeq V_\natural^{\otimes N}$ for some $N > 0$ as a $\mathbf{U}^\imath_i$-module. Now, the assertion follows from Remark \ref{icrystal on tensor power}.
\end{proof}

\begin{prop}\label{tensor product of morphisms}
Let $\clB_1,\clB_3$ be $\imath$crystals, $\clB_2,\clB_4$ crystals, $\mu_1 : \clB_1 \rightarrow \clB_3$ an $\imath$crystal morphisms, $\mu_2 : \clB_2 \rightarrow \clB_4$ a crystal morphism. Let $\mu_1 \otimes \mu_2 : \mathbb{C}\clB_1 \otimes \mathbb{C}\clB_2 \rightarrow \mathbb{C}\clB_3 \otimes \mathbb{C}\clB_4$ denote the $\C$-linear map given by $(\mu_1 \otimes \mu_2)(b_1 \otimes b_2) = \mu_1(b_1) \otimes \mu_2(b_2)$. Then, the following hold:
\begin{enumerate}
\item $\mu_1 \otimes \mu_2$ is an $\imath$crystal morphism.
\item If $\mu_1$ and $\mu_2$ are strict, then so is $\mu_1 \otimes \mu_2$.
\item If $\mu_1$ is very strict and $\mu_2$ is strict, then $\mu_1 \otimes \mu_2$ is very strict.
\item If $\mu_1$ is an equivalence and $\mu_2$ is an isomorphism, then $\mu_1 \otimes \mu_2$ is an equivalence.
\item If $\mu_1$ and $\mu_2$ are isomorphisms, then so is $\mu_1 \otimes \mu_2$.
\end{enumerate} 
\end{prop}

\begin{proof}
First, we confirm Definition \ref{Def: icrystal morphism} \eqref{Def: icrystal morphism 1}. Let $b_i \in \clB_i$ be such that $(\mu_1(b_1) \otimes \mu_2(b_2), b_3 \otimes b_4) \neq 0$. Then, we have $(\mu_1(b_1), b_3) \neq 0$ and  $\mu_2(b_2) = b_4$. Setting $b := b_1 \otimes b_2$ and $b' := b_3 \otimes b_4$, we obtain
\begin{align}
\begin{split}
\wti(b_3 \otimes b_4) &= \wti(b_3)+\ol{\wt(b_4)} = \wti(b_1)+\ol{\wt(b_2)} = \wti(b_1 \otimes b_2), \\
\beta_i(b_3 \otimes b_4) &= \begin{cases}
\max(F_i(b'), B_i(b'), E_i(b'))-\wt_i(b_4) & \IF a_{i,\tau(i)} = 2, \\
\max(F_i(b'), B_i(b'), E_i(b'))+\wti_i(b_3)-s_i-\wt_{\tau(i)}(b_4) & \IF a_{i,\tau(i)} \neq 2
\end{cases} \\
&= \begin{cases}
\max(F_i(b), B_i(b), E_i(b))-\wt_i(b_2) & \IF a_{i,\tau(i)} = 2, \\
\max(F_i(b), B_i(b), E_i(b))+\wti_i(b_1)-s_i-\wt_{\tau(i)}(b_2) & \IF a_{i,\tau(i)} \neq 2
\end{cases} \\
&= \beta_i(b_1 \otimes b_2).
\end{split} \nonumber
\end{align}
This confirms Definition \ref{Def: icrystal morphism} \eqref{Def: icrystal morphism 1}.

Next, let us confirm Definition \ref{Def: icrystal morphism} \eqref{Def: icrystal morphism 2}. Suppose that $\Btil_i(b_1 \otimes b_2) \in \clB_1 \otimes \clB_2$. Then, we have
$$
\Btil_i(b_1 \otimes b_2) = \begin{cases}
b_1 \otimes \Ftil_i b_2 & \IF F_i(b) > B_i(b), E_i(b), \\
\Btil_i b_1 \otimes b_2 & \IF F_i(b) \leq B_i(b) > E_i(b), \\
b_1 \otimes \Etil_{\tau(i)} b_2 & \IF F_i(b), B_i(b) \leq E_i(b),
\end{cases}
$$
We claim that
$$
\Btil_i(b_3 \otimes b_4) = \begin{cases}
b_3 \otimes \Ftil_i b_4 & \IF F_i(b) > B_i(b), E_i(b), \\
\Btil_i b_3 \otimes b_4 & \IF F_i(b) \leq B_i(b) > E_i(b), \\
b_3 \otimes \Etil_{\tau(i)} b_4 & \IF F_i(b), B_i(b) \leq E_i(b).
\end{cases}
$$
First, consider the case when $F_i(b') > B_i(b'),E_i(b')$. If our claim fails, then we have $F_i(b') = E_i(b')+1$ and $\vphi_{\tau(i)}(\Ftil_i b_4) = \vphi_{\tau(i)}(b_4)+1$. Since $\Ftil_i b_2 \neq 0$, we have $\mu_2(\Ftil_i b_2) = \Ftil_i \mu_2(b_2) = \Ftil_i b_4$, and hence,
$$
\vphi_{\tau(i)}(\Ftil_i b_2) = \vphi_{\tau(i)}(\Ftil_i b_4) = \vphi_{\tau(i)}(b_4)+1 = \vphi_{\tau(i)}(b_2)+1.
$$
This, together with $F_i(b) = F_i(b') = E_i(b')+1 = E_i(b)+1$ implies that $\Btil_i b = \frac{1}{\sqrt{2}} b_1 \otimes \Ftil_ib_2$, which is a contradiction. Thus, we obtain $\Btil_i b' = b_3 \otimes \Ftil_i b_4$. By the same way, our claim follows in the case when $F_i(b'),B_i(b') \leq E_i(b')$.

Next, consider the case when $F_i(b') \leq B_i(b') > E_i(b')$. Since $\Btil_i b \in \clB_1 \otimes \clB_2$, we have $\Btil_i b_1 \in \clB_1$. If it holds that $F_i(b') = B_i(b') \neq \beta_{\tau(i)}(b_3)$, then we obtain $F_i(b) = B_i(b) \neq \beta_{\tau(i)}(b_1)$, and hence, $\Btil_i b = \frac{1}{\sqrt{2}}(\Btil_i b_1 \otimes b_2 + b_1 \otimes \Ftil_i b_2)$. This is a contradiction. On the other hand, if $B_i(b') = E_i(b')+1$ and $\beta_i(\Btil_i b_3) = \beta_i(b_3)-2$, then it follows that $\mu_1(\Btil_i b_1) = \Btil_i \mu_1(b_1) = \Btil_i b_3$. Therefore, we have $B_i(b)=E_i(b)+1$ and $\beta_i(\Btil_i b_1) = \beta_i(b_1)-2$, and hence, $\Btil_i b = \frac{1}{\sqrt{2}} \Btil_i b_1 \otimes b_2$. This causes a contradiction, too. Thus, we obtain $\Btil_i b' = \Btil_i b_3 \otimes b_4$.

Now, we compute as follows:
\begin{align}
\begin{split}
\Btil_i(\mu_1(b_1) \otimes \mu_2(b_2)) &= \sum_{b_3 \in \clB_3} (\mu_1(b_1),b_3) \Btil_i(b_3 \otimes b_4) \\
&= \begin{cases}
\sum_{b_3 \in \clB_3} (\mu_1(b_1),b_3) b_3 \otimes \Ftil_i b_4 & \IF F_i(b) > B_i(b), E_i(b), \\
\sum_{b_3 \in \clB_3} (\mu_1(b_1),b_3) \Btil_i b_3 \otimes b_4 & \IF F_i(b) \leq B_i(b) > E_i(b), \\
\sum_{b_3 \in \clB_3} (\mu_1(b_1),b_3) b_3 \otimes \Etil_{\tau(i)} b_4 & \IF F_i(b), B_i(b) \leq E_i(b)
\end{cases} \\
&=\begin{cases}
\mu_1(b_1) \otimes \Ftil_i b_4 & \IF F_i(b) > B_i(b), E_i(b), \\
\Btil_i \mu_1(b_1) \otimes b_4 & \IF F_i(b) \leq B_i(b) > E_i(b), \\
\mu_1(b_1) \otimes \Etil_{\tau(i)} b_4 & \IF F_i(b), B_i(b) \leq E_i(b)
\end{cases} \\
&= (\mu_1 \otimes \mu_2)(\Btil_i(b_1 \otimes b_2)).
\end{split} \nonumber
\end{align}
This confirms Definition \ref{Def: icrystal morphism} \eqref{Def: icrystal morphism 2}. Thus, the first assertion of the proposition follows.

Let us prove the second assertion. By the first assertion, it suffices to show that $\Btil_i(\mu_1(b_1) \otimes \mu_2(b_2)) = (\mu_1 \otimes \mu_2)(\Btil_i(b_1 \otimes b_2))$ for all $b_1 \otimes b_2 \in \clB_1 \otimes \clB_2$ and $i \in I$. Let $b_3 \otimes b_4 \in \clB_3 \otimes \clB_4$ be such that $(\mu_1(b_1) \otimes \mu_2(b_2), b_3 \otimes b_4) \neq 0$. By the definition of $\Btil_i$ on tensor products, and the strictness of $\mu_1,\mu_2$, we have
\begin{align}
\begin{split}
&\Btil_i(b_1 \otimes b_2) = c_F b_1 \otimes \Ftil_i b_2 + c_B \Btil_i b_1 \otimes b_2 + c_E b_1 \otimes \Etil_{\tau(i)} b_2, \\
&\Btil_i(b_3 \otimes b_4) = c_F b_3 \otimes \Ftil_i b_4 + c_B \Btil_i b_3 \otimes b_4 + c_E b_3 \otimes \Etil_{\tau(i)} b_4
\end{split} \nonumber
\end{align}
for some $c_F,c_B,c_E \in \C$. Therefore, we compute
\begin{align}
\begin{split}
\Btil_i(\mu_1(b_1) \otimes \mu_2(b_2)) &= \sum_{b_3 \in \clB_3} (\mu_1(b_1),b_3) \Btil_i(b_3 \otimes b_4) \\
&= \sum_{b_3 \in \clB_3} (\mu_1(b_1), b_3)(c_F b_3 \otimes \Ftil_i b_4 + c_B \Btil_i b_3 \otimes b_4 + c_E b_1 \otimes \Etil_{\tau(i)} b_4) \\
&= c_F \mu_1(b_1) \otimes \Ftil_i b_4 + c_B \Btil_i \mu_1(b_1) \otimes b_4 + c_E \mu_1(b_1) \otimes \Etil_{\tau(i)} b_4 \\
&= c_F \mu_1(b_1) \otimes \mu_2(\Ftil_i b_2) + c_B \mu_1(\Btil_i b_1) \otimes \mu_2(b_2) + c_E \mu_1(b_1) \otimes \mu_2(\Etil_{\tau(i)} b_2) \\
&= (\mu_1 \otimes \mu_2)(\Btil_i(b_1 \otimes b_2)).
\end{split} \nonumber
\end{align}
This proves the second assertion. The remaining assertions are now obvious. Thus, the proof completes.
\end{proof}

\section{$\imath$Crystal basis of the modified $\imath$quantum group}\label{Section: stability}
In this section, we construct projective systems of $\imath$crystals and very strict morphisms, and describe their projective limits explicitly. Then, we lift this result to based $\Ui$-modules.

\subsection{$\imath$Crystal basis of $\Uidot$}\label{subsect: icrystal of Uidot}
For the rest of this paper, we fix $\sigma \in X^+$ such that
$$
\la h_i,\sigma \ra = \begin{cases}
|s_i| & \IF a_{i,\tau(i)} = 2, \\
0 & \IF a_{i,\tau(i)} = 0, \\
\max(s_i,0) & \IF a_{i,\tau(i)} = -1 \AND i \in I_\tau, \\
\max(-s_{\tau(i)},0) & \IF a_{i,\tau(i)} = -1 \AND i \notin I_\tau.
 \end{cases}
$$
Note that we have
$$
\la h_i-h_{\tau(i)},\sigma \ra = \begin{cases}
s_i & \IF i \in I_\tau, \\
-s_{\tau(i)} & \IF i \notin I_\tau
 \end{cases}
$$
for all $i \in I$ with $a_{i,\tau(i)} = -1$.

Recall from Section \ref{Section: iquantum groups and icrystals} the $1$-dimensional $\Ui$-module $V(0)^\sigma = \bbK v_0^\sigma$ and its $\bbK_\infty$-lattice $\clL(0)^\sigma = \bbK_\infty v_0^\sigma$. Set $b_0^\sigma := \ev_\infty(v_0^\sigma)$, $\clB(0)^\sigma := \{ b_0^\sigma \}$.

\begin{lem}\label{icrystral structure of B(0)sigma}
The pair $(\mathcal{L}(0)^\sigma, \mathcal{B}(0)^\sigma)$ is an $\imath$crystal base of $V(0)^\sigma$; its $\imath$crystal structure is given as follows:
$$
\wti(b_0^\sigma) = \ol{\sigma}, \qu \beta_i(b_0^\sigma) = 0, \qu \Btil_i b_0^\sigma = 0.
$$
\end{lem}

\begin{proof}
The assertion is clear from the $\Ui$-module structure of $V(0)^\sigma$.
\end{proof}

\begin{lem}\label{existence of gamma nu}
Let $\nu \in X^+$. Then, there exists a very strict $\imath$crystal morphism $\gamma_\nu : \clB(\sigma+\nu+\tau(\nu)) \rightarrow \clB(0)^\sigma$ such that
$$
\gamma_\nu(b) = \delta_{b,b_{\sigma+\nu+\tau(\nu)}} b_0^\sigma.
$$
for all $b \in \clB(\sigma+\nu+\tau(\nu))$.
\end{lem}

\begin{proof}
By Corollary \ref{icrystal structure on crystal}, we have
$$
\wti(b_{\sigma+\nu+\tau(\nu)}) = \ol{\sigma}, \qu \beta_i(b_{\sigma+\nu+\tau(\nu)}) = 0, \qu \Btil_i b_{\sigma+\nu+\tau(\nu)} = 0.
$$
This, together with Lemma \ref{icrystral structure of B(0)sigma}, shows that $\{ b_{\sigma+\nu+\tau(\nu)} \}$ forms an $\imath$crystal isomorphic to $\clB(0)^\sigma$. Thus, the assertion follows.
\end{proof}

Recall that for each $\lm \in X^+$, we set
$$
V(\lm)^\sigma := V(0)^\sigma \otimes V(\lm), \qu \clB(\lm)^\sigma := \clB(0)^\sigma \otimes \clB(\lm),
$$
and
$$
v^\sigma := v_0^\sigma \otimes v, \qu b^\sigma := b_0^\sigma \otimes b
$$
for all $v \in V(\lm)$ and $b \in \clB(\lm)$.

\begin{lem}\label{existence of rho lm}
Let $\lm \in X^+$. Then, there exists a very strict $\imath$crystal morphism $\rho_\lm : \clB(\sigma+\lm) \rightarrow \clB(\lm)^\sigma$ such that
$$
\rho_\lm(\pi_{\sigma+\lm}(b)) = \pi_\lm(b)^\sigma
$$
for all $b \in \clB(\infty)$. Consequently, there exists an injective very strict $\imath$crystal morphism $\clB(\lm)^\sigma \hookrightarrow \clB(\sigma+\lm)$ which sends $\pi_\lm(b)^\sigma$ to $\pi_{\sigma+\lm}(b)$ for all $b \in \clB(\infty;\lm)$.
\end{lem}

\begin{proof}
By \cite[Proposition 25.1.2]{L10}, there exists an injective strict crystal morphism $\eta_{\sigma,\lm} : \clB(\sigma+\lm) \rightarrow \clB(\sigma) \otimes \clB(\lm)$ such that
$$
\eta_{\sigma,\lm}(\pi_{\sigma+\lm}(b)) = b_\sigma \otimes \pi_\lm(b) \qu \Forall b \in \clB(\infty;\lm),
$$
and
$$
\eta_{\sigma,\lm}(\pi_{\sigma+\lm}(b)) \notin b_\sigma \otimes \clB(\lm) \qu \Forall b \in \clB(\infty) {\setminus} \clB(\infty;\lm).
$$
By Lemma \ref{existence of gamma nu} and Proposition \ref{tensor product of morphisms}, there exists a very strict $\imath$crystal morphism
$$
\rho_\lm : \clB(\sigma+\lm) \xrightarrow[]{\eta_{\sigma,\lm}} \clB(\sigma) \otimes \clB(\lm) \xrightarrow[]{\gamma_0 \otimes 1} \clB(0)^\sigma \otimes \clB(\lm) = \clB(\lm)^\sigma.
$$
Then, it is clear that this morphism has the required property.
\end{proof}

\begin{prop}\label{existence of pii lm nu}
Let $\lm,\nu \in X^+$. Then, there exists a very strict $\imath$crystal morphism
$$
\pi^\imath_{\lm,\nu} : \clB(\lm+\nu+\tau(\nu))^\sigma \rightarrow \clB(\lm)^\sigma
$$
such that
$$
\pi^\imath_{\lm,\nu}(\pi_{\lm+\nu+\tau(\nu)}(b)^\sigma) = \pi_\lm(b)^\sigma
$$
for all $b \in \clB(\infty)$.
\end{prop}

\begin{proof}
As in the proof of Lemma \ref{existence of rho lm}, there exists a very strict $\imath$crystal morphism
$$
\clB(\sigma+\lm+\nu+\tau(\nu)) \xrightarrow[]{\eta_{\sigma+\nu+\tau(\nu),\lm}} \clB(\sigma+\nu+\tau(\nu)) \otimes \clB(\lm) \xrightarrow[]{\gamma_\nu \otimes 1} \clB(0)^\sigma \otimes \clB(\lm) = \clB(\lm)^\sigma
$$
which sends $\pi_{\sigma+\lm+\nu+\tau(\nu)}(b)$ to $\pi_\lm(b)^\sigma$. On the other hand, by Lemma \ref{existence of rho lm}, we have a very strict $\imath$crystal morphism
$$
\clB(\lm+\nu+\tau(\nu))^\sigma \hookrightarrow \clB(\sigma+\lm+\nu+\tau(\nu)).
$$
Combining these two morphisms, we obtain a very strict $\imath$crystal morphism
$$
\pi^\imath_{\lm,\nu} : \clB(\lm+\nu+\tau(\nu))^\sigma \rightarrow \clB(\lm)^\sigma.
$$
Then, it is clear that this morphism satisfies the required property.
\end{proof}

Now, for each $\zeta \in X^\imath$, we obtain a projective system $\{ \clB(\lm)^\sigma \}_{\lm \in X^+, \ol{\lm} = \zeta}$ of $\imath$crystals with very strict morphisms $\pi^\imath_{\lm,\nu} : \clB(\lm+\nu+\tau(\nu))^\sigma \rightarrow \clB(\lm)^\sigma$. We shall describe its projective limit in the category of $\imath$crystals and very strict morphisms. To do so, we prepare three lemmas.

\begin{lem}\label{lamma; B(lm)sigma to Tsigma tensor B(lm)}
Let $\lm \in X^+$. Then, there exists an $\imath$crystal isomorphism $\clB(\lm)^\sigma \rightarrow \clT_{\ol{\sigma}} \otimes \clB(\lm)$ which sends $b^\sigma$ to $t_{\ol{\sigma}} \otimes b$ for all $b \in \clB(\lm)$.
\end{lem}

\begin{proof}
Let us compare the $\imath$crystal structures of $\clB(\lm)^\sigma$ and $\clT_{\ol{\sigma}} \otimes \clB(\lm)$. To do so, recall from Lemma \ref{Deduction from condition S's} \eqref{Deduction from condition S's 4} that $\vphi_i(\Etil_{\tau(i)} b) = \vphi_i(b)-1$ is equivalent to $\vphi_{\tau(i)}(\Ftil_i b) = \vphi_{\tau(i)}(b)$. Then, the the $\imath$crystal structure of $\clB(\lm)^\sigma$ is described as follows: Let $b \in \clB(\lm)$ and $i \in I$.
\begin{enumerate}
\item $\wti(b^\sigma) = \ol{\sigma} + \ol{\wt(b)}$.
\item When $a_{i,\tau(i)} = 2$.
\begin{align}
\begin{split}
&\beta_i(b^\sigma) = \begin{cases}
\vep_i(b)+1 & \IF \ol{\vphi_i(b)} = \ol{1}, \\
\vep_i(b) & \IF \ol{\vphi_i(b)} = \ol{0},
\end{cases} \\
&\Btil_i b^\sigma = \begin{cases}
(\Ftil_i b)^\sigma & \IF \ol{\vphi_i(b)} = \ol{1}, \\
(\Etil_i b)^\sigma & \IF \ol{\vphi_i(b)} = \ol{0}.
\end{cases}
\end{split} \nonumber
\end{align}
\item When $a_{i,\tau(i)} = 0$.
\begin{align}
\begin{split}
&\beta_i(b^\sigma) = \begin{cases}
\vphi_i(b)-\wt_{\tau(i)}(b) & \IF \vphi_i(b) > \vphi_{\tau(i)}(b), \\
\vep_{\tau(i)}(b) & \IF \vphi_i(b) \leq \vphi_{\tau(i)}(b),
\end{cases} \\
&\Btil_i b^\sigma = \begin{cases}
(\Ftil_i b)^\sigma & \IF \vphi_i(b) > \vphi_{\tau(i)}(b), \\
(\Etil_{\tau(i)} b)^\sigma & \IF \vphi_i(b) \leq \vphi_{\tau(i)}(b).
\end{cases}
\end{split} \nonumber
\end{align}
\item When $a_{i,\tau(i)} = -1$. Setting $s'_i := \begin{cases}
0 & \IF i \in I_\tau, \\
1 & \IF i \notin I_\tau,
\end{cases}$ we have
\begin{align}
\begin{split}
&F_i(b^\sigma) = \vphi_i(b), \\
&B_i(b^\sigma) = s'_i, \\
&E_i(b^\sigma) = \vphi_{\tau(i)}(b)+s'_i.
\end{split} \nonumber
\end{align}
This implies that $B_i(b^\sigma) \leq E_i(b^\sigma)$. Therefore, we obtain the following:
\begin{enumerate}
\item When $\vphi_i(b) > \vphi_{\tau(i)}(b)+s'_i$. In this case, we have
\begin{align}
\begin{split}
\beta_i(b^\sigma) &= \vphi_i(b)-s'_i-\wt_{\tau(i)}(b), \\
\Btil_i b^\sigma &= \begin{cases}
\frac{1}{\sqrt{2}} (\Ftil_i b)^\sigma & \IF \vphi_i(b) = \vphi_{\tau(i)}(b)+s'_i+1 \AND \vphi_{\tau(i)}(\Ftil_i b) = \vphi_{\tau(i)}(b)+1, \\
(\Ftil_i b)^\sigma & \OW. 
 \end{cases}
\end{split} \nonumber
\end{align}
\item When $\vphi_i(b) \leq \vphi_{\tau(i)}(b)+s'_i$. In this case, we have
\begin{align}
\begin{split}
\beta_i(b^\sigma) &= \vep_{\tau(i)}(b), \\
\Btil_i b^\sigma &= \begin{cases}
\frac{1}{\sqrt{2}} (\Etil_{\tau(i)} b)^\sigma & \IF \vphi_i(b) = \vphi_{\tau(i)}(b)+s'_i \AND \vphi_i(\Etil_{\tau(i)}b) = \vphi_i(b), \\
\frac{1}{\sqrt{2}}((\Etil_{\tau(i)}b)^\sigma + (\Ftil_i b)^\sigma) & \IF \vphi_i(b) = \vphi_{\tau(i)}(b)+s'_i > 0 \AND \vphi_{\tau(i)}(\Ftil_i b) = \vphi_{\tau(i)}(b), \\
(\Etil_{\tau(i)}b)^\sigma & \OW.
 \end{cases}
\end{split} \nonumber
\end{align}
\end{enumerate}
\end{enumerate}

On the other hand, the $\imath$crystal structure of $\clT_{\ol{\sigma}} \otimes \clB(\lm)$ is described as follows:
\begin{enumerate}
\item $\wti(t_{\ol{\sigma}} \otimes b) = \ol{\sigma} + \ol{\wt(b)}$.
\item When $a_{i,\tau(i)} = 2$.
\begin{align}
\begin{split}
&\beta_i(t_{\ol{\sigma}} \otimes b) = \begin{cases}
\vep_i(b)+1 & \IF \ol{\vphi_i(b)} = \ol{1}, \\
\vep_i(b) & \IF \ol{\vphi_i(b)} = \ol{0},
\end{cases} \\
&\Btil_i t_{\ol{\sigma}} \otimes b = \begin{cases}
t_{\ol{\sigma}} \otimes \Ftil_i b & \IF \ol{\vphi_i(b)} = \ol{1}, \\
t_{\ol{\sigma}} \otimes \Etil_i b & \IF \ol{\vphi_i(b)} = \ol{0}.
\end{cases}
\end{split} \nonumber
\end{align}
\item When $a_{i,\tau(i)} = 0$.
\begin{align}
\begin{split}
&\beta_i(t_{\ol{\sigma}} \otimes b) = \begin{cases}
\vphi_i(b)-\wt_{\tau(i)}(b) & \IF \vphi_i(b) > \vphi_{\tau(i)}(b), \\
\vep_{\tau(i)}(b) & \IF \vphi_i(b) \leq \vphi_{\tau(i)}(b),
\end{cases} \\
&\Btil_i t_{\ol{\sigma}} \otimes b = \begin{cases}
t_{\ol{\sigma}} \otimes \Ftil_i b & \IF \vphi_i(b) > \vphi_{\tau(i)}(b), \\
t_{\ol{\sigma}} \otimes \Etil_{\tau(i)} b & \IF \vphi_i(b) \leq \vphi_{\tau(i)}(b).
\end{cases}
\end{split} \nonumber
\end{align}
\item When $a_{i,\tau(i)} = -1$. Setting $s'_i := \begin{cases}
0 & \IF i \in I_\tau, \\
1 & \IF i \notin I_\tau,
\end{cases}$
we have
\begin{align}
\begin{split}
&F_i(t_{\ol{\sigma}} \otimes b) = \vphi_i(b), \\
&B_i(t_{\ol{\sigma}} \otimes b) = -\infty, \\
&E_i(t_{\ol{\sigma}} \otimes b) = \vphi_{\tau(i)}(b)+s'_i.
\end{split} \nonumber
\end{align}
This implies that $B_i(t_{\ol{\sigma}} \otimes b) < E_i(t_{\ol{\sigma}} \otimes b)$. Therefore, we obtain the following:
\begin{enumerate}
\item When $\vphi_i(b) > \vphi_{\tau(i)}(b)+s'_i$. In this case, we have
\begin{align}
\begin{split}
\beta_i(t_{\ol{\sigma}} \otimes b) &= \vphi_i(b)-s'_i-\wt_{\tau(i)}(b), \\
\Btil_i(t_{\ol{\sigma}} \otimes b) &= \begin{cases}
\frac{1}{\sqrt{2}} t_{\ol{\sigma}} \otimes \Ftil_i b & \IF \vphi_i(b) = \vphi_{\tau(i)}(b)+s'_i+1 \AND \vphi_{\tau(i)}(\Ftil_i b) = \vphi_{\tau(i)}(b)+1, \\
t_{\ol{\sigma}} \otimes \Ftil_i b & \OW.
 \end{cases}
\end{split} \nonumber
\end{align}

\item When $\vphi_i(b) \leq \vphi_{\tau(i)}(b)+s'_i$. In this case, we have
\begin{align}
\begin{split}
\beta_i(t_{\ol{\sigma}} \otimes b) &= \vep_{\tau(i)}(b), \\
\Btil_i(t_{\ol{\sigma}} \otimes b) &= \begin{cases}
\frac{1}{\sqrt{2}} t_{\ol{\sigma}} \otimes \Etil_{\tau(i)}b & \IF \vphi_i(b) = \vphi_{\tau(i)}(b)+s'_i \AND \vphi_i(\Etil_{\tau(i)}b) = \vphi_i(b), \\
\frac{1}{\sqrt{2}} t_{\ol{\sigma}} \otimes (\Etil_{\tau(i)}b + \Ftil_i b) & \IF \vphi_i(b) = \vphi_{\tau(i)}(b)+s'_i \AND \vphi_{\tau(i)}(\Ftil_i b_2) = \vphi_{\tau(i)}(b), \\
t_{\ol{\sigma}} \otimes \Etil_{\tau(i)}b & \OW.
 \end{cases}
\end{split} \nonumber
\end{align}
\end{enumerate}
\end{enumerate}

Thus, the proof completes (note that $\vphi_{\tau(i)}(\Ftil_i b_2) = \vphi_{\tau(i)}(b)$ implies that $\Ftil_i b_2 \in \clB(\lm)$, and hence, $\vphi_i(b) > 0$).
\end{proof}

\begin{lem}\label{Tzeta otimes Tmu}
Let $\zeta \in X^\imath$ and $\mu \in X$. Then, we have an $\imath$crystal isomorphism $\clT_\zeta \otimes \clT_\mu \simeq \clT_{\zeta + \ol{\mu}}$.
\end{lem}

\begin{proof}
For each $i \in I$ with $a_{i,\tau(i)} = 2$, let $\zeta_i \in \Z/2\Z$ denote the value of $\zeta$ at $i$. Then, we have
\begin{align}
\begin{split}
&\wti(t_\zeta \otimes t_\mu) = \zeta + \ol{\mu}, \\
&\beta_i(t_\zeta \otimes t_\mu) = \begin{cases}
-\infty_\ev & \IF a_{i,\tau(i)} = 2 \AND \zeta_i + \ol{\la h_i,\mu \ra} = \ol{s_i}, \\
-\infty_\odd & \IF a_{i,\tau(i)} = 2 \AND \zeta_i + \ol{\la h_i,\mu \ra} \neq \ol{s_i}, \\
-\infty & \IF a_{i,\tau(i)} \neq 2,
\end{cases} \\
&\Btil_i (t_\zeta \otimes t_\mu) = 0.
\end{split} \nonumber
\end{align}
This datum coincides with that of $t_{\zeta + \ol{\mu}} \in \clT_{\zeta + \ol{\mu}}$. Hence, the assertion follows.
\end{proof}

\begin{lem}\label{lemma: B(lm)sigma to T tensor Binfty}
Let $\lm \in X^+$. Then, there exists an injective very strict $\imath$crystal morphism $\clB(\lm)^\sigma \hookrightarrow \clT_{\ol{\sigma+\lm}} \otimes \clB(\infty)$ which sends $\pi_\lm(b)^\sigma$ to $t_{\ol{\sigma+\lm}} \otimes b$ for all $b \in \clB(\infty;\lm)$.
\end{lem}

\begin{proof}
By Lemmas \ref{lamma; B(lm)sigma to Tsigma tensor B(lm)} and \ref{Tzeta otimes Tmu}, we have an injective $\imath$crystal morphism
$$
\clB(\lm)^\sigma \simeq \clT_{\ol{\sigma}} \otimes \clB(\lm) \hookrightarrow \clT_{\ol{\sigma}} \otimes (\clT_\lm \otimes \clB(\infty)) \simeq \clT_{\ol{\sigma+\lm}} \otimes \clB(\infty)
$$
which sends $\pi_\lm(b)^\sigma$ to $t_{\ol{\sigma+\lm}} \otimes b$ for all $b \in \clB(\infty;\lm)$. Hence, to prove the assertion, it suffices to show that for each $i \in I$, $b \in \clB(\infty;\lm)$, and $b' \in \clB(\infty)$ with $(\Btil_i (t_{\ol{\sigma+\lm}} \otimes b), t_{\ol{\sigma+\lm}} \otimes b') \neq 0$, we have $b' \in \clB(\infty;\lm)$.

Let $i \in I$, $b \in \clB(\infty;\lm)$. Since there are at most two $b' \in \clB(\infty)$ satisfying $(\Btil_i (t_{\ol{\sigma+\lm}} \otimes b),t_{\ol{\sigma+\lm}} \otimes b') \neq 0$, we may take $\nu \in X^+$ such that $b' \in \clB(\infty;\lm+\nu+\tau(\nu))$ for all such $b'$.

As before, we have an injective $\imath$crystal morphism
$$
\clB(\lm+\nu+\tau(\nu))^\sigma \hookrightarrow \clT_{\ol{\sigma+\lm+\nu+\tau(\nu)}} \otimes \clB(\infty) = \clT_{\ol{\sigma+\lm}} \otimes \clB(\infty).
$$
Therefore, if we write $\Btil_i(t_{\ol{\sigma+\lm}} \otimes b) = \sum_{b' \in \clB(\infty;\lm+\nu+\tau(\nu))} c_{b'} t_{\ol{\sigma+\lm}} \otimes b'$ for some $c_{b'} \in \C$, we obtain
$$
\Btil_i \pi_{\lm+\nu+\tau(\nu)}(b)^\sigma = \sum_{b' \in \clB(\infty;\lm+\nu+\tau(\nu))} c_{b'} \pi_{\lm+\nu+\tau(\nu)}(b')^\sigma.
$$

On the other hand, by Lemma \ref{existence of pii lm nu}, we have an injective very strict $\imath$crystal morphism
$$
\clB(\lm)^\sigma \hookrightarrow \clB(\lm+\nu+\tau(\nu))^\sigma
$$
whose image is $\{ b^\sigma \mid b \in \clB(\lm+\nu+\tau(\nu);\lm) \}$. Therefore, we obtain $\pi_{\lm+\nu+\tau(\nu)}(b') \in \clB(\lm+\nu+\tau(\nu);\lm)$ for all $b'$ with $c_{b'} \neq 0$. In other words, we have $b' \in \clB(\infty;\lm)$ for all $b'$ with $c_{b'} \neq 0$. This proves our claim, and hence, the assertion follows.
\end{proof}

\begin{theo}\label{main at q=infty}
Let $\zeta \in X^\imath$. For each $\lm \in X^+$ such that $\ol{\sigma+\lm} = \zeta$, there exists a very strict $\imath$crystal morphism
$$
\pi^\imath_\lm : \clT_\zeta \otimes \clB(\infty) \rightarrow \clB(\lm)^\sigma;\ t_{\zeta} \otimes b \mapsto \pi_\lm(b)^\sigma.
$$
Moreover, $\clT_\zeta \otimes \clB(\infty)$ and $\pi^\imath_\lm$'s form the projective limit of $\{ \clB(\lm)^\sigma \}_{\lm \in X^+, \ol{\sigma+\lm} = \zeta}$ in the category of $\imath$crystals and very strict morphisms.
\end{theo}

\begin{proof}
Let $\lm \in X^+$ be such that $\ol{\sigma+\lm} = \zeta$, and $b \in \clB(\lm)$. By Lemma \ref{lemma: B(lm)sigma to T tensor Binfty}, we have an injective very strict $\imath$crystal morphism
$$
\clB(\lm)^\sigma \hookrightarrow \clT_{\zeta} \otimes \clB(\infty)
$$
which sends $\pi_\lm(b)^\sigma$ to $t_{\zeta} \otimes b$ for all $b \in \clB(\infty;\lm)$. Therefore, there exists a very strict $\imath$crystal morphism
$$
\pi^\imath_\lm : \clT_{\zeta} \otimes \clB(\infty) \rightarrow \clB(\lm)^\sigma;\ t_{\zeta} \otimes b \mapsto \pi_\lm(b)^\sigma.
$$
Then, we have
$$
\pi^\imath_{\lm,\nu} \circ \pi^\imath_{\lm+\nu+\tau(\nu)} = \pi^\imath_\lm
$$
for all $\nu \in X^+$.

Let us prove the universality. Let $\clB$ be an $\imath$crystal and $\mu_\lm : \clB \rightarrow \clB(\lm)^\sigma$ a very strict $\imath$crystal morphism such that $\pi^\imath_{\lm,\nu} \circ \mu_{\lm+\nu+\tau(\nu)} = \mu_\lm$ for all $\lm,\nu \in X^+$ with $\ol{\sigma+\lm} = \zeta$. Define a map $\mu : \clB \rightarrow \clT_\zeta \otimes \clB(\infty)$ as follows. Let $b \in \clB$. If $\mu_\lm(b) = 0$ for all $\lm \in X^+$ with $\ol{\sigma+\lm} = \zeta$, then set $\mu(b) := 0$. If $\mu_\lm(b) = m_\lm(b)^\sigma$ for some $\lm \in X^+$ with $\ol{\sigma+\lm} = \zeta$ and $m_\lm(b) \in \clB(\lm)$, then set $\mu(b) := t_\zeta \otimes m(b)$, where $m(b) \in \clB(\infty)$ is such that $\pi_\lm(m(b)) = m_\lm(b)$. In order to see the well-definedness, let $\lm' \in X^+$ be such that $\ol{\sigma+\lm'} = \zeta$ and $\mu_{\lm'}(b) = m_{\lm'}(b)^\sigma$ for some $m_{\lm'}(b) \in \clB(\lm')$. Let $m'(b) \in \clB(\infty)$ be such that $\pi_{\lm'}(m'(b)) = m_{\lm'}(b)$. We want to show that $m(b) = m'(b)$. Since $\ol{\lm-\lm'} = \ol{0}$, there exists $\nu,\nu' \in X^+$ such that $\lm-\lm' = \nu'+\tau(\nu')-(\nu+\tau(\nu))$. Set $\lm'' := \lm+\nu+\tau(\nu) = \lm'+\nu'+\tau(\nu')$. Since
$$
\pi^\imath_{\lm,\nu}(\mu_{\lm''}(b)) = \mu_\lm(b) = m_\lm(b)^\sigma,
$$
we see that $\mu_{\lm''}(b) \neq 0$. Hence, there exists $b'' \in \clB(\infty)$ such that $\mu_{\lm''}(b) = \pi_{\lm''}(b'')^\sigma$. This implies that
$$
\pi_\lm(b'')^\sigma = \pi^\imath_{\lm,\nu}(\pi_{\lm''}(b'')^\sigma) = \pi^\imath_{\lm,\nu}(\mu_{\lm''}(b)) = m_\lm(b)^\sigma = \pi_\lm(m(b))^\sigma
$$
and hence,
$$
b'' = m(b).
$$
Similarly, we obtain $b'' = m'(b)$. Thus, we obtain $m(b) = m'(b)$, as desired. Hence, $\mu$ is well-defined.

For each $\lm \in X^+$ such that $\ol{\sigma+\lm} = \zeta$, we have
$$
\pi^\imath_\lm(\mu(b)) = \pi^\imath_\lm(t_\zeta \otimes m(b)) = \pi_\lm(m(b))^\sigma = m_\lm(b)^\sigma = \mu_\lm(b).
$$
Here, we set $m(b) = 0$ if $\mu(b) = 0$. Therefore, we obtain
$$
\pi^\imath_\lm \circ \mu = \mu_\lm.
$$

It remains to shows that $\mu$ is a very strict $\imath$crystal morphism. Let $b \in \clB$ and $i \in I$. Let us write $\Btil_i b = \sum_{b' \in \clB} c_{b'} b'$ and $\Btil_i \mu(b) = \sum_{b'' \in \clB(\infty)} d_{b''} t_\zeta \otimes b''$ for some $c_{b'},d_{b''} \in \C$. We can take $\lm \in X^+$ such that $\ol{\sigma+\lm} = \zeta$ and $\mu_\lm(b'), \pi_\lm(b'') \in \clB(\lm)$ for all $b' \in \clB$ and $b'' \in \clB(\infty)$ with $c_{b'},d_{b''} \neq 0$ and $\mu(b') \neq 0$. Then, we have
\begin{align}
\begin{split}
\sum_{\substack{c_{b'} \neq 0 \\ \mu(b') \neq 0}} c_{b'} \mu_\lm(b') &= \mu_\lm(\Btil_i b) = \Btil_i \mu_\lm(b) = \Btil_i \pi^\imath_\lm(\mu(b)) = \pi^\imath_\lm(\Btil_i \mu(b)) = \sum_{d_{b''} \neq 0} d_{b''} \pi_\lm(b'')^\sigma.
\end{split} \nonumber
\end{align}
This implies that $c_{b'} = d_{b''}$ if $\mu_\lm(b') = \pi_\lm(b'')^\sigma$. Therefore, noting that $\mu_\lm(b') = \pi_\lm(b'')^\sigma$ implies that $\mu(b') = t_\zeta \otimes b''$, we obtain
$$
\mu(\Btil_i b) = \sum_{\substack{c_{b'} \neq 0 \\ \mu(b') \neq 0}} c_{b'} \mu(b') = \sum_{d_{b''} \neq 0} d_{b''} t_\zeta \otimes b'' = \Btil_i \mu(b).
$$
Thus, the proof completes.
\end{proof}

\begin{ex}\label{Tzeta otimes Binfty for diagonal type}\normalfont
Suppose that our Satake diagram is of diagonal type. We retain notation in Examples \ref{set up for diagonal type}. Let us see the $\imath$crystal structure of $\clT_\zeta \otimes \clB(\infty)$, $\zeta \in X^\imath$. For each $i \in I_\tau = I_2$ and $b \in \clB(\infty)$, we have
\begin{align}
\begin{split}
&F_i(t_\zeta \otimes b) = \vphi_i(b), \qu B_i(t_\zeta \otimes b) = -\infty, \qu E_i(t_\zeta \otimes b) = \vphi_{\tau(i)}(b)-\zeta_i, \\
&F_{\tau(i)}(t_\zeta \otimes b) = \vphi_{\tau(i)}(b), \qu B_{\tau(i)}(t_\zeta \otimes b) = -\infty, \qu E_{\tau(i)}(t_\zeta \otimes b)= \vphi_i(b)+\zeta_i,
\end{split} \nonumber
\end{align}
where $\zeta_i := \la h_i-h_{\tau(i)}, \zeta \ra$. Then, we obtain
\begin{align}
\begin{split}
&\beta_i(t_\zeta \otimes b) = \max(\vphi_i(b)+\zeta_i-\wt_{\tau(i)}(b), \vep_{\tau(i)}(b)), \\
&\beta_{\tau(i)}(t_\zeta \otimes b) = \max(\vphi_{\tau(i)}(b)-\zeta_i-\wt_i(b), \vep_i(b)), \\
&\Btil_i(t_\zeta \otimes b) = \begin{cases}
t_\zeta \otimes \Ftil_i b & \IF \vphi_i(b) > \vphi_{\tau(i)}(b)-\zeta_i, \\
t_\zeta \otimes \Etil_{\tau(i)} b & \IF \vphi_i(b) \leq \vphi_{\tau(i)}(b)-\zeta_i,
\end{cases} \\
&\Btil_{\tau(i)}(t_\zeta \otimes b) = \begin{cases}
t_\zeta \otimes \Ftil_{\tau(i)} b & \IF \vphi_{\tau(i)}(b) > \vphi_i(b)+\zeta_i, \\
t_\zeta \otimes \Etil_i b & \IF \vphi_{\tau(i)}(b) \leq \vphi_i(b)+\zeta_i.
\end{cases}
\end{split} \nonumber
\end{align}

Under the identification $\U \simeq \U_{I'} \otimes \U_{I'}$, the crystal $\clB(\infty)$ is identified with $\clB(-\infty)_{I'} \otimes \clB(\infty)_{I'}$, where $\clB(-\infty)_{I'}$ denotes the crystal basis of the positive part of $\U_{I'}$. If we write $b \in \clB(\infty)$ as $b_1 \otimes b_2 \in \clB(-\infty)_{I'} \otimes \clB(\infty)_{I'}$, the $\U$-crystal structure of $\clB(\infty)$ and the $\U_{I'} \otimes \U_{I'}$-crystal structure of $\clB(\infty) = \clB(-\infty)_{I'} \otimes \clB(\infty)_{I'}$ are related as follows: Let $k \in I'$.
\begin{align}
\begin{split}
&\vphi_{k_1}(b) = \vep_k(b_1), \qu \vphi_{k_2}(b) = \vphi_k(b_2), \qu \vep_{k_1}(b) = \vphi_k(b_1), \qu \vep_{k_2}(b) = \vep_k(b_2), \\
&\Ftil_{k_1} b = \Etil_k b_1 \otimes b_2, \qu \Ftil_{k_2} b = b_1 \otimes \Ftil_k b_2, \qu \Etil_{k_1} b = \Ftil_k b_1 \otimes b_2, \qu \Etil_{k_2} b = b_1 \otimes \Etil_k b_2.
\end{split} \nonumber
\end{align}
Also, $X^\imath$ is identified with $X_{I'}$ in a way such that $\zeta_k := \la h_k,\zeta \ra = \la h_{k_2}-h_{k_1}, \zeta \ra = \zeta_{k_2}$.

In Example \ref{icrystal of diagonal type is crystal}, we see that an $\imath$crystal can be thought of as a crystal over $\U_{I'}$. Then, the $\U_{I'}$-crystal structure of the $\imath$crystal $\clT_\zeta \otimes \clB(\infty)$ is described as follows: Let $k \in I'$ and $\clB(\infty) \ni b = b_1 \otimes b_2 \in \clB(-\infty)_{I'} \otimes \clB(\infty)_{I'}$.
\begin{align}
\begin{split}
&\vphi_k(t_\zeta \otimes b) = \beta_{k_2}(t_\zeta \otimes b) = \max(\vphi_k(b_2)+\zeta_k+\wt_k(b_1),\vphi_k(b_1)), \\
&\vep_k(t_\zeta \otimes b) = \beta_{k_1}(t_\zeta \otimes b) = \max(\vep_k(b_1)-\zeta_k-\wt_k(b_2), \vep_k(b_2)), \\
&\Ftil_k(t_\zeta \otimes b) = \Btil_{k_2}(t_\zeta \otimes b) = \begin{cases}
t_\zeta \otimes (b_1 \otimes \Ftil_k b_2) & \IF \vphi_k(b_2) > \vep_k(b_1)-\zeta_k, \\
t_\zeta \otimes (\Ftil_k b_1 \otimes b_2) & \IF \vphi_k(b_2) \leq \vep_k(b_1)-\zeta_k,
\end{cases} \\
&\Etil_k(t_\zeta \otimes b) = \Btil_{k_1}(t_\zeta \otimes b) = \begin{cases}
t_\zeta \otimes (\Etil_k b_1 \otimes b_2) & \IF \vep_k(b_1) > \vphi_k(b_2)+\zeta_k, \\
t_\zeta \otimes (b_1 \otimes \Etil_k b_2) & \IF \vep_k(b_1) \leq \vphi_k(b_2)+\zeta_k.
\end{cases}
\end{split} \nonumber
\end{align}

On the other hand, the $\U_{I'}$-crystal structure of $\clB(-\infty)_{I'} \otimes \clT_\zeta \otimes \clB(\infty)_{I'}$ is described as follows: Let $k \in I'$, $b_1 \in \clB(-\infty)_{I'}$, and $b_2 \in \clB(\infty)_{I'}$. First, we have
$$
\vep_k(b_1 \otimes t_\zeta) = \vep_k(b_1)-\zeta_k, \qu \vphi_k(b_1 \otimes t_\zeta) = \vphi_k(b_1).
$$
Hence, we obtain
\begin{align}
\begin{split}
&\vphi_k(b_1 \otimes t_\zeta \otimes b_2) = \max(\vphi_k(b_1), \vphi_k(b_2)+\zeta_k+\wt_k(b_1)), \\
&\vep_k(b_1 \otimes t_\zeta \otimes b_2) = \max(\vep_k(b_1)-\zeta_k-\wt_k(b_2), \vep_k(b_2)), \\
&\Ftil_k(b_1 \otimes t_\zeta \otimes b_2) = \begin{cases}
b_1 \otimes t_\zeta \otimes \Ftil_k b_2 & \IF \vep_k(b_1)-\zeta_k < \vphi_k(b_2), \\
\Ftil_k b_1 \otimes t_\zeta \otimes b_2 & \IF \vep_k(b_1)-\zeta_k \geq \vphi_k(b_2),
\end{cases} \\
&\Etil_k(b_1 \otimes t_\zeta \otimes b_2) = \begin{cases}
\Etil_k b_1 \otimes t_\zeta \otimes b_2 & \IF \vep_k(b_1)-\zeta_k > \vphi_k(b_2), \\
b_1 \otimes t_\zeta \otimes \Etil_k b_2 & \IF \vep_k(b_1)-\zeta_k \leq \vphi_k(b_2).
\end{cases}
\end{split} \nonumber
\end{align}
Therefore, the $\imath$crystal $\clT_\zeta \otimes \clB(\infty)$ is essentially the same as the $\U_{I'}$-crystal $\clB(-\infty)_{I'} \otimes \clT_\zeta \otimes \clB(\infty)_{I'}$.
\end{ex}

\subsection{Stability of the $\imath$canonical bases}
In this subsection, we lift the results obtained in the previous subsection to based $\Ui$-modules. Lemmas \ref{cyclicity}--\ref{characterization of highest weight element} are preparations for this purpose.

Given a sequence $\bfi = (i_1,\ldots,i_r) \in I^r$, set
$$
F_{\bfi} := F_{i_1} \cdots F_{i_r}, \qu B_{\bfi} := B_{i_1} \cdots B_{i_r}.
$$
We understand that $F_{\bfi} = 1 = B_{\bfi}$ when $r = 0$. Let $\clI \subset \bigsqcup_{r \geq 0} I^r$ be such that $\{ F_{\bfi} \mid \bfi \in \clI \}$ forms a basis of $\U^-$. By \cite[Proposition 6.2]{Ko14}, the set $\{ B_{\bfi} K_h \mid \bfi \in \clI,\ h \in Y^\imath \}$ forms a basis of $\Ui$.

For each $r \in \Z_{\geq 0}$, set $\clI_r := \clI \cap I^r$, $\clI_{< r} := \bigsqcup_{s < r} \clI_s$. For each $\bfi \in \clI_r$, set $|\bfi| := r$. Then, for each $\bfi \in \clI$, we have
\begin{align}\label{leading term of Bbfi}
\begin{split}
B_{\bfi} - F_{\bfi} \in \sum_{\bfi' \in \clI_{< |\bfi|}} F_{\bfi'} \U^{\geq 0},
\end{split}
\end{align}
where $\U^{\geq 0}$ denotes the subalgebra of $\U$ generated by $K_h, E_i$, $h \in Y$, $i \in I$.

\begin{lem}\label{cyclicity}
Let $M$ be a weight $\U$-module, $\lm \in X$, and $v \in M_\lm$. Suppose that $E_i v = 0$ for all $i \in I$. Then, we have
$$
\Ui v = \U v.
$$
\end{lem}

\begin{proof}
The submodule $\Ui v$ is spanned by vectors of the form $B_{\bfi} v$, $\bfi \in \clI$. By equation \eqref{leading term of Bbfi}, we have
$$
B_{\bfi} v = F_{\bfi} v + \sum_{\bfi' \in \clI_{< |\bfi|}} c_{\bfi',\bfi} F_{\bfi'} v
$$
for some $c_{\bfi',\bfi} \in \bbK$. This implies that $F_{\bfi} v \in \Ui v$. Since $\{ F_{\bfi} v \mid \bfi \in \clI \}$ spans $\U v$, the assertion follows.
\end{proof}

\begin{lem}\label{presentation of Verma as quotient of Ui}
Let $\lm \in X$. Then, as a $\Ui$-module, we have
$$
M(\lm) \simeq \Ui/\sum_{h \in Y^\imath} \Ui(K_h-q^{\la h,\lm \ra}).
$$
\end{lem}

\begin{proof}
By Lemma \ref{cyclicity}, we have $M(\lm) = \Ui v_\lm$. Hence, there exists a surjective $\Ui$-module homomorphism $f : \Ui \rightarrow M(\lm)$ such that $f(1) = v_\lm$. It is clear that $\sum_{h \in Y^\imath} \Ui(K_h - q^{\la h,\lm \ra}) \subset \Ker f$. Let us prove the opposite direction. Let $x \in \Ker f$. We can write $x = \sum_{(\bfi,h) \in \clI \times Y^\imath} c_{\bfi,h} B_{\bfi} K_h$. Then, we have
\begin{align}\label{expansion of xv}
\begin{split}
0 = xv_\lm = \sum_{(\bfi,h)} c_{\bfi,h} q^{\la h,\lm \ra} B_{\bfi} v_\lm.
\end{split}
\end{align}

On the other hand, by equation \eqref{leading term of Bbfi}, we have
$$
B_{\bfi} v_\lm = F_{\bfi} v_\lm + \sum_{\bfi' \in \clI_{< |\bfi|}} c_{\bfi',\bfi} F_{\bfi'} v_\lm
$$
for some $c_{\bfi',\bfi} \in \bbK$. Since $\{ F_{\bfi} v_\lm \mid \bfi \in \clI \}$ forms a basis of $M(\lm)$, we see that $\{ B_{\bfi} v_\lm \mid \bfi \in \clI \}$ forms a basis of $M(\lm)$. This, together with identity \eqref{expansion of xv}, implies that
$$
\sum_{h \in Y^\imath} c_{\bfi,h} q^{\la h,\lm \ra} = 0 \qu \Forall \bfi \in \clI.
$$

Let $\U^{\imath,0}$ denote the subalgebra of $\Ui$ generated by $K_h$, $h \in Y^\imath$, and, consider the algebra homomorphism $g : \U^{\imath,0} \rightarrow \bbK$ which sends $K_h$ to $q^{\la h,\lm \ra}$. Then, we have $\sum_{h \in Y^\imath} c_{\bfi,h} K_h \in \Ker g$. Since the subalgebra of $\U^{\imath,0}$ generated by $K_h-q^{\la h,\lm \ra}$, $h \in Y^\imath$ is contained in the kernel of $g$, and the quotient algebra is one-dimensional, we see that the subalgebra coincides with the kernel of $g$. Therefore, we obtain
$$
\sum_h c_{\bfi,h} K_h \in \sum_{h' \in Y^\imath}\U^{\imath,0}(K_{h'} - q^{\la h',\lm \ra}) \qu \Forall \bfi \in \clI.
$$
Since $x = \sum_{\bfi \in \clI} B_{\bfi}(\sum_{h \in Y^\imath} c_{\bfi,h} K_h)$, we conclude that
$$
x \in \sum_{h \in Y^\imath} \Ui(K_h - q^{\la h,\lm \ra}),
$$
as desired. Thus, the proof completes.
\end{proof}

\begin{lem}\label{defining relation of V(sigma)}
Let $\lm \in X$. Set $N(\lm)$ to be the $\Ui$-submodule of $M(\lm)$ generated by $b_i^{\la h_i,\lm \ra+1} v_\lm$, $i \in I$, where for each $n \geq 0$, we set $b_i^{n+1} := B_i^{n+1}$ if $a_{i,\tau(i)} \neq 2$, and
\begin{align}
\begin{split}
&b_i^{n+1} := \begin{cases}
\prod_{l=0}^{n}(B_i-\sgn(s_i)[|s_i|-n+2l]_i) & \IF n < |s_i|, \\
B_i \prod_{l=1}^{\frac{n-|s_i|}{2}}(B_i^2-[2l]_i^2) & \\
\cdot \prod_{l=n-|s_i|+1}^{n}(B_i-\sgn(s_i)[|s_i|-n+2l]_i) & \IF n \geq |s_i| \AND \ol{n} = \ol{s_i}, \\
\prod_{l=1}^{\frac{n-|s_i|+1}{2}}(B_i^2-[2l-1]_i^2) & \\
\cdot \prod_{l=n-|s_i|+1}^{n}(B_i-\sgn(s_i)[|s_i|-n+2l]_i) & \IF n \geq |s_i| \AND \ol{n} \neq \ol{s_i}
\end{cases}
\end{split} \nonumber
\end{align}
if $a_{i,\tau(i)} = 2$. Then, we have $V(\lm) = M(\lm)/N(\lm)$.
\end{lem}

\begin{proof}
Set $n_i := \la h_i,\lm \ra$. Since
$$
V(\lm) = M(\lm)/\sum_{i \in I} \U F_i^{n_i + 1}v_\lm,
$$
and $E_i F_j^{n_j+1} v_\lm = 0$ for all $i,j \in I$, it suffices, by Lemma \ref{cyclicity}, to show that $b_i^{n_i+1} v_\lm = F_i^{n_i+1} v_\lm$ for all $i \in I$. Let us first consider the case when $a_{i,\tau(i)} \neq 2$. Since $B_i = F_i + q_i^{s_i} E_{\tau(i)} K_i\inv$ and $E_{\tau(i)} F_i = F_i E_{\tau(i)}$, we have
$$
B_i^{n_i+1} v_\lm = F_i^{n_i+1} v_\lm,
$$
as desired.

Next, let us consider the case when $a_{i,\tau(i)} = 2$. From example \ref{icrystal structure of V(lm) for AI} and the definitions of $\Btil_i$ and $\beta_i$, we see that $B_i$ acts on the $(n_i+1)$-dimensional irreducible $\U_i$-module diagonally with eigenvalues
\begin{itemize}
\item $\{ \sgn(s_i)[|s_i-n_i|+2l]_i \mid 0 \leq l \leq n_i \}$ when $n_i < |s_i|$,
\item $\{ 0 \} \sqcup \{ \pm[2l]_i \mid 1 \leq l \leq \frac{n_i-|s_i|}{2} \} \sqcup \{ \sgn(s_i)[|s_i|-n_i+2l]_i \mid n_i-|s_i|+1 \leq l \leq n_i \}$ when $n_i \geq |s_i|$ and $\ol{n_i} = \ol{s_i}$,
\item $\{ \pm[2l-1]_i \mid 1 \leq l \leq \frac{n_i-|s_i|+1}{2} \} \sqcup \{ \sgn(s_i)[|s_i|-n_i+2l]_i \mid n_i-|s_i|+1 \leq l \leq n_i \}$ when $n_i \geq |s_i|$ and $\ol{n_i} \neq \ol{s_i}$.
\end{itemize}
This implies that
\begin{align}\label{vanishing}
\begin{split}
b_i^{n_i+1} v_\lm = 0 \qu \text{ in } V(\lm).
\end{split}
\end{align}

On the other hand, $b_i$ is of the form
$$
b_i^{n_i+1} = B_i^{n_i+1} + \sum_{k = 0}^{n_i} c_k B_i^k
$$
for some $c_k \in \bbK$. Then, equation \eqref{leading term of Bbfi} implies that
$$
b_i^{n_i+1} v_\lm = F_i^{n_i+1} v_\lm + \sum_{k= 0}^{n_i} c'_k F_i^k v_\lm
$$
for some $c'_k \in \bbK$, and hence, $b_i^{n_i+1} v_\lm = \sum_{k=0}^{n_i} c'_k F_i^k v_\lm$ in $V(\lm)$. Since $\{ F_i^k v_\lm \mid 0 \leq k \leq n_i \}$ forms a linearly independent set of $V(\lm)$, identity \eqref{vanishing} implies that $c'_k = 0$ for all $k$. Therefore, we obtain
$$
b_i^{n_i+1} v_\lm = F_i^{n_i+1} v_\lm,
$$
as desired. Thus, the proof completes.
\end{proof}

Let us recall that we have fixed $\sigma \in X^+$ at the beginning of Subsection \ref{subsect: icrystal of Uidot}.

\begin{lem}\label{characterization of highest weight element}
Let $\nu \in X^+$ and $b \in \clB(\sigma+\nu+\tau(\nu))$. Then, we have
$$
\beta_i(b) = 0, \qu \Btil_i b = 0 \qu \Forall i \in I
$$
if and only if $b = b_{\sigma+\nu+\tau(\nu)}$.
\end{lem}

\begin{proof}
By Corollary \ref{icrystal structure on crystal}, we see that $\beta_i(b) = 0$ and $\Btil_i b = 0$ if and only if
\begin{itemize}
\item $|s_i| \leq \vphi$, $\ol{s_i} = \ol{\vphi_i(b)}$, and $\vep_i(b) = 0$ when $a_{i,\tau(i)} = 2$,
\item $\vphi_i(b) \leq \vphi_{\tau(i)}(b)+s_i$ and $\vep_{\tau(i)}(b) = 0$ when $a_{i,\tau(i)} \neq 2$.
\end{itemize}
It is easily verified that $b_{\sigma+\nu+\tau(\nu)}$ satisfies the latter condition. Conversely, if $b$ satisfies the latter condition, we have $\vep_i(b) = 0$ for all $i \in I$. This implies that $b = b_{\simga+\nu+\tau(\nu)}$. Thus, the proof completes.
\end{proof}

Now, recall from Proposition \ref{V(lm)sigma is based} (see also Example \ref{A-forms}) that $V(\lm)^\sigma$ is a based $\Ui$-module for all $\lm \in X^+$.

\begin{prop}\label{gamma_sigma}
Let $\nu \in X^+$. Then, there exists a based $\Ui$-module homomorphism $\gamma_\nu :V(\sigma+\nu+\tau(\nu)) \rightarrow V(0)^\sigma$ such that $\gamma_\nu(v_{\sigma+\nu+\tau(\nu)}) = v_0^\sigma$.
\end{prop}

\begin{proof}
Recall that $V(0)^\sigma$ is isomorphic to the quotient of $\Ui$ factored by the left $\Ui$-submodule generated by $B_i$, $i \in I$ and $K_h - q^{\la h,\sigma \ra}$, $h \in Y^\imath$. On the other hand, by Lemmas \ref{presentation of Verma as quotient of Ui} and \ref{defining relation of V(sigma)}, $V(\sigma+\nu+\tau(\nu))$ is isomorphic to the quotient of $\Ui$ factored by the left $\Ui$-submodule generated by $b_i^{\la h_i,\sigma+\nu+\tau(\nu) \ra+1}$, $i \in I$ and $K_h-q^{\la h,\sigma+\nu+\tau(\nu) \ra}$, $h \in Y^\imath$. Noting that $b_i^{\la h_i,\sigma+\nu+\tau(\nu) \ra + 1} \in \Ui B_i$ and $\la h,\sigma+\nu+\tau(\nu) \ra = \la h,\sigma \ra$ for all $i \in I$ and $h \in Y^\imath$, we see that there exists a surjective $\Ui$-module homomorphism $\gamma_\nu : V(\sigma+\nu+\tau(\nu)) \rightarrow V(0)^\sigma$ such that $\gamma_\nu(v_{\sigma+\nu+\tau(\nu)}) = v_0^\sigma$.

Let us show that $\gamma_\nu$ is a based $\Ui$-module homomorphism. Set $K := \Ker \gamma_\nu$. Since $\wp^*$ preserves $\Ui$, the complement $K^\perp \subset V(\sigma+\nu+\tau(\nu))$ of $K$ is isomorphic to $V(0)^\sigma$. Hence, there exists $v'_0 \in V(\sigma+\nu+\tau(\nu))$ such that $K^\perp = \bbK v'_0$. We may assume that $v'_0 \in \clL(\sigma+\nu+\tau(\nu))$ and $b'_0 := \ev_\infty(v'_0) \neq 0$. Since $K^\perp \simeq V(0)^\sigma$, we must have
$$
\beta_i(v'_0) = 0, \ \Btil_i v'_0 = 0 \qu \Forall i \in I.
$$
Since $(\clL(\sigma+\nu+\tau(\nu)), \clB(\sigma+\nu+\tau(\nu)))$ is an $\imath$crystal base, $b'_0$ is a linear combination of vectors $b \in \clB(\sigma+\nu+\tau(\nu))$ such that $\beta_i(b) = 0$ for all $i \in I$ (cf. Definition \ref{def: icrystal base}).
Also, by Theorem \ref{thm: crystal base is icrystal base} and Example \ref{examples of icrystals} \eqref{examples of icrystals 3}, \eqref{examples of icrystals 5}, and \eqref{examples of icrystals 6}, we see that if $b \in \mathcal{B}(\sigma+\nu+\tau(\nu))$ satisfies $\beta_i(b) = 0$ for all $i \in I$, then we have $\Btil_i b = 0$ for all $i \in I$.
Therefore, the element $b'_0$ is a linear combination of elements $b \in \mathcal{B}(\sigma + \nu + \tau(\nu))$ with $\beta_i(b) = 0$ and $\Btil_i b = 0$ for all $i \in I$. By Lemma \ref{characterization of highest weight element}, this implies that
$$
b'_0 = b_{\sigma+\nu+\tau(\nu)}.
$$

Let $b \in \clB(\sigma+\nu+\tau(\nu))$. Then, we have
$$
(G^\imath(b),v'_0) \equiv (b,b'_0) = (b,b_{\sigma+\nu+\tau(\nu)}) = \delta_{b,b_{\sigma+\nu+\tau(\nu)}} \qu \pmod{q\inv \bbK_\infty}.
$$
Hence, we see that
$$
\gamma_\nu(G^\imath(b)) = \gamma_\nu(\frac{(G^\imath(b),v'_0)}{(v'_0,v'_0)} v'_0) = \frac{(G^\imath(b),v'_0)}{(v'_0,v'_0)} v_0^\sigma \in (\delta_{b,b_{\sigma+\nu+\tau(\nu)}} + q\inv \bbK_\infty) v_0^\sigma.
$$
On the other hand, since $G^\imath(b) \in V(\sigma+\nu+\tau(\nu))_{\bfA} = \Uidot_{\bfA} v_{\sigma+\nu+\tau(\nu)}$, we have $\gamma_\nu(G^\imath(b)) \in \Uidot_{\bfA} v_0^\simga = V(0)^\sigma_{\bfA}$. Similarly, since $G^\imath(b)$ is bar-invariant, so is $\gamma_\nu(G^\imath(b))$. Thus, we obtain
$$
\gamma_\nu(G^\imath(b)) = 0
$$
if $b \neq b_{\sigma+\nu+\tau(\nu)}$. This completes the proof.
\end{proof}

\begin{prop}\label{rholm at module level}
Let $\lm \in X^+$. Then, there exists a based $\Ui$-module homomorphism $\rho_\lm : V(\sigma+\lm) \rightarrow V(\lm)^\sigma$ such that $\rho_\lm(G^\imath(\pi_{\sigma+\lm}(b))) = G^\imath(\pi_\lm(b)^\sigma)$ for all $b \in \clB(\infty)$.
\end{prop}

\begin{proof}
Let $\eta_{\sigma,\lm} : V(\sigma+\lm) \rightarrow V(\sigma) \otimes V(\lm)$ denote the $\U$-module homomorphism such that $\eta(v_{\sigma+\lm}) = v_\simga \otimes v_\lm$. By \cite[Proposition 25.1.2]{L10}, we have
$$
\eta_{\sigma,\lm}(G(\pi_{\sigma+\lm}(b))) \in v_\sigma \otimes G(\pi_{\lm}(b)) + q\inv \clL(\sigma) \otimes \clL(\lm) \qu \Forall b \in \clB(\infty;\lm),
$$
and
$$
(\eta_{\sigma,\lm}(G(\pi_{\sigma+\lm}(b))), v_\sigma \otimes G(b')) \in q\inv \bbK_\infty \qu \Forall b \in \clB(\infty) {\setminus} \clB(\infty;\lm),\ b' \in \clB(\lm).
$$

Composing $\gamma_0$ in Proposition \ref{gamma_sigma} on the first factor, we obtain a $\Ui$-module homomorphism
$$
\rho_\lm := (\gamma_0 \otimes \id) \circ \eta_{\sigma,\lm} : V(\sigma+\lm) \rightarrow V(\lm)^\sigma.
$$
Then, we see that $\rho_\lm(G^\imath(b))$ is $\imath$bar-invariant, and belongs to the intersection of the $\bbK_\infty$-form and the $\bfA$-form. Moreover, for each $b \in \clB(\infty)$, we have
$$
\ev_\infty(\rho_\lm(G^\imath(\pi_{\sigma+\lm}(b)))) = \pi_\lm(b)^\sigma.
$$
By above, we conclude that
$$
\rho_\lm(G^\imath(\pi_{\sigma+\lm}(b))) = G^\imath(\pi_\lm(b)^\sigma)
$$
as desired. Thus, the proof completes.
\end{proof}

\begin{prop}
Let $\lm,\nu \in X^+$. Then, there exists a based $\Ui$-module homomorphism $\pi^\imath_{\lm,\nu} : V(\lm+\nu+\tau(\nu))^\sigma \rightarrow V(\lm)^\sigma$ such that
$$
\pi^\imath_{\lm,\nu}(G^\imath(\pi_{\lm+\nu+\tau(\nu)}(b)^\sigma)) = G^\imath(\pi_{\lm}(b)^\sigma)
$$
for all $b \in \clB(\infty)$.
\end{prop}

\begin{proof}
Consider the linear map $\pi^\imath_{\lm,\nu} : V(\lm+\nu+\tau(\nu))^\sigma \rightarrow V(\lm)^\sigma$ given by
$$
\pi^\imath_{\lm,\nu}(G^\imath(\pi_{\lm+\nu+\tau(\nu)}(b)^\sigma)) = G^\imath(\pi_{\lm}(b)^\sigma), \qu b \in \clB(\infty).
$$
In order to prove the assertion, it suffices to show that $\pi^\imath_{\lm,\nu}$ is a $\Ui$-module homomorphism. Note that the following diagram
$$
\xymatrix@C=50pt{
V(\sigma+\lm+\nu+\tau(\nu)) \ar[r]^-{\eta_{\sigma+\nu+\tau(\nu),\lm}} \ar[d]_-{\rho_{\lm+\nu+\tau(\nu)}} & V(\sigma+\nu+\tau(\nu)) \otimes V(\lm) \ar[d]^-{{\gamma_\nu} \otimes \id} \\
V(\lm+\nu+\tau(\nu))^\sigma \ar[r]_-{\pi^\imath_{\lm,\nu}} & V(\lm)^\sigma
}
$$
commutes. For each $b \in \clB(\infty)$ and $x \in \Ui$, we have
\begin{align}
\begin{split}
\pi^\imath_{\lm,\nu}(x \cdot G^\imath(\pi_{\lm+\nu+\tau(\nu)}(b)^\sigma)) &= \pi^\imath_{\lm,\nu}(x \cdot \rho_{\lm+\nu+\tau(\nu)}(G^\imath(\pi_{\sigma+\lm+\nu+\tau(\nu)}(b)))) \\
&=(\pi^\imath_{\lm,\nu} \circ \rho_{\lm+\nu+\tau(\nu)})(x \cdot G^\imath(\pi_{\sigma+\lm+\nu+\tau(\nu)}(b))) \\
&=((\gamma_\nu \otimes \id) \circ \eta_{\sigma+\nu+\tau(\nu),\lm})(xG^\imath(\pi_{\sigma+\lm+\nu+\tau(\nu)}(b))) \\
&=x((\gamma_\nu \otimes \id) \circ \eta_{\sigma+\nu+\tau(\nu),\lm})(G^\imath(\pi_{\sigma+\lm+\nu+\tau(\nu)}(b))) \\
&=x(\pi^\imath_{\lm,\nu} \circ \rho_{\lm+\nu+\tau(\nu)})(G^\imath(\pi_{\sigma+\lm+\nu+\tau(\nu)}(b))) \\
&=x \cdot G^\imath(\pi_\lm(b)^\sigma) \\
&=x \cdot \pi^\imath_{\lm,\nu}(G^\imath(\pi_{\lm+\nu+\tau(\nu)}(b)^\sigma)).
\end{split} \nonumber
\end{align}
This shows that $\pi^\imath_{\lm,\nu}$ is a $\Ui$-module homomorphism. Hence, the proof completes.
\end{proof}

Now, for each $\zeta \in X^\imath$, we obtain a projective system $\{ V(\lm)^\sigma \}_{\lm \in X^+, \ol{\lm} = \zeta}$ of based $\Ui$-modules and based homomorphisms $\pi^\imath_{\lm,\nu}$.

\begin{theo}
Let $\zeta \in X^\imath$. Then, for each $\lm \in X^+$ such that $\ol{\sigma+\lm} = \zeta$, there exists a based $\Ui$-module homomorphism $\pi^\imath_\lm : \Uidot \mathbf{1}_\zeta \rightarrow V(\lm)^\sigma$ such that
$$
\pi^\imath_\lm(G^\imath_\zeta(b)) = G^\imath(\pi_\lm(b)^\sigma)
$$
for all $b \in \clB(\infty)$. Moreover, $\Uidot \mathbf{1}_\zeta$ is the projective limit of $\{ V(\lm)^\sigma \}_{\lm \in X^+, \ol{\sigma+\lm} = \zeta}$ in the category of based $\Ui$-modules and based homomorphisms.
\end{theo}

\begin{proof}
Let $\lm \in X^+$ be such that $\ol{\sigma+\lm} = \zeta$. There exists a $\Ui$-module homomorphism $\Uidot \mathbf{1}_\zeta \rightarrow V(\sigma+\lm)$ which sends $\mathbf{1}_\zeta$ to $v_{\sigma+\lm}$. Combining $\rho_\lm$, we obtain a $\Ui$-module homomorphism $\pi^\imath_\lm : \Uidot \rightarrow V(\lm)^\sigma$ which sends $\mathbf{1}_\zeta$ to $v_\lm^\sigma$. Then, for each $\nu \in X^+$, we have
$$
\pi^\imath_{\lm,\nu} \circ \pi^\imath_{\lm+\nu+\tau(\nu)} = \pi^\imath_\lm.
$$

Let us show that $\pi^\imath_\lm$ is based. Let $b \in \clB(\infty)$. By Theorem \ref{asymptotical limit} there exists $\nu \in X^+$ such that
$$
G^\imath_{\zeta}(b) v_{\sigma+\lm+\nu+\tau(\nu)} = G^\imath(\pi_{\sigma+\lm+\nu+\tau(\nu)}(b)).
$$
Then, we have
\begin{align}
\begin{split}
\pi^\imath_\lm(G^\imath_\zeta(b)) &= \pi^\imath_{\lm,\nu} \circ \pi^\imath_{\lm+\nu+\tau(\nu)}(G^\imath_\zeta(b)) \\
&= \pi^\imath_{\lm,\nu} \circ \rho_{\lm+\nu+\tau(\nu)}(G^\imath(\pi_{\sigma+\lm+\nu+\tau(\nu)}(b))) \\
&= \pi^\imath_{\lm,\nu}(G^\imath(\pi_{\lm+\nu+\tau(\nu)}(b)^\sigma)) \\
&= G^\imath(\pi_\lm(b)^\sigma),
\end{split} \nonumber
\end{align}
as desired.

The universality can be proved in a similar way to Theorem \ref{main at q=infty}. Thus, the proof completes.
\end{proof}

\begin{rem}\normalfont
  The previous theorem partly settles Bao and Wang's conjecture in \cite[Remark 6.18]{BW18}.
  Actually, when $s_i = 0$ for all $i \in I_\tau$, we can take $\sigma = 0$, and in this case, the previous theorem states that the $\imath$canonical bases are stable (or strongly compatible in Bao-Wang's terminology). 
\end{rem}

The very strict $\imath$crystal morphism $\pi^\imath_\lm : \clT_\zeta \otimes \clB(\infty) \rightarrow \clB(\lm)^\sigma$ can be thought of as the crystal limit of the based $\Ui$-module homomoprhism $\pi^\imath_\lm : \Uidot \mathbf{1}_\zeta \rightarrow V(\lm)^\sigma$. Hence, it is reasonable to denote $\bigsqcup_{\zeta \in X^\imath} \clT_\zeta \otimes \clB(\infty)$ by $\clBidot$, and call it the $\imath$crystal basis of $\Uidot$.

\begin{ex}\normalfont
Suppose that our Satake diagram is of diagonal type. As we have seen in Example \ref{Tzeta otimes Binfty for diagonal type}, the $\imath$crystal $\clT_\zeta \otimes \clB(\infty)$ is essentially the same as the $\U_{I'}$-crystal $\clB(-\infty)_{I'} \otimes \clT_\zeta \otimes \clB(\infty)_{I'}$. Therefore, our description $\clBidot = \bigsqcup_{\zeta \in X^\imath} \clT_\zeta \otimes \clB(\infty)$ of the $\imath$crystal basis of the modified $\imath$quantum group $\Uidot$ is essentially the same as Kashiwara's description \cite[Theorem 3.1.1]{Ka94} $\clBdot = \bigsqcup_{\zeta \in X_{I'}} \clB(-\infty)_{I'} \otimes \clT_\zeta \otimes \clB(\infty)_{I'}$ of the crystal basis of the modified quantum group $\Udot_{I'}$.
\end{ex}

\section{Proofs}\label{Section: proofs}
In this section, we give proofs of Propositions \ref{tensor product of icrystal and crystal} and \ref{associativity for icrystal}, and complete our argument.

\subsection{Proof of Proposition \ref{tensor product of icrystal and crystal}}\label{Subsection: proof of tensor product rule}
Let $b_1 \in \clB_1$, $b_2 \in \clB_2$, and $i \in I$. Set $b := b_1 \otimes b_2$, $\beta_i := \beta_i(b_1)$, $\wti := \wti(b_1)$, $\wti_i := \wti_i(b_1)$, $\vep_i := \vep_i(b_2)$, $\vphi_i := \vphi_i(b_2)$, $\wt := \wt(b_2)$, $\wt_i := \wt_i(b_2)$.

\begin{lem}\label{estimate}
Suppose that $a_{i,\tau(i)} = -1$. Then, the following hold:
\begin{enumerate}
\item If $B_i(b) = \beta_{\tau(i)}$, then
$$
\beta_{\tau(i)}(b)+\wti_i(b)-s_i = \max(E_i(b)-1, B_i(b), F_i(b))+\wti_i-s_i-\wt_{\tau(i)}.
$$
\item If $B_i(b) \neq \beta_{\tau(i)}$, then
$$
\beta_{\tau(i)}(b)+\wti_i(b)-s_i = \max(E_i(b)-1, B_i(b)-1, F_i(b))+\wti_i-s_i-\wt_{\tau(i)}.
$$
\end{enumerate}
\end{lem}

\begin{proof}
The assertions follow from definitions and direct calculation.
\end{proof}

First, we confirm Definition \ref{Def: icrystal} \eqref{Def: icrystal 2}. Let $b' = b'_1 \otimes b'_2 \in \clB$ be such that $(\Btil_i b,b') \neq 0$. By the definition of $\Btil_i b$, we see that $b'$ is either $b_1 \otimes \Ftil_i b_2$, $b_1 \otimes \Etil_{\tau(i)} b_2$, or $b''_1 \otimes b_2$ for some $b''_1 \in \clB_1$ with $(\Btil_i b_1,b''_1) \neq 0$. Each element has weight $\wti(b)-\ol{\alpha_i}$. This confirms the axiom.

Below, we confirm the remaining axioms

\subsubsection{When $a_{i,\tau(i)} = 2$}
By the definition of $\beta_i(b)$, we have $\beta_i(b) \notin \Z$ if and only if $\beta_i,\vphi_i \notin \Z$. In this case, we have $\beta_i(b) = \beta_i-\wt_i \in \{ -\infty_\ev, -\infty_\odd \}$ and
$$
\Btil_i b = \Btil_i b_1 \otimes b_2 = 0.
$$
Hence, Definition \ref{Def: icrystal} \eqref{Def: icrystal 1} and \eqref{Def: icrystal 3a} are satisfied.

\begin{enumerate}
\item When $\beta_i \leq \vphi_i$ and $\ol{\beta_i} \neq \ol{\vphi_i}$. In this case, we have
\begin{align}
\begin{split}
&\beta_i(b) = \vep_i+1, \\
&\Btil_i b = b_1 \otimes \Ftil_i b_2.
\end{split} \nonumber
\end{align}
Noting that $\wti_i = \ol{\beta_i + s_i}$ and $\wt_i = \vphi_i-\vep_i$, we see that
$$
\wti_i(b) = \ol{\beta_i+s_i+\wt_i} = \ol{s_i-\vep_i+1} = \ol{\beta_i(b)+s_i}.
$$
This confirms Definition \ref{Def: icrystal} \eqref{Def: icrystal 3b}.

Let $b' = b'_1 \otimes b'_2 \in \clB$ be such that $(\Btil_i b,b') \neq 0$. Then, we have $b' = b_1 \otimes \Ftil_i b_2 = \Btil_i b$. Since $\vphi_i(b'_2) = \vphi_i-1 \geq \beta_i$ (note that $\beta_i \leq \vphi_i$ and $\ol{\beta_i} \neq \ol{\vphi_i}$ implies $\beta_i < \vphi_i$) and $\ol{\vphi_i(b'_2)} \neq \ol{\vphi_i} \neq \ol{\beta_i}$, we obtain
\begin{align}
\begin{split}
&\beta_i(b') = \vep_i(b'_2) = \vep_i+1 = \beta_i(b), \\
&\Btil_i b' = b'_1 \otimes \Etil_i b'_2 = b_1 \otimes b_2 = b.
\end{split} \nonumber
\end{align}
This confirms Definition \ref{Def: icrystal} \eqref{Def: icrystal 2.5}, \eqref{Def: icrystal 2.6}, and \eqref{Def: icrystal 3c}.
\item When $\beta_i > \vphi_i$. In this case, we have
\begin{align}
\begin{split}
&\beta_i(b) = \beta_i-\wt_i, \\
&\Btil_i b = \Btil_i b_1 \otimes b_2.
\end{split} \nonumber
\end{align}
Since $\wti_i = \ol{\beta_i+s_i}$, we obtain
$$
\wti_i(b) = \wti_i + \ol{\wt_i} = \ol{\beta_i+s_i+\wt_i} = \ol{\beta_i(b)+s_i}.
$$
This confirms Definition \ref{Def: icrystal} \eqref{Def: icrystal 3b}.

Let $b' = b'_1 \otimes b'_2 \in \clB$ be such that $(\Btil_i b,b') \neq 0$. Then, we have $(\Btil_i b_1,b'_1) \neq 0$, and $b'_2 = b_2$. Since $\beta_i(b'_1) = \beta_i$, we obtain $\beta_i(b'_1) > \vphi_i$. Hence, it follows that
\begin{align}
\begin{split}
&\beta_i(b') = \beta_i(b'_1)-\wt_i(b'_2) = \beta_i-\wt_i = \beta_i(b), \\
&\Btil_i b' = \Btil_i b'_1 \otimes b'_2 = \Btil_i b'_1 \otimes b_2, \\
&(b, \Btil_i b') = (b_1, \Btil_i b'_1) = (\Btil_i b_1, b'_1) = (\Btil_i b, b').
\end{split} \nonumber
\end{align}
This confirms Definition \ref{Def: icrystal} \eqref{Def: icrystal 2.5} and \eqref{Def: icrystal 3c}. Furthermore, if $\Btil_i b \in \clB$, we have $b'_1 = \Btil_i b_1$, and hence,
$$
\Btil_i b' = \Btil_i b'_1 \otimes b_2 = b_1 \otimes b_2 = b.
$$
This confirms Definition \ref{Def: icrystal} \eqref{Def: icrystal 2.6}.
\item When $\beta_i \leq \vphi_i$ and $\ol{\beta_i} = \ol{\vphi_i}$. In this case, we have
\begin{align}
\begin{split}
&\beta_i(b) = \vep_i, \\
&\Btil_i b = b_1 \otimes \Etil_i b_2.
\end{split} \nonumber
\end{align}
Noting that $\wti_i = \ol{\beta_i + s_i}$ and $\wt_i = \vphi_i-\vep_i$, we see that
$$
\wti_i(b) = \ol{\beta_i+s_i+\wt_i} = \ol{s_i-\vep_i} = \ol{\beta_i(b)+s_i}.
$$
This confirms Definition \ref{Def: icrystal} \eqref{Def: icrystal 3b}.

Let $b' = b'_1 \otimes b'_2 \in \clB$ be such that $(\Btil_i b,b') \neq 0$. Then, we have $b' = b_1 \otimes \Etil_i b_2 = \Btil_i b$. Since $\vphi_i(b'_2) = \vphi_i+1 > \beta_i$ and $\ol{\vphi_i(b'_2)} \neq \ol{\vphi_i} = \ol{\beta_i}$, we obtain
\begin{align}
\begin{split}
&\beta_i(b') = \vep_i(b'_2) + 1 = \vep_i = \beta_i(b), \\
&\Btil_i b' = b'_1 \otimes \Ftil_i b'_2 = b_1 \otimes b_2 = b.
\end{split} \nonumber
\end{align}
This confirms Definition \ref{Def: icrystal} \eqref{Def: icrystal 2.5}, \eqref{Def: icrystal 2.6}, and \eqref{Def: icrystal 3c}.
\end{enumerate}

\subsubsection{When $a_{i,\tau(i)} = 0$}
By the definition of $\beta_i(b)$, we have $\beta_i(b) \notin \Z$ if and only if $\vphi_i,\beta_i,\vphi_{\tau(i)} = -\infty$. In this case, we have $\beta_i(b) = \vep_{\tau(i)} = -\infty$ and
$$
\Btil_i b = b_1 \otimes \Etil_{\tau(i)} b_2 = 0.
$$
Hence, Definition \ref{Def: icrystal} \eqref{Def: icrystal 1} and \eqref{Def: icrystal 4a} are satisfied.

Also, we have
\begin{align}
\begin{split}
\beta_{\tau(i)}(b)+\wti_i(b) &= \max(\vphi_{\tau(i)}+\wti_{\tau(i)}-\wt_i, \beta_{\tau(i)}-\wt_i, \vep_i) + (\wti_i + \wt_i-\wt_{\tau(i)}) \\
&= \max(\vep_{\tau(i)}, \beta_i-\wt_{\tau(i)}, \vphi_i+\wti_i-\wt_{\tau(i)}) = \beta_i(b).
\end{split} \nonumber
\end{align}
This confirms \ref{Def: icrystal} \eqref{Def: icrystal 4b}.

\begin{enumerate}
\item When $\vphi_i > \beta_{\tau(i)}, \vphi_{\tau(i)}-\wti_i$. In this case, we have
\begin{align}
\begin{split}
&\beta_i(b) = \vphi_i+\wti_i-\wt_{\tau(i)}, \\
&\Btil_i b = b_1 \otimes \Ftil_i b_2.
\end{split} \nonumber
\end{align}

Let $b' = b'_1 \otimes b'_2 \in \clB$ be such that $(\Btil_i b, b') \neq 0$. Then, we have $b' = b_1 \otimes \Ftil_i b_2 = \Btil_i b$. We compute as
\begin{align}
\begin{split}
&\vphi_{\tau(i)}(b'_2) = \vphi_{\tau(i)} < \vphi_i+\wti_i, \\
&\beta_i(b'_1) = \beta_i = \beta_{\tau(i)}+\wti_i < \vphi_i+\wti_i, \\
&\vphi_i(b'_2)-\wti_{\tau(i)}(b'_1) = (\vphi_i-1)+\wti_i.
\end{split} \nonumber
\end{align}
This implies that $F_{\tau(i)}(b'), B_{\tau(i)}(b') \leq E_{\tau(i)}(b')$, and hence,
\begin{align}
\begin{split}
&\beta_{\tau(i)}(b') = \vep_i(b'_2) = \vep_i+1, \\
&\beta_i(b') = \beta_{\tau(i)}(b')+\wti_i(b') = (\vep_i+1)+(\wti_i+\wt_i-\wt_{\tau(i)}-2) = \beta_i(b)-1, \\
&\Btil_{\tau(i)} b' = b'_1 \otimes \Etil_i b'_2 = b_1 \otimes b_2 = b.
\end{split} \nonumber
\end{align}
This confirms Definition \ref{Def: icrystal} \eqref{Def: icrystal 2.5}, \eqref{Def: icrystal 2.6}, and \eqref{Def: icrystal 4c}.
\item When $\vphi_i \leq \beta_{\tau(i)} > \vphi_{\tau(i)}-\wti_i$. In this case, we have
\begin{align}
\begin{split}
&\beta_i(b) = \beta_i-\wt_{\tau(i)}, \\
&\Btil_i b = \Btil_i b_1 \otimes b_2.
\end{split} \nonumber
\end{align}

Let $b' = b'_1 \otimes b'_2 \in \clB$ be such that $(\Btil_i b, b') \neq 0$. Then, we have $b' = \Btil_i b_1 \otimes b_2 = \Btil_i b$. We compute as
\begin{align}
\begin{split}
&\vphi_{\tau(i)}(b'_2) = \vphi_{\tau(i)} < \beta_{\tau(i)}+\wti_i = \beta_i, \\
&\beta_i(b'_1) = \beta_i-1, \\
&\vphi_i(b'_2)-\wti_{\tau(i)}(b'_1) = \vphi_i+(\wti_i-2) \leq \beta_{\tau(i)}+\wti_i-2 = \beta_i-2.
\end{split} \nonumber
\end{align}
This implies that $F_{\tau(i)}(b') \leq B_{\tau(i)}(b') > E_{\tau(i)}(b')$, and hence,
\begin{align}
\begin{split}
&\beta_{\tau(i)}(b') = \beta_{\tau(i)}(b'_1)-\wt_i(b'_2) = (\beta_{\tau(i)}+1)-\wt_i, \\
&\beta_i(b') = \beta_{\tau(i)}(b')+\wti_i(b') = (\beta_{\tau(i)}-\wt_i+1)+(\wti_i+\wt_i-\wt_{\tau(i)}-2) = \beta_i(b)-1, \\
&\Btil_{\tau(i)} b' = \Btil_{\tau(i)} b'_1 \otimes b'_2 = b_1 \otimes b_2 = b.
\end{split} \nonumber
\end{align}
This confirms Definition \ref{Def: icrystal} \eqref{Def: icrystal 2.5}, \eqref{Def: icrystal 2.6}, and \eqref{Def: icrystal 4c}.
\item When $\vphi_i,\beta_{\tau(i)} \leq \vphi_{\tau(i)}-\wti_i$. In this case, we have
\begin{align}
\begin{split}
&\beta_i(b) = \vep_{\tau(i)}, \\
&\Btil_i b = b_1 \otimes \Etil_{\tau(i)} b_2.
\end{split} \nonumber
\end{align}

Let $b' = b'_1 \otimes b'_2 \in \clB$ be such that $(\Btil_i b, b') \neq 0$. Then, we have $b' = b_1 \otimes \Etil_{\tau(i)} b_2 = \Btil_i b$. We compute as
\begin{align}
\begin{split}
&\vphi_{\tau(i)}(b'_2) = \vphi_{\tau(i)}+1, \\
&\beta_i(b'_1) = \beta_i = \beta_{\tau(i)}+\wti_i \leq \vphi_{\tau(i)}, \\
&\vphi_i(b'_2)-\wti_{\tau(i)}(b'_1) = \vphi_i+\wti_i \leq \vphi_{\tau(i)}.
\end{split} \nonumber
\end{align}
This implies that $F_{\tau(i)}(b') > B_{\tau(i)}(b'), E_{\tau(i)}(b')$, and hence,
\begin{align}
\begin{split}
&\beta_{\tau(i)}(b') = \vphi_{\tau(i)}(b'_2)+\wti_{\tau(i)}(b'_1)-\wt_i(b'_2) = (\vphi_{\tau(i)}+1)-\wti_i-\wt_i, \\
&\beta_i(b') = \beta_{\tau(i)}(b')+\wti_i(b') = (\vphi_{\tau(i)}-\wti_i-\wt_i+1)+(\wti_i+\wt_i-\wt_{\tau(i)}-2) = \beta_i(b)-1, \\
&\Btil_{\tau(i)} b' = b'_1 \otimes \Ftil_{\tau(i)} b'_2 = b_1 \otimes b_2 = b.
\end{split} \nonumber
\end{align}
This confirms Definition \ref{Def: icrystal} \eqref{Def: icrystal 2.5}, \eqref{Def: icrystal 2.6}, and \eqref{Def: icrystal 4c}.
\end{enumerate}

\subsubsection{When $a_{i,\tau(i)} = -1$}
By the definition of $\beta_i(b)$, we have $\beta_i(b) \notin \Z$ if and only if $\vphi_i,\beta_i,\vphi_{\tau(i)} = -\infty$. In this case, we have $\beta_i(b) = \vep_{\tau(i)} = -\infty$, and
$$
\Btil_i b = b_1 \otimes \Etil_{\tau(i)} b_2 = 0.
$$
Hence, Definition \ref{Def: icrystal} \eqref{Def: icrystal 1} and \eqref{Def: icrystal 5a} are satisfied.

\begin{enumerate}
\item When $B_i(b) < F_i(b) = E_i(b)+1$ and $\vphi_{\tau(i)}(\Ftil_i b_2) = \vphi_{\tau(i)}+1$. In this case, we have
\begin{align}
\begin{split}
&\beta_i(b) = F_i(b)+\wti_i-s_i-\wt_{\tau(i)} = (\vphi_{\tau(i)}+1)-\wt_{\tau(i)} = \vep_{\tau(i)}+1, \\
&\beta_{\tau(i)}(b)+\wti_i(b)-s_i = F_i(b)+\wti_i-s_i-\wt_{\tau(i)} = \beta_i(b), \\
&\Btil_i b = \frac{1}{\sqrt{2}} b_1 \otimes \Ftil_i b_2 \notin \clB.
\end{split} \nonumber
\end{align}
For the second line of the identity above, we used Lemma \ref{estimate}. This confirms Definition \ref{Def: icrystal} \eqref{Def: icrystal 2.6}, \eqref{Def: icrystal 5b}, and \eqref{Def: icrystal 5c}.

Let $b' = b'_1 \otimes b'_2 \in \clB$ be such that $(\Btil_i b,b') \neq 0$. Then, we have $b' = b_1 \otimes \Ftil_i b_2$. By our assumption, we see that
\begin{align}
\begin{split}
&F_i(b') = \vphi_i-1 < F_i(b), \\
&B_i(b') = B_i(b) < F_i(b), \\
&E_i(b') = E_i(b)+1 = F_i(b),
\end{split} \nonumber
\end{align}
and hence,
\begin{align}
\begin{split}
&\beta_i(b') = E_i(b')+\wti_i-s_i-\wt_{\tau(i)}(b') = \vep_{\tau(i)}(b'_2) = \vep_{\tau(i)} = \beta_i(b)-1, \\
&\beta_{\tau(i)}(b')+\wti_i(b')-s_i = (E_i(b')-1)+\wti_i-s_i-\wt_{\tau(i)}(b'_2) \neq \beta_i(b').
\end{split} \nonumber
\end{align}
For the last line of the identity above, we used Lemma \ref{estimate}. This confirms Definition \ref{Def: icrystal} \eqref{Def: icrystal 5d}.

Now, we compute as
\begin{align}
\begin{split}
&F_{\tau(i)}(b') = \vphi_{\tau(i)}+1, \\
&B_{\tau(i)}(b') = \beta_{\tau(i)}+\wti_i-s_i+1 \leq \beta_i+1 < \vphi_{\tau(i)}+2, \\
&E_{\tau(i)}(b') = (\vphi_i-1)+\wti_i-s_i+1 = \vphi_i+\wti_i-s_i = \vphi_{\tau(i)}+1, \\
&\beta_i(b'_1) = \beta_i < \vphi_{\tau(i)}+1, \\
&\Etil_i b'_2 = b_2 \neq 0, \\
&\vphi_{\tau(i)}(\Etil_i b'_2) = \vphi_{\tau(i)} = \vphi_{\tau(i)}(b'_2)-1.
\end{split} \nonumber
\end{align}
This implies that
\begin{align}
\begin{split}
&\Btil_{\tau(i)} b' = \frac{1}{\sqrt{2}}(b'_1 \otimes \Etil_i b'_2 + b'_1 \otimes \Ftil_{\tau(i)} b'_2) = \frac{1}{\sqrt{2}}(b + b_1 \otimes \Ftil_{\tau(i)} b'_2),
\end{split} \nonumber
\end{align}
which shows that
$$
(\Btil_i b,b') = \frac{1}{\sqrt{2}} = (b, \Btil_{\tau(i)}b').
$$
This confirms Definition \ref{Def: icrystal} \eqref{Def: icrystal 2.5}.
\item When $F_i(b) \leq B_i(b) = E_i(b)+1$ and $\beta_i(\Btil_i b_1) = \beta_i-2$. In this case, we have
\begin{align}
\begin{split}
&\beta_{\tau(i)} = B_i(b), \\
&\beta_i(b) = B_i(b)+\wti_i-s_i-\wt_{\tau(i)} = (\vphi_{\tau(i)}+1)-\wt_{\tau(i)} = \vep_{\tau(i)}+1, \\
&\beta_{\tau(i)}(b)+\wti_i(b)-s_i = B_i(b)+\wti_i-s_i-\wt_{\tau(i)} = \beta_i(b), \\
&\Btil_i b = \frac{1}{\sqrt{2}} \Btil_i b_1 \otimes b_2 \notin \clB.
\end{split} \nonumber
\end{align}
This confirms Definition \ref{Def: icrystal} \eqref{Def: icrystal 2.6}, \eqref{Def: icrystal 5b}, and \eqref{Def: icrystal 5c}.

Let $b' = b'_1 \otimes b'_2 \in \clB$ be such that $(\Btil_i b,b') \neq 0$. Then, we have $b' = \Btil_i b_1 \otimes b_2$. By our assumption, we see that
\begin{align}
\begin{split}
&F_i(b') = \vphi_i \leq B_i(b), \\
&B_i(b') = (\beta_i-2)-(\wti_i-3)+s_i = B_i(b)+1, \\
&E_i(b') = E_i(b)+3 = B_i(b)+2,
\end{split} \nonumber
\end{align}
and hence,
\begin{align}
\begin{split}
&\beta_i(b') = E_i(b')+\wti_i(b'_1)-s_i-\wt_{\tau(i)} = \vep_{\tau(i)}(b'_2) = \vep_{\tau(i)} = \beta_i(b)-1, \\
&\beta_{\tau(i)}(b')+\wti_i(b')-s_i = (E_i(b')-1)+\wti_i(b'_1)-s_i-\wt_{\tau(i)} \neq \beta_i(b').
\end{split} \nonumber
\end{align}
This confirms Definition \ref{Def: icrystal} \eqref{Def: icrystal 5d}.

Now, we compute as
\begin{align}
\begin{split}
&F_{\tau(i)}(b') = \vphi_{\tau(i)}, \\
&B_{\tau(i)}(b') = \beta_{\tau(i)}(b'_1)+(\wti_i-3)-s_i+1 = \beta_{\tau(i)}+\wti_i-s_i-1 = \vphi_{\tau(i)}, \\
&E_{\tau(i)}(b') = \vphi_i+(\wti_i-3)-s_i+1 = \vphi_i+\wti_i-s_i-2 \leq \vphi_{\tau(i)}-1, \\
&\beta_i(b'_1) = \beta_i-2 = \vphi_{\tau(i)}-1.
\end{split} \nonumber
\end{align}
This implies that
\begin{align}
\begin{split}
&\Btil_{\tau(i)} b' = \frac{1}{\sqrt{2}}(\Btil_{\tau(i)} b'_1 \otimes b'_2 + b'_1 \otimes \Ftil_{\tau(i)} b'_2) = \frac{1}{\sqrt{2}}(b + b'_1 \otimes \Ftil_{\tau(i)} b_2),
\end{split} \nonumber
\end{align}
which shows that
$$
(\Btil_i b,b') = \frac{1}{\sqrt{2}} = (b, \Btil_{\tau(i)}b').
$$
This confirms Definition \ref{Def: icrystal} \eqref{Def: icrystal 2.5}.
\item When $E_i(b) < F_i(b) = B_i(b) \neq \beta_{\tau(i)}$. Since $B_i(b) \neq \beta_{\tau(i)}$, we have
$$
\Btil_i b_1 \in \clB_1, \qu \beta_i(\Btil_i b_1) = \beta_i-1, \qu \beta_{\tau(i)}(\Btil_i b_1) = \beta_{\tau(i)}+2
$$
if $\Btil_i b_1 \neq 0$. Also, we have
\begin{align}
\begin{split}
&\beta_i(b) = B_i(b)+\wti_i-s_i-\wt_{\tau(i)} = \beta_i-\wt_{\tau(i)}, \\
&\beta_{\tau(i)}(b)+\wti_i(b)-s_i = F_i(b)+\wti_i-s_i-\wt_{\tau(i)} = \beta_i(b), \\
&\Btil_i b = \frac{1}{\sqrt{2}}(\Btil_i b_1 \otimes b_2 + b_1 \otimes \Ftil_i b_2) \notin \clB.
\end{split} \nonumber
\end{align}
This confirms Definition \ref{Def: icrystal} \eqref{Def: icrystal 2.6}, \eqref{Def: icrystal 5b}, and \eqref{Def: icrystal 5c}.

Let $b' = b'_1 \otimes b'_2 \in \clB$ be such that $(\Btil_i b,b') \neq 0$. Then, we have $b' = \Btil_i b_1 \otimes b_2$ or $b' = b'' = b''_1 \otimes b''_2 = b_1 \otimes \Ftil_i b_2$. By our assumption, we see that
\begin{align}
\begin{split}
&F_i(b') = \vphi_i = B_i(b), \\
&B_i(b') = (\beta_i-1)-(\wti_i-3)+s_i = B_i(b)+2, \\
&E_i(b') = E_i(b)+3 \leq B_i(b)+2,
\end{split} \nonumber
\end{align}
and hence,
\begin{align}\label{id 1}
\begin{split}
&\beta_i(b') = B_i(b')+\wti_i(b'_1)-s_i-\wt_{\tau(i)} = \beta_i(b'_1)-\wt_{\tau(i)} = \beta_i(b)-1, \\
&\beta_{\tau(i)}(b')+\wti_i(b')-s_i = (B_i(b')-1)+\wti_i(b'_1)-s_i-\wt_{\tau(i)} \neq \beta_i(b').
\end{split}
\end{align}
Also, we compute as
\begin{align}
\begin{split}
&F_i(b'') = \vphi_i-1 = B_i(b)-1, \\
&B_i(b'') = B_i(b), \\
&E_i(b'') \leq E_i(b)+1 \leq B_i(b).
\end{split} \nonumber
\end{align}
This shows that
\begin{align}\label{id 2}
\begin{split}
&\beta_i(b'') = B_i(b'')+\wti_i-s_i-\wt_{\tau(i)}(b''_2) = \beta_i-\wt_{\tau(i)}(b''_2) = \beta_i(b)-1, \\
&\beta_{\tau(i)}(b'')+\wti_i(b'')-s_i = (B_i(b'')-1)+\wti_i-s_i-\wt_{\tau(i)}(b''_2) \neq \beta_i(b'').
\end{split}
\end{align}
Identities \eqref{id 1} and \eqref{id 2} confirm Definition \ref{Def: icrystal} \eqref{Def: icrystal 5d}.

Now, we compute as
\begin{align}
\begin{split}
&F_{\tau(i)}(b') = \vphi_{\tau(i)} < \beta_i, \\
&B_{\tau(i)}(b') = (\beta_{\tau(i)}+2)+(\wti_i-3)-s_i+1 = \beta_i-1, \\
&E_{\tau(i)}(b') = \vphi_i+\wti_i-s_i-2 = \beta_i-2, \\
&\Btil_{\tau(i)} b'_1 = b_1 \in \clB_1, \\
&\beta_{\tau(i)}(\Btil_{\tau(i)} b'_1) = \beta_{\tau(i)} = \beta_{\tau(i)}(b'_1)-2.
\end{split} \nonumber
\end{align}
This implies that
\begin{align}
\begin{split}
&\Btil_{\tau(i)} b' = \frac{1}{\sqrt{2}} \Btil_{\tau(i)} b'_1 \otimes b'_2 = \frac{1}{\sqrt{2}}b.
\end{split} \nonumber
\end{align}
The last identity shows that
\begin{align}\label{id 3}
(\Btil_i b,b') = \frac{1}{\sqrt{2}} = (b, \Btil_{\tau(i)}b').
\end{align}
Also, we have
\begin{align}
\begin{split}
&F_{\tau(i)}(b'') \leq \vphi_{\tau(i)}+1 < \beta_i+1, \\
&B_{\tau(i)}(b'') = \beta_{\tau(i)}+\wti_i-s_i+1 = \beta_i, \\
&E_{\tau(i)}(b'') = (\vphi_i-1)+\wti_i-s_i+1 = \beta_i, \\
&\beta_i(b''_1) = \beta_i, \\
&\Etil_i b''_2 = b_2 \neq 0, \\
&\vphi_{\tau(i)}(\Etil_i b''_2) = \vphi_{\tau(i)} < \beta_i.
\end{split} \nonumber
\end{align}
This shows that
\begin{align}
\begin{split}
&\Btil_{\tau(i)} b'' = \frac{1}{\sqrt{2}} b''_1 \otimes \Etil_i b''_2 = \frac{1}{\sqrt{2}}b.
\end{split} \nonumber
\end{align}
The last identity shows that
\begin{align}\label{id 4}
(\Btil_i b,b'') = \frac{1}{\sqrt{2}} = (b, \Btil_{\tau(i)}b'').
\end{align}
Identities \eqref{id 3} and \eqref{id 4} confirm Definition \ref{Def: icrystal} \eqref{Def: icrystal 2.5}.
\item When $B_i(b) \leq E_i(b) = F_i(b)$ and $\vphi_i(\Etil_{\tau(i)} b_2) = \vphi_i$. In this case, we have
\begin{align}
\begin{split}
&\beta_i(b) = E_i(b)+\wti_i-s_i-\wt_{\tau(i)} = \vep_{\tau(i)}, \\
&\beta_{\tau(i)}(b)+\wti_i(b)-s_i = F_i(b)+\wti_i-s_i-\wt_{\tau(i)} = \beta_i(b), \\
&\Btil_i b = \frac{1}{\sqrt{2}} b_1 \otimes \Etil_{\tau(i)} b_2 \notin \clB.
\end{split} \nonumber
\end{align}
This confirms Definition \ref{Def: icrystal} \eqref{Def: icrystal 2.6}, \eqref{Def: icrystal 5b}, and \eqref{Def: icrystal 5c}.

Let $b' = b'_1 \otimes b'_2 \in \clB$ be such that $(\Btil_i b,b') \neq 0$. Then, we have $b' = b_1 \otimes \Etil_{\tau(i)} b_2$. By our assumption, we see that
\begin{align}
\begin{split}
&F_i(b') = \vphi_i(\Etil_{\tau(i)} b_2) = \vphi_i, \\
&B_i(b') = B_i(b) \leq \vphi_i, \\
&E_i(b') = E_i(b)+1 = \vphi_i+1,
\end{split} \nonumber
\end{align}
and hence,
\begin{align}
\begin{split}
&\beta_i(b') = E_i(b')+\wti_i-s_i-\wt_{\tau(i)}(b'_2) = \vep_{\tau(i)}(b'_2) = \vep_{\tau(i)}-1 = \beta_i(b)-1, \\
&\beta_{\tau(i)}(b')+\wti_i(b')-s_i = (E_i(b')-1)+\wti_i-s_i-\wt_{\tau(i)}(b'_2) \neq \beta_i(b').
\end{split} \nonumber
\end{align}
This confirms Definition \ref{Def: icrystal} \eqref{Def: icrystal 5d}.

Now, we compute as
\begin{align}
\begin{split}
&F_{\tau(i)}(b') = \vphi_{\tau(i)}(\Etil_{\tau(i)} b_2) = \vphi_{\tau(i)}+1, \\
&B_{\tau(i)}(b') = \beta_{\tau(i)}+\wti_i-s_i+1 \leq B_i(b)+\wti_i-s_i+1 \leq \vphi_{\tau(i)}+1, \\
&E_{\tau(i)}(b') = \vphi_i(\Etil_{\tau(i)} b_2)+\wti_i-s_i+1 = \vphi_i+\wti_i-s_i+1 = \vphi_{\tau(i)}+1, \\
&\beta_i(b'_1) = \beta_i \leq \vphi_{\tau(i)}, \\
&\Ftil_{\tau(i)} b'_2 = b_2 \neq 0, \\
&\vphi_i(\Ftil_{\tau(i)} b'_2) = \vphi_i = \vphi_i(b'_2).
\end{split} \nonumber
\end{align}
Since $\vphi_i(\Ftil_{\tau(i)} b'_2) = \vphi_i(b'_2)$ is equivalent to $\vphi_{\tau(i)}(\Etil_i b'_2) = \vphi_{\tau(i)}(b'_2)-1$ By Lemma \ref{Deduction from condition S's}, this implies that
\begin{align}
\begin{split}
&\Btil_{\tau(i)} b' = \frac{1}{\sqrt{2}}(b'_1 \otimes \Etil_i b'_2 + b_2),
\end{split} \nonumber
\end{align}
which shows that
$$
(\Btil_i b,b') = \frac{1}{\sqrt{2}} = (b, \Btil_{\tau(i)}b').
$$
This confirms Definition \ref{Def: icrystal} \eqref{Def: icrystal 2.5}.
\item When $F_i(b) \leq E_i(b) = B_i(b) = \beta_{\tau(i)}$ and $\vphi_i(\Etil_{\tau(i)} b_2) < E_i(b)$. Since $B_i(b) = \beta_{\tau(i)}$, we have
$$
\Btil_{\tau(i)} b_1 \in \clB_1, \qu \beta_{\tau(i)}(\Btil_{\tau(i)} b_1) = \beta_{\tau(i)}-1, \qu \beta_i(\Btil_{\tau(i)} b_1) = \beta_i+2
$$
if $\Btil_{\tau(i)} b_1 \neq 0$. Also, we have
\begin{align}
\begin{split}
&\beta_i(b) = E_i(b)+\wti_i-s_i-\wt_{\tau(i)} = \vep_{\tau(i)}, \\
&\beta_{\tau(i)}(b)+\wti_i-s_i = B_i(b)+\wti_i-s_i-\wt_{\tau(i)} = \beta_i(b), \\
&\Btil_i b = \frac{1}{\sqrt{2}} b_1 \otimes \Etil_{\tau(i)} b_2 \notin \clB.
\end{split} \nonumber
\end{align}
This confirms Definition \ref{Def: icrystal} \eqref{Def: icrystal 2.6}, \eqref{Def: icrystal 5b}, and \eqref{Def: icrystal 5c}.

Let $b' = b'_1 \otimes b'_2 \in \clB$ be such that $(\Btil_i b,b') \neq 0$. Then, we have $b' = b_1 \otimes \Etil_{\tau(i)} b_2$. By our assumption, we see that
\begin{align}
\begin{split}
&F_i(b') = \vphi_i(\Etil_{\tau(i)} b_2) < E_i(b), \\
&B_i(b') = B_i(b) = E_i(b), \\
&E_i(b') = E_i(b)+1,
\end{split} \nonumber
\end{align}
and hence,
\begin{align}
\begin{split}
&\beta_i(b') = E_i(b')+\wti_i-s_i-\wt_{\tau(i)}(b'_2) = \vep_{\tau(i)}(b'_2) = \vep_{\tau(i)}-1 = \beta_i(b)-1, \\
&\beta_{\tau(i)}(b')+\wti_i(b')-s_i = (E_i(b')-1)+\wti_i-s_i-\wt_{\tau(i)}(b'_2) \neq \beta_i(b').
\end{split} \nonumber
\end{align}
This confirms Definition \ref{Def: icrystal} \eqref{Def: icrystal 5d}.

Now, we compute as
\begin{align}
\begin{split}
&F_{\tau(i)}(b') = \vphi_{\tau(i)}+1, \\
&B_{\tau(i)}(b') = \beta_{\tau(i)}+\wti_i-s_i+1 = \vphi_{\tau(i)}+1, \\
&E_{\tau(i)}(b') = \vphi_i(\Etil_{\tau(i)} b_2)+\wti_i-s_i+1 < \vphi_{\tau(i)}+1, \\
&\beta_i(b'_1) = \beta_i = \vphi_{\tau(i)}.
\end{split} \nonumber
\end{align}
This implies that
\begin{align}
\begin{split}
&\Btil_{\tau(i)} b' = \frac{1}{\sqrt{2}}(\Btil_{\tau(i)} b_1 \otimes \Etil_{\tau(i)} b_2 + b),
\end{split} \nonumber
\end{align}
which shows that
$$
(\Btil_i b,b') = \frac{1}{\sqrt{2}} = (b, \Btil_{\tau(i)}b').
$$
This confirms Definition \ref{Def: icrystal} \eqref{Def: icrystal 2.5}.
\item When $B_i(b) \leq E_i(b) = F_i(b) > \beta_{\tau(i)}$ and $\vphi_i(\Etil_{\tau(i)} b_2) = \vphi_i-1$. In this case, we have
\begin{align}
\begin{split}
&\beta_i(b) = E_i(b)+\wti_i-s_i-\wt_{\tau(i)} = \vep_{\tau(i)}, \\
&\beta_{\tau(i)}(b)+\wti_i(b)-s_i = F_i(b)+\wti_i-s_i-\wt_{\tau(i)} = \beta_i(b), \\
&\Btil_i b = \frac{1}{\sqrt{2}}(b_1 \otimes \Etil_{\tau(i)} b_2 + b_1 \otimes \Ftil_i b_2) \notin \clB.
\end{split} \nonumber
\end{align}
This confirms Definition \ref{Def: icrystal} \eqref{Def: icrystal 2.6}, \eqref{Def: icrystal 5b}, and \eqref{Def: icrystal 5c}.

Let $b' = b'_1 \otimes b'_2 \in \clB$ be such that $(\Btil_i b,b') \neq 0$. Then, we have either $b' = b_1 \otimes \Etil_{\tau(i)} b_2$ or $b' = b'' = b''_1 \otimes b''_2 = b_1 \otimes \Ftil_i b_2$. By our assumption, we see that
\begin{align}
\begin{split}
&F_i(b') = \vphi_i(\Etil_{\tau(i)} b_2) = \vphi_i-1, \\
&B_i(b') = B_i(b) \leq \vphi_i, \\
&E_i(b') = E_i(b)+1 \leq \vphi_i+1,
\end{split} \nonumber
\end{align}
and hence,
\begin{align}\label{id 5}
\begin{split}
&\beta_i(b') = E_i(b')+\wti_i-s_i-\wt_{\tau(i)}(b'_2) = \vep_{\tau(i)}(b'_2) = \vep-1 = \beta_i(b)-1, \\
&\beta_{\tau(i)}(b')+\wti_i(b')-s_i = (E_i(b')-1)+\wti_i-s_i-\wt_{\tau(i)}(b'_2) \neq \beta_i(b').
\end{split}
\end{align}
Also, we compute as
\begin{align}
\begin{split}
&F_i(b'') = \vphi_i-1, \\
&B_i(b'') = B_i(b) \leq \vphi_i, \\
&E_i(b'') \leq E_i(b) = \vphi_i.
\end{split} \nonumber
\end{align}
This shows that
\begin{align}\label{id 6}
\begin{split}
&\beta_i(b'') = E_i(b'')+\wti_i-s_i-\wt_{\tau(i)}(b''_2) = \vep_{\tau(i)}(b''_2) = \vep_{\tau(i)}-1 = \beta_i(b)-1, \\
&\beta_{\tau(i)}(b'')+\wti_i(b'')-s_i = (E_i(b'')-1)+\wti_i-s_i-\wt_{\tau(i)}(b''_2) \neq \beta_i(b'').
\end{split}
\end{align}
Identities \eqref{id 5} and \eqref{id 6} confirm Definition \ref{Def: icrystal} \eqref{Def: icrystal 5d}.

Now, we compute as
\begin{align}
\begin{split}
&F_{\tau(i)}(b') = \vphi_{\tau(i)}+1, \\
&B_{\tau(i)}(b') = \beta_{\tau(i)}+\wti_i-s_i+1 < \vphi_{\tau(i)}+1, \\
&E_{\tau(i)}(b') = \vphi_i+\wti_i-s_i = \vphi_{\tau(i)}, \\
&\Ftil_{\tau(i)} b'_2 = b_2 \neq 0, \\
&\vphi_i(\Ftil_{\tau(i)} b'_2) = \vphi_i = \vphi_i(b'_2)+1.
\end{split} \nonumber
\end{align}
This implies that
\begin{align}
\begin{split}
&\Btil_{\tau(i)} b' = \frac{1}{\sqrt{2}} b'_1 \otimes \Ftil_{\tau(i)} b'_2 = \frac{1}{\sqrt{2}}b,
\end{split} \nonumber
\end{align}
which shows that
\begin{align}\label{id 7}
(\Btil_i b,b') = \frac{1}{\sqrt{2}} = (b, \Btil_{\tau(i)}b').
\end{align}
Also, noting that $\vphi_i(\Etil_{\tau(i)} b_2) = \vphi_i-1$ is equivalent to $\vphi_{\tau(i)}(\Ftil_i b_2) = \vphi_{\tau(i)}$, we have
\begin{align}
\begin{split}
&F_{\tau(i)}(b'') = \vphi_{\tau(i)}, \\
&B_{\tau(i)}(b'') = \beta_{\tau(i)}+\wti_i-s_i+1 < \vphi_{\tau(i)}+1, \\
&E_{\tau(i)}(b'') = \vphi_i+\wti_i-s_i = \vphi_{\tau(i)}, \\
&\Etil_i b''_2 = b_2 \neq 0, \\
&\vphi_{\tau(i)}(\Etil_i b''_2) = \vphi_{\tau(i)} = \vphi_{\tau(i)}(b''_2).
\end{split} \nonumber
\end{align}
This implies that
\begin{align}
\begin{split}
&\Btil_{\tau(i)} b'' = \frac{1}{\sqrt{2}} b''_1 \otimes \Etil_i b''_2 = \frac{1}{\sqrt{2}}b,
\end{split} \nonumber
\end{align}
which shows that
\begin{align}\label{id 8}
(\Btil_i b,b'') = \frac{1}{\sqrt{2}} = (b, \Btil_{\tau(i)}b'').
\end{align}
Identities \eqref{id 7} and \eqref{id 8} confirm Definition \ref{Def: icrystal} \eqref{Def: icrystal 2.5}.
\item When $F_i(b) > B_i(b), E_i(b)$, and $\Btil_i b = b_1 \otimes \Ftil_i b_2$. In this case, we have
\begin{align}
\begin{split}
&\beta_i(b) = F_i(b)+\wti_i-s_i-\wt_{\tau(i)} = \vphi_i-\wt_{\tau(i)}+\wti_i-s_i, \\
&\beta_{\tau(i)}(b)+\wti_i-s_i = F_i(b)+\wti_i-s_i-\wt_{\tau(i)} = \beta_i(b).
\end{split} \nonumber
\end{align}
This confirms Definition \ref{Def: icrystal} \eqref{Def: icrystal 5b} and \eqref{Def: icrystal 5c}.

Let $b' = b'_1 \otimes b'_2 \in \clB$ be such that $(\Btil_i b,b') \neq 0$. Then, we have $b' = b_1 \otimes \Ftil_i b_2 = \Btil_i b$. By our assumption, we see that
\begin{align}
\begin{split}
&F_i(b') = \vphi_i-1, \\
&B_i(b') = B_i(b) < \vphi_i, \\
&E_i(b') = \vphi_{\tau(i)}(\Ftil_i b_2)+\wti_i-s_i \leq E_i(b)+1 \leq \vphi_i.
\end{split} \nonumber
\end{align}
We claim that $E_i(b') < \vphi_i$. Otherwise, by the last inequality, we must have $E_i(b)+1 = F_i(b)$ and $\vphi_{\tau(i)}(\Ftil_i b_2) = \vphi_{\tau(i)}+1$. However, in this case, it holds that $\Btil_i b = \frac{1}{\sqrt{2}} b_1 \otimes \Ftil_i b_2$, which contradicts our assumption that $\Btil_i b = b'$. Thus, our claim follows. Then, we obtain
\begin{align}
\begin{split}
&\beta_i(b') = F_i(b')+\wti_i-s_i-\wt_{\tau(i)}(b'_2), \\
&\beta_{\tau(i)}(b')+\wti_i(b')-s_i = F_i(b')+\wti_i-s_i-\wt_{\tau(i)}(b'_2) = \beta_i(b').
\end{split} \nonumber
\end{align}
This confirms Definition \ref{Def: icrystal} \eqref{Def: icrystal 5d}.

Now, we compute as
\begin{align}
\begin{split}
&F_{\tau(i)}(b') = \vphi_{\tau(i)}(\Ftil_i b_2) \leq \vphi_{\tau(i)}+1 < \vphi_i+\wti_i-s_i+1, \\
&B_{\tau(i)}(b') = \beta_{\tau(i)}+\wti_i-s_i+1 \leq B_i(b)+\wti_i-s_i+1 < \vphi_i+\wti_i-s_i+1, \\
&E_{\tau(i)}(b') = \vphi_i+\wti_i-s_i.
\end{split} \nonumber
\end{align}
This implies that $F_{\tau(i)}(b'), B_{\tau(i)}(b') \leq E_{\tau(i)}(b')$. We claim that $\Btil_{\tau(i)} b' = b'_1 \otimes \Etil_i b'_2$. Otherwise, we have either $\Btil_{\tau(i)} b' = \frac{1}{\sqrt{2}} b'_1 \otimes \Etil_i b'_2$ or $\Btil_{\tau(i)} b' = \frac{1}{\sqrt{2}} (b'_1 \otimes \Etil_i b'_2 + b'_1 \otimes \Ftil_{\tau(i)} b'_2)$. In each case, by consideration above, we obtain $\Btil_i(b'_1 \otimes \Etil_i b'_2) = \frac{1}{\sqrt{2}} b'_1 \otimes \Ftil_i \Etil_i b'_2 = \frac{1}{\sqrt{2}} b'$, which is a contradiction because $b'_1 \otimes \Etil_i b'_2 = b$. Thus, it follows that
$$
\Btil_{\tau(i)} b' = b'_1 \otimes \Etil_i b'_2 = b.
$$
This confirms Definition \ref{Def: icrystal} \eqref{Def: icrystal 2.5} and \eqref{Def: icrystal 2.6}.
\item When $F_i(b) \leq B_i(b) > E_i(b)$, $\Btil_i b = \Btil_i b_1 \otimes b_2$, and $B_i(b) = \beta_{\tau(i)}$. In this case, we have
\begin{align}
\begin{split}
&\beta_i(b) = B_i(b)+\wti_i-s_i-\wt_{\tau(i)} = \beta_i-\wt_{\tau(i)}, \\
&\beta_{\tau(i)}(b)+\wti_i(b)-s_i = B_i(b)+\wti_i-s_i-\wt_{\tau(i)} = \beta_i(b).
\end{split} \nonumber
\end{align}
This confirms Definition \ref{Def: icrystal} \eqref{Def: icrystal 5b} and \eqref{Def: icrystal 5c}.

Let $b' = b'_1 \otimes b'_2 \in \clB$ be such that $(\Btil_i b,b') \neq 0$. Then, we have $(\Btil_i b_1,b'_1) \neq 0$ and $b'_2 = b_2$. By our assumption, we see that
\begin{align}
\begin{split}
&F_i(b') = F_i(b) \leq B_i(b), \\
&B_i(b') = \beta_i(b'_1)-(\wti_i-3)+s_i = \begin{cases}
B_i(b)+1 & \IF \beta_{\tau(i)}(b'_1) = B_i(b'), \\
B_i(b)+2 & \IF \beta_{\tau(i)}(b'_1) \neq B_i(b'), \\
\end{cases} \\
&E_i(b') = E_i(b)+3 \leq B_i(b)+2.
\end{split} \nonumber
\end{align}

Let us consider the case when $\beta_{\tau(i)}(b'_1) = B_i(b')$ and $E_i(b') < B_i(b)+2$. By Proposition \ref{basic property for a = -1}, we have $b'_1 = \Btil_i b'$, $\beta_i(b'_1) = \beta_i -2$, and hence,
\begin{align}
\begin{split}
&\beta_i(b') = B_i(b')+\wti_i(b'_1)-s_i-\wt_{\tau(i)}, \\
&\beta_{\tau(i)}(b')+\wti_i(b')-s_i = B_i(b')+\wti_i(b'_1)-s_i-\wt_{\tau(i)} = \beta_i(b').
\end{split} \nonumber
\end{align}
This confirms Definition \ref{Def: icrystal} \eqref{Def: icrystal 5d}. Also, We compute as
\begin{align}
\begin{split}
&F_{\tau(i)}(b') = \vphi_{\tau(i)} < \beta_i, \\
&B_{\tau(i)}(b') = \beta_{\tau(i)}(b'_1)+\wti_i(b'_1)-s_i+1 = \beta_i(b'_1)+1 = \beta_i-1, \\
&E_{\tau(i)}(b') = \vphi_i+(\wti_i-3)-s_i+1 \leq \beta_i-2.
\end{split} \nonumber
\end{align}
This implies that $F_{\tau(i)}(b') \leq B_{\tau(i)}(b') > E_{\tau(i)}(b')$. Now, as in the previous case, we obtain
$$
\Btil_{\tau(i)} b' = \Btil_{\tau(i)} b'_1 \otimes b'_2 = b.
$$
This confirms Definition \ref{Def: icrystal} \eqref{Def: icrystal 2.5} and \eqref{Def: icrystal 2.6}.

Next, let us consider the case when $\beta_{\tau(i)}(b'_1) = B_i(b')$ and $E_i(b') = B_i(b)+2$. In this case, we have $b'_1 = \Btil_i b'$, $\beta_i(b'_1) = \beta_i -2$, and
\begin{align}
\begin{split}
&\beta_i(b') = E_i(b')+\wti_i(b'_1)-s_i-\wt_{\tau(i)} = B_i(b)+\wti_i-s_i-\wt_{\tau(i)}-1 = \beta_i(b)-1, \\
&\beta_{\tau(i)}(b')+\wti_i(b')-s_i = (E_i(b')-1)+\wti_i(b'_1)-s_i-\wt_{\tau(i)} \neq \beta_i(b').
\end{split} \nonumber
\end{align}
This confirms Definition \ref{Def: icrystal} \eqref{Def: icrystal 5d}. We compute as
\begin{align}
\begin{split}
&F_{\tau(i)}(b') = \vphi_{\tau(i)} < \beta_i, \\
&B_{\tau(i)}(b') = \beta_{\tau(i)}(b')+\wti_i(b')-s_i+1 = \beta_i(b')+1 = \beta_i-1, \\
&E_{\tau(i)}(b') = \vphi_i+(\wti_i-3)-s_i+1 \leq \beta_i-2.
\end{split} \nonumber
\end{align}
This implies that $F_{\tau(i)}(b') \leq B_{\tau(i)}(b') > E_{\tau(i)}(b')$. Hence, we have
$$
\Btil_{\tau(i)} b' = \Btil_{\tau(i)} b'_1 \otimes b'_2 = b.
$$
This confirms Definition \ref{Def: icrystal} \eqref{Def: icrystal 2.5} and \eqref{Def: icrystal 2.6}.

Finally, let us consider the case when $\beta_{\tau(i)}(b'_1) \neq B_i(b')$. By Proposition \ref{basic property for a = -1}, we have $\beta_i(b') = \beta_i -1$, $\beta_{\tau(i)}(b') = \beta_{\tau(i)}+1$, and hence,
\begin{align}
\begin{split}
&\beta_i(b') = B_i(b')+\wti_i(b'_1)-s_i-\wt_{\tau(i)} = B_i(b)+\wti_i-s_i-\wt_{\tau(i)}-1 = \beta_i(b)-1, \\
&\beta_{\tau(i)}(b')+\wti_i(b')-s_i = (B_i(b')-1)+\wti_i(b'_1)-s_i-\wt_{\tau(i)} \neq \beta_i(b').
\end{split} \nonumber
\end{align}
This confirms Definition \ref{Def: icrystal} \eqref{Def: icrystal 5d}. We compute as
\begin{align}
\begin{split}
&F_{\tau(i)}(b') = \vphi_{\tau(i)} < \beta_i, \\
&B_{\tau(i)}(b') = \beta_{\tau(i)}(b')+\wti_i(b')-s_i+1 = \beta_i(b')+1 = \beta_i-1, \\
&E_{\tau(i)}(b') = \vphi_i+(\wti_i-3)-s_i+1 \leq \beta_i-2.
\end{split} \nonumber
\end{align}
This implies that $F_{\tau(i)}(b') \leq B_{\tau(i)}(b') > E_{\tau(i)}(b')$. Hence, we have
$$
\Btil_{\tau(i)} b' = \Btil_{\tau(i)} b'_1 \otimes b'_2,
$$
which implies that $\Btil_{\tau(i)} b' = b$ if $b' = \Btil_i b$, and that
$$
(\Btil_i b,b') = (\Btil_i b_1,b'_1) = (b_1, \Btil_{\tau(i)} b'_1) = (b, \Btil_{\tau(i)} b').
$$
This confirms Definition \ref{Def: icrystal} \eqref{Def: icrystal 2.5} and \eqref{Def: icrystal 2.6}.
\item When $F_i(b) \leq B_i(b) > E_i(b)$, $\Btil_i b = \Btil_i b_1 \otimes b_2$, and $B_i(b) \neq \beta_{\tau(i)}$. In this case, we have
\begin{align}
\begin{split}
&\beta_i(b) = B_i(b)+\wti_i-s_i-\wt_{\tau(i)}, \\
&\beta_{\tau(i)}(b)+\wti_i(b)-s_i = \begin{cases}
(B_i(b)-1)+\wti_i-s_i-\wt_{\tau(i)} & \IF F_i(b) < B_i(b), \\
B_i(b)+\wti_i-s_i-\wt_{\tau(i)} & \IF F_i(b) = B_i(b).
\end{cases}
\end{split} \nonumber
\end{align}
This confirms Definition \ref{Def: icrystal} \eqref{Def: icrystal 5b}.

Let $b' = b'_1 \otimes b'_2 \in \clB$ be such that $(\Btil_i b,b') \neq 0$. Then, we have $(\Btil_i b_1,b'_1) \neq 0$ and $b'_2 = b_2$. Since $B_i(b) \neq \beta_{\tau(i)}$, it follows from Proposition \ref{basic property for a = -1} that $b'_1 = \Btil_i b_1$, $\beta_i(b'_1) \neq B_i(b')$, and $\beta_i(b'_1) = \beta_i-1$. By our assumption, we see that
\begin{align}
\begin{split}
&F_i(b') = F_i(b) \leq B_i(b), \\
&B_i(b') = \beta_i(b'_1)-(\wti_i-3)+s_i = B_i(b)+2, \\
&E_i(b') = E_i(b)+3 \leq B_i(b)+2.
\end{split} \nonumber
\end{align}
This implies that
\begin{align}
\begin{split}
&\beta_i(b') = B_i(b')+\wti_i(b'_1)-s_i-\wt_{\tau(i)} = \beta_i(b)-1, \\
&\beta_{\tau(i)}(b')+\wti_i(b')-s_i = (B_i(b')-1)+\wti_i(b'_1)-s_i-\wt_{\tau(i)} \neq \beta_i(b').
\end{split} \nonumber
\end{align}
This confirms Definition \ref{Def: icrystal} \eqref{Def: icrystal 5c} and \eqref{Def: icrystal 5d} (note that $\Btil_i b = \Btil_i b_1 \otimes b_2 = b' \in \clB$).

We compute as
\begin{align}
\begin{split}
&F_{\tau(i)}(b') = \vphi_{\tau(i)} < \beta_i, \\
&B_{\tau(i)}(b') = \beta_{\tau(i)}(b'_1)+\wti_i(b'_1)-s_i+1 = \beta_i(b'_1) = \beta_i-1, \\
&E_{\tau(i)}(b') = \vphi_i+(\wti_i-3)-s_i+1 \leq \beta_i-2.
\end{split} \nonumber
\end{align}
This implies that $F_{\tau(i)}(b') \leq B_{\tau(i)}(b') > E_{\tau(i)}(b')$. Hence, we have
$$
\Btil_{\tau(i)} b' = \Btil_{\tau(i)} b'_1 \otimes b'_2 = b.
$$
This confirms Definition \ref{Def: icrystal} \eqref{Def: icrystal 2.5} and \eqref{Def: icrystal 2.6}.
\item When $F_i(b), B_i(b) \leq E_i(b)$ and $\Btil_i b = b_1 \otimes \Etil_{\tau(i)} b_2$. In this case, we have
\begin{align}
\begin{split}
&\beta_i(b) = E_i(b)+\wti_i-s_i-\wt_{\tau(i)} = \vep_{\tau(i)}.
\end{split} \nonumber
\end{align}

Let $b' = b'_1 \otimes b'_2 \in \clB$ be such that $(\Btil_i b,b') \neq 0$. Then, we have $b' = b_1 \otimes \Etil_{\tau(i)} b_2 = \Btil_i b$. By our assumption, we see that
\begin{align}
\begin{split}
&F_i(b') = \vphi_i(\Etil_{\tau(i)} b_2) \leq \vphi_i \leq E_i(b), \\
&B_i(b') = B_i(b) \leq E_i(b), \\
&E_i(b') = E_i(b)+1.
\end{split} \nonumber
\end{align}
This implies that
\begin{align}
\begin{split}
&\beta_i(b') = E_i(b')+\wti_i-s_i-\wt_{\tau(i)}(b'_2) = \vep_{\tau(i)}(b'_2) = \beta_i(b)-1, \\
&\beta_{\tau(i)}(b')+\wti_i(b')-s_i = (E_i(b')-1)+\wti_i-s_i-\wt_{\tau(i)}(b'_2) \neq \beta_i(b').
\end{split} \nonumber
\end{align}
This confirms Definition \ref{Def: icrystal} \eqref{Def: icrystal 5c} and \eqref{Def: icrystal 5d}.

We compute as
\begin{align}
\begin{split}
&F_{\tau(i)}(b') = \vphi_{\tau(i)}+1, \\
&B_{\tau(i)}(b') = \beta_{\tau(i)}+\wti_i-s_i+1 \leq \beta_i+1 \leq \vphi_{\tau(i)}+1, \\
&E_{\tau(i)}(b') = \vphi_i(b'_2)+\wti_i-s_i+1 \leq \vphi_i+\wti_i-s_i+1 \leq \vphi_{\tau(i)}+1.
\end{split} \nonumber
\end{align}
We claim that $F_{\tau(i)}(b') > B_{\tau(i)}(b'), E_{\tau(i)}(b')$. Assume contrary that $E_{\tau(i)}(b') = F_{\tau(i)}(b')$. In this case, we must have $\vphi_i(\Etil_{\tau(i)} b_2) = \vphi_i$ and $F_i(b) = E_i(b)$. This implies that $\Btil_i b = \frac{1}{\sqrt{2}} b_1 \otimes \Etil_{\tau(i)}b_2$, which contradicts our assumption that $\Btil_i b = b'$. Next, assume contrary that $B_{\tau(i)}(b') = F_{\tau(i)}(b')$. In this case, we must have $B_i(b) = \beta_{\tau(i)}$, $B_i(b) = E_i(b)$. If $\vphi_i(\Etil_{\tau(i)} b_2) < E_i(b)$, then it follows that $\Btil_i b = \frac{1}{\sqrt{2}} b_1 \otimes \Etil_{\tau(i)} b_2$, which causes a contradiction again. Hence, we obtain $\vphi_i(\Etil_{\tau(i)} b_2) = E_i(b)$. Since $\vphi_i(\Etil_{\tau(i)} b_2) \leq \vphi_i \leq E_i(b)$, we must have $\vphi_i(\Etil_{\tau(i)} b_2) = \vphi_i$ and $F_i(b) = E_i(b)$. This implies that $\Btil_i b = \frac{1}{\sqrt{2}} b_1 \otimes \Etil_{\tau(i)} b_2$, which is a contradiction. Thus, our claim follows. Consequently, we obtain
$$
\Btil_{\tau(i)} b' = b'_1 \otimes \Ftil_i b'_2 = b.
$$
This confirms Definition \ref{Def: icrystal} \eqref{Def: icrystal 2.5} and \eqref{Def: icrystal 2.6}.
\end{enumerate}
Now, we have exhausted all the cases. Hence, the proof completes.

\subsection{Proof of Proposition \ref{associativity for icrystal}}\label{Subsection: proof of associativity}
Let $b_i \in \clB_i$, $i = 1,2,3$, and set $b := b_1 \otimes (b_2 \otimes b_3)$, $b' := (b_1 \otimes b_2) \otimes b_3$. It suffices to show that $\beta_i(b) = \beta_i(b')$ and $\Btil_i b = \Btil_i b'$ for all $i \in I$.

For a later use, we note the following:
\begin{align}
\begin{split}
&\vphi_i(b_2 \otimes b_3) = \max(\vphi_i(b_3)+\wt_i(b_2),\vphi_i(b_2)), \\
&\Ftil_i(b_2 \otimes b_3) = \begin{cases}
b_2 \otimes \Ftil_i b_3 & \IF \vphi_i(b_3)+\wt_i(b_2) > \vphi_i(b_2), \\
\Ftil_i b_2 \otimes b_3 & \IF \vphi_i(b_3)+\wt_i(b_2) \leq \vphi_i(b_2),
\end{cases} \\
&\vep_{\tau(i)}(b_2 \otimes b_3) = \max(\vphi_{\tau(i)}(b_2), \vphi_{\tau(i)}(b_3)+\wt_{\tau(i)}(b_2))-\wt_{\tau(i)}(b_2)-\wt_{\tau(i)}(b_3) \\
&\Etil_{\tau(i)}(b_2 \otimes b_3) = \begin{cases}
\Etil_{\tau(i)} b_2 \otimes b_3 & \IF \vphi_{\tau(i)}(b_2) > \vphi_{\tau(i)}(b_3)+\wt_{\tau(i)}(b_2), \\
b_2 \otimes \Etil_{\tau(i)} b_3 & \IF \vphi_{\tau(i)}(b_2) \leq \vphi_{\tau(i)}(b_3)+\wt_{\tau(i)}(b_2).
\end{cases}
\end{split} \nonumber
\end{align}
\subsubsection{When $a_{i,\tau(i)} = 2$}
Setting
\begin{align}
\begin{split}
&A_1 := \vphi_i(b_3)+\wt_i(b_2)+\delta_{\ol{\beta_i(b_1)+1}, \ol{\vphi_i(b_3)+\wt_i(b_2)}}, \\
&A_2 := \vphi_i(b_2) + \delta_{\ol{\beta_i(b_1)+1}, \ol{\vphi_i(b_2)}}, \\
&A_3 := \beta_i(b_1), \\
&A_4 := \vphi_i(b_2), \\
&A_5 := \vphi_i(b_3)+\wt_i(b_2),
\end{split} \nonumber
\end{align}
we see that
\begin{align}
\begin{split}
&F_i(b) = \max(A_1,A_2), \qu B_i(b) = A_3, \qu E_i(b) = \max(A_4,A_5), \\
&F_i(b') = A_1-\wt_i(b_2), \qu B_i(b') = \max(A_2,A_3,A_4)-\wt_i(b_2), \qu E_i(b') = A_5-\wt_i(b_2), \\
&F_i(b_1 \otimes b_2) = A_2, \qu B_i(b_1 \otimes b_2) = A_3, \qu E_i(b_1 \otimes b_2) = A_4.
\end{split} \nonumber
\end{align}
Therefore, we compute as
\begin{align}
\begin{split}
\beta_i(b) &= \max(\max(A_1,A_2),A_3,\max(A_4,A_5)) - \wt_i(b_2 \otimes b_3) \\
&= \max(A_1,A_2,A_3,A_4,A_5)-\wt_i(b_2)-\wt_i(b_3), \\
\beta_i(b') &= \max(A_1-\wt_i(b_2),\max(A_2,A_3,A_4)-\wt_i(b_2),A_5-\wt_i(b_2))-\wt_i(b_3) \\
&= \max(A_1,A_2,A_3,A_4,A_5)-\wt_i(b_2)-\wt_i(b_3).
\end{split} \nonumber
\end{align}
This implies that $\beta_i(b) = \beta_i(b')$. Also, we compute as
\begin{align}
\begin{split}
\Btil_i b &= \begin{cases}
b_1 \otimes \Ftil_i(b_2 \otimes b_3) & \IF \max(A_1,A_2) > A_3,\max(A_4,A_5), \\
\Btil_i b_1 \otimes (b_2 \otimes b_3) & \IF \max(A_1,A_2) \leq A_3 > \max(A_4,A_5), \\
b_1 \otimes \Etil_i(b_2 \otimes b_3) & \IF \max(A_1,A_2), A_3 \leq \max(A_4,A_5),
\end{cases} \\
&= \begin{cases}
b_1 \otimes (b_2 \otimes \Ftil_i b_3) & \IF A_1 > A_2,A_3,A_4,A_5, \\
b_1 \otimes (\Ftil_i b_2 \otimes b_3) & \IF A_1 \leq A_2 > A_3,A_4,A_5, \\
\Btil_i b_1 \otimes (b_2 \otimes b_3) & \IF A_1,A_2 \leq A_3 > A_4,A_5, \\
b_1 \otimes (\Etil_i b_2 \otimes b_3) & \IF A_1,A_2,A_3 \leq A_4 > A_5, \\
b_1 \otimes (b_2 \otimes \Etil_i b_3) & \IF A_1,A_2,A_3,A_4 \leq A_5,
\end{cases} \\
\Btil_i b' &= \begin{cases}
(b_1 \otimes b_2) \otimes \Ftil_ib_3 & \IF A_1 > \max(A_2,A_3,A_4),A_5, \\
\Btil_i(b_1 \otimes b_2) \otimes b_3 & \IF A_1 \leq \max(A_2,A_3,A_4) > A_5, \\
(b_1 \otimes b_2) \otimes \Etil_i b_3 & \IF A_1, \max(A_2,A_3,A_4) \leq A_5,
\end{cases} \\
&= \begin{cases}
(b_1 \otimes b_2) \otimes \Ftil_i b_3 & \IF A_1 > A_2,A_3,A_4,A_5, \\
(b_1 \otimes \Ftil_i b_2) \otimes b_3 & \IF A_1 \leq A_2 > A_3,A_4,A_5, \\
(\Btil_i b_1 \otimes b_2) \otimes b_3 & \IF A_1,A_2 \leq A_3 > A_4,A_5, \\
(b_1 \otimes \Etil_i b_2) \otimes b_3 & \IF A_1,A_2,A_3 \leq A_4 > A_5, \\
(b_1 \otimes b_2) \otimes \Etil_i b_3 & \IF A_1,A_2,A_3,A_4 \leq A_5.
\end{cases}
\end{split} \nonumber
\end{align}
Thus, we obtain $\Btil_i b = \Btil_i b'$.

\subsubsection{When $a_{i,\tau(i)} = 0$}
Setting
\begin{align}
\begin{split}
&A_1 := \vphi_i(b_3)+\wt_i(b_2), \\
&A_2 := \vphi_i(b_2), \\
&A_3 := \beta_i(b_1)-\wti_i(b_1), \\
&A_4 := \vphi_{\tau(i)}(b_2)-\wti_i(b_1), \\
&A_5 := \vphi_{\tau(i)}(b_3)+\wt_{\tau(i)}(b_2)-\wti_i(b_1),
\end{split} \nonumber
\end{align}
we see that
\begin{align}
\begin{split}
&F_i(b) = \max(A_1,A_2), \qu B_i(b) = A_3, \qu E_i(b) = \max(A_4,A_5), \\
&F_i(b') = A_1-\wt_i(b_2), \qu B_i(b') = \max(A_2,A_3,A_4)-\wt_i(b_2), \qu E_i(b') = A_5-\wt_i(b_2), \\
&F_i(b_1 \otimes b_2) = A_2, \qu B_i(b_1 \otimes b_2) = A_3, \qu E_i(b_1 \otimes b_2) = A_4.
\end{split} \nonumber
\end{align}
Therefore, we can compute as before to obtain $\beta_i(b) = \beta_i(b')$ and $\Btil_i b = \Btil_i b'$.

\subsubsection{When $a_{i,\tau(i)} = -1$}
Setting
\begin{align}
\begin{split}
&A_1 := \vphi_i(b_3)+\wt_i(b_2), \\
&A_2 := \vphi_i(b_2), \\
&A_3 := \beta_i(b_1)-\wti_i(b_1)+s_i, \\
&A_4 := \vphi_{\tau(i)}(b_2)-\wti_i(b_1)+s_i, \\
&A_5 := \vphi_{\tau(i)}(b_3)+\wt_{\tau(i)}(b_2)-\wti_i(b_1)+s_i,
\end{split} \nonumber
\end{align}
we see that
\begin{align}
\begin{split}
&F_i(b) = \max(A_1,A_2), \qu B_i(b) = A_3, \qu E_i(b) = \max(A_4,A_5), \\
&F_i(b') = A_1-\wt_i(b_2), \qu B_i(b') = \max(A_2,A_3,A_4)-\wt_i(b_2), \qu E_i(b') = A_5-\wt_i(b_2), \\
&F_i(b_1 \otimes b_2) = A_2, \qu B_i(b_1 \otimes b_2) = A_3, \qu E_i(b_1 \otimes b_2) = A_4.
\end{split} \nonumber
\end{align}
Hence, we can compute as before to obtain $\beta_i(b) = \beta_i(b')$.

Note also that
\begin{align}
\begin{split}
\beta_i(b_1 \otimes b_2) &= \max(A_2,A_3,A_4)+\wti_i(b_1)-s_i-\wt_{\tau(i)}(b_2), \\
\beta_{\tau(i)}(b_1 \otimes b_2) &= \max(A_4-1, \beta_{\tau(i)}(b_1), A_2)-\wt_i(b_2) \\
&= \begin{cases}
\max(A_4-1,A_3,A_2)-\wt_i(b_2) & \IF \beta_{\tau(i)}(b_1) = A_3, \\
\max(A_4-1,A_3-1,A_2)-\wt_i(b_2) & \IF \beta_{\tau(i)}(b_1) \neq A_3.
\end{cases}
\end{split} \nonumber
\end{align}

\begin{enumerate}
\item When $A_1 > A_2,A_3,A_4,A_5$. In this case, we have
$$
\Ftil_i(b_2 \otimes b_3) = b_2 \otimes \Ftil_i b_3,
$$
and hence,
$$
\Btil_i b = \begin{cases}
\frac{1}{\sqrt{2}} b_1 \otimes (b_2 \otimes \Ftil_i b_3) & \IF A_1 = \max(A_4,A_5)+1, \\
&\AND \vphi_{\tau(i)}(b_2 \otimes \Ftil_i b_3) = \vphi_{\tau(i)}(b_2 \otimes b_3)+1, \\
b_1 \otimes (b_2 \otimes \Ftil_i b_3) & \OW.
\end{cases}
$$
Since $\vphi_{\tau(i)}(b_2 \otimes \Ftil_i b_3) = \max(\vphi_{\tau(i)}(b_2), \vphi_{\tau(i)}(\Ftil_ib_3)+\wt_{\tau(i)}(b_2))$ and $\vphi_{\tau(i)}(b_2 \otimes b_3) = \max(A_4,A_5)$, we have $\vphi_{\tau(i)}(b_2 \otimes \Ftil_i b_3) = \vphi_{\tau(i)}(b_2 \otimes b_3)+1$ if and only if $A_4 \leq A_5$ and $\vphi_{\tau(i)}(\Ftil_i b_3) = \vphi_{\tau(i)}(b_3)+1$. Therefore, we obtain
$$
\Btil_i b = \begin{cases}
\frac{1}{\sqrt{2}} b_1 \otimes (b_2 \otimes \Ftil_i b_3) & \IF A_1 = A_5+1 \AND \vphi_{\tau(i)}(\Ftil_i b_3) = \vphi_{\tau(i)}(b_3)+1, \\
b_1 \otimes (b_2 \otimes \Ftil_i b_3) & \OW.
\end{cases}
$$
Note that $A_1 = A_5+1$ implies $A_4 \leq A_1-1 = A_5$.

On the other hand, we have
\begin{align}
\begin{split}
&\Btil_i b' = \begin{cases}
\frac{1}{\sqrt{2}}(b_1 \otimes b_2) \otimes \Ftil_i b_3 & \IF A_1 = A_5+1 \AND \vphi_{\tau(i)}(\Ftil_i b_3) = \vphi_{\tau(i)}(b_3)+1, \\
(b_1 \otimes b_2) \otimes \Ftil_i b_3 & \OW.
\end{cases}
\end{split} \nonumber
\end{align}
Thus we obtain
$$
\Btil_i b = \Btil_i b'.
$$
\item When $A_1 \leq A_2 > A_3,A_4,A_5$ and $A_4 > A_5$. In this case, we have
\begin{align}
\begin{split}
&\Ftil_i(b_2 \otimes b_3) = \Ftil_i b_2 \otimes b_3,
\end{split} \nonumber
\end{align}
and hence,
$$
\Btil_i b = \begin{cases}
\frac{1}{\sqrt{2}} b_1 \otimes (\Ftil_i b_2 \otimes b_3) & \IF A_2=A_4+1, \\
& \AND \vphi_{\tau(i)}(\Ftil_i b_2 \otimes b_3) = \vphi_{\tau(i)}(b_2 \otimes b_3)+1, \\
b_1 \otimes (\Ftil_i b_2 \otimes b_3) & \OW.
\end{cases}
$$
Since $\vphi_{\tau(i)}(\Ftil_i b_2 \otimes b_3) = \vphi_{\tau(i)}(\Ftil_i b_2)$ and $\vphi_{\tau(i)}(b_2 \otimes b_3) = \vphi_{\tau(i)}(b_2)$, we have $\vphi_{\tau(i)}(\Ftil_i b_2 \otimes b_3) = \vphi_{\tau(i)}(b_2 \otimes b_3)+1$ if and only if $\vphi_{\tau(i)}(\Ftil_i b_2) = \vphi_{\tau(i)}(b_2)+1$. Therefore, we obtain
$$
\Btil_i b = \begin{cases}
\frac{1}{\sqrt{2}} b_1 \otimes (\Ftil_i b_2 \otimes b_3) & \IF A_2=A_4+1 \AND \vphi_{\tau(i)}(\Ftil_i b_2) = \vphi_{\tau(i)}(b_2)+1, \\
b_1 \otimes (\Ftil_i b_2 \otimes b_3) & \OW.
\end{cases}
$$

On the other hand, since $A_2 > A_3,A_4$, we have
$$
\beta_{\tau(i)}(b_1 \otimes b_2) = A_2-\wt_i(b_2) = B_i(b'),
$$
and hence,
\begin{align}
\begin{split}
&\Btil_i b' = \begin{cases}
\frac{1}{\sqrt{2}} \Btil_i(b_1 \otimes b_2) \otimes b_3 & \IF A_2=A_5+1, \\
& \AND \beta_i(\Btil_i(b_1 \otimes b_2)) = \beta_i(b_1 \otimes b_2)-2, \\
\Btil_i(b_1 \otimes b_2) \otimes b_3 & \OW.
\end{cases}
\end{split} \nonumber
\end{align}
However, since $A_2 > A_4 > A_5$, it never happens that $A_2 = A_5+1$. Therefore, we obtain
\begin{align}
\begin{split}
\Btil_i b' &= \Btil_i(b_1 \otimes b_2) \otimes b_3 \\
&= \begin{cases}
\frac{1}{\sqrt{2}}(b_1 \otimes \Ftil_i b_2) \otimes b_3 & \IF A_2 = A_4+1 \AND \vphi_{\tau(i)}(\Ftil_i b_2) = \vphi_{\tau(i)}(b_2)+1, \\
(b_1 \otimes \Ftil_i b_2) \otimes b_3 & \OW. 
\end{cases}
\end{split} \nonumber
\end{align}
Thus, we conclude
$$
\Btil_i b = \Btil_i b'.
$$
\item When $A_1 \leq A_2 > A_3,A_4,A_5$ and $A_4 \leq A_5$. In this case, we have
\begin{align}
\begin{split}
&\Ftil_i(b_2 \otimes b_3) = \Ftil_i b_2 \otimes b_3,
\end{split} \nonumber
\end{align}
and hence,
$$
\Btil_i b = \begin{cases}
\frac{1}{\sqrt{2}} b_1 \otimes (\Ftil_i b_2 \otimes b_3) & \IF A_2=A_5+1, \\
& \AND \vphi_{\tau(i)}(\Ftil_i b_2 \otimes b_3) = \vphi_{\tau(i)}(b_2 \otimes b_3)+1, \\
b_1 \otimes (\Ftil_i b_2 \otimes b_3) & \OW.
\end{cases}
$$
Since $\vphi_{\tau(i)}(\Ftil_i b_2 \otimes b_3) = \vphi_{\tau(i)}(b_3)+\wt_{\tau(i)}(b_2)+1$ and $\vphi_{\tau(i)}(b_2 \otimes b_3) = \vphi_{\tau(i)}(b_3)+\wt_{\tau(i)}(b_2)$, we always have $\vphi_{\tau(i)}(\Ftil_i b_2 \otimes b_3) = \vphi_{\tau(i)}(b_2 \otimes b_3)+1$. Therefore, we obtain
$$
\Btil_i b = \begin{cases}
\frac{1}{\sqrt{2}} b_1 \otimes (\Ftil_i b_2 \otimes b_3) & \IF A_2=A_5+1, \\
b_1 \otimes (\Ftil_i b_2 \otimes b_3) & \OW.
\end{cases}
$$

On the other hand, since $A_2 > A_3,A_4$, we have
$$
\beta_{\tau(i)}(b_1 \otimes b_2) = A_2-\wt_i(b_2) = B_i(b'),
$$
and hence,
\begin{align}
\begin{split}
&\Btil_i b' = \begin{cases}
\frac{1}{\sqrt{2}} \Btil_i(b_1 \otimes b_2) \otimes b_3 & \IF A_2=A_5+1, \\
& \AND \beta_i(\Btil_i(b_1 \otimes b_2)) = \beta_i(b_1 \otimes b_2)-2, \\
\Btil_i(b_1 \otimes b_2) \otimes b_3 & \OW.
\end{cases}
\end{split} \nonumber
\end{align}
Since $A_2 > A_3,A_4$, we have $\Btil_i(b_1 \otimes b_2) \in \clB_1 \otimes \clB_2$ if and only if $\Btil_i(b_1 \otimes b_2) = b_1 \otimes \Ftil_i b_2$. In this case, we always have $\beta_i(\Btil_i(b_1 \otimes b_2)) = \beta_i(b_1 \otimes b_2)-2$. Otherwise, we have either $\Ftil_i b_2 = 0$ or $A_2=A_4+1$ and $\vphi_{\tau(i)}(\Ftil_i b_2) = \vphi_{\tau(i)}(b_2)+1$. Noting that $A_2 = A_4+1$ implies $A_2 = A_5+1$, we obtain
\begin{align}
\begin{split}
\Btil_i b' &= \begin{cases}
\frac{1}{\sqrt{2}}(b_1 \otimes \Ftil_i b_2) \otimes b_3 & \IF A_2 = A_5+1, \\
(b_1 \otimes \Ftil_i b_2) \otimes b_3 & \OW. 
\end{cases}
\end{split} \nonumber
\end{align}
Thus, we conclude
$$
\Btil_i b = \Btil_i b'.
$$
\item When $A_1, A_2 \leq A_3 > A_4,A_5$ and $\beta_{\tau(i)}(b_1) = A_3$. In this case, we have
$$
\Btil_i b = \begin{cases}
\frac{1}{\sqrt{2}} \Btil_i b_1 \otimes (b_2 \otimes b_3) & \IF A_3 = \max(A_4,A_5)+1,\\
& \AND \beta_i(\Btil_i b_1) = \beta_i(b_1)-2, \\
\Btil_i b_1 \otimes (b_2 \otimes b_3) & \OW,
\end{cases}
$$

On the other hand, since $A_2 \leq A_3 > A_4$ and $\beta_{\tau(i)}(b_1)=A_3$, we have
$$
\beta_{\tau(i)}(b_1 \otimes b_2) = A_3-\wt_i(b_2) = B_i(b'),
$$
and hence,
$$
\Btil_i b' = \begin{cases}
\frac{1}{\sqrt{2}}(\Btil_i(b_1 \otimes b_2) \otimes b_3) & \IF A_3 = A_5+1, \\
&\AND \beta_i(\Btil_i(b_1 \otimes b_2)) = \beta_i(b_1 \otimes b_2)-2, \\
\Btil_i(b_1 \otimes b_2) \otimes b_3 & \OW,
\end{cases}
$$
Since $A_2 \leq A_3 > A_4$, we have $\Btil_i(b_1 \otimes b_2) \in \clB_1 \otimes \clB_2$ if and only if $\Btil_i(b_1 \otimes b_2) = \Btil_i b_1 \otimes b_2 \in \clB_1 \otimes \clB_2$. In this case, the equality $\beta_i(\Btil_i b_1 \otimes b_2) = \beta_i(b_1 \otimes b_2)-2$ is equivalent to that $\beta_i(\Btil_i b_1) = \beta_i(b_1)-2$ and $A_3+1 \geq A_4+3$ since we have
\begin{align}\label{calc}
\begin{split}
&F_i(\Btil_i b_1 \otimes b_2) = F_i(b_1 \otimes b_2), \\
&B_i(\Btil_i b_1 \otimes b_2) = B_i(b_1 \otimes b_2)+(\beta_i(\Btil_i b_1) - \beta_i(b_1))+3, \\
&E_i(\Btil_ib_1 \otimes b_2)=E_i(b_1 \otimes b_2)+3.
\end{split}
\end{align}
Now, consider the case when $A_4 > A_5$. In this case, it never happens that $A_3 = A_5+1$. Hence, we obtain
\begin{align}
\begin{split}
\Btil_i b' &= \Btil_i(b_1 \otimes b_2) \otimes b_3 \\
&= \begin{cases}
\frac{1}{\sqrt{2}}(\Btil_i b_1 \otimes b_2) \otimes b_3 & \IF A_3=A_4+1 \AND \beta_i(\Btil_i b_1) = \beta_i(b_1)-2, \\
(\Btil_i b_1 \otimes b_2) \otimes b_3 & \OW.
\end{cases}
\end{split} \nonumber
\end{align}
Next, consider the case when $A_4 \leq A_5$. Noting that $A_3 = A_4+1$ implies $A_3 = A_5+1$, we obtain
\begin{align}
\begin{split}
\Btil_i b' &= \begin{cases}
\frac{1}{\sqrt{2}}(\Btil_i b_1 \otimes b_2) \otimes b_3 & \IF A_3 = A_5+1 \AND \beta_i(\Btil_i b_1) = \beta_i(b_1)-2, \\
(\Btil_i b_1 \otimes b_2) \otimes b_3 & \OW.
\end{cases}
\end{split} \nonumber
\end{align}
In each case, we have
$$
\Btil_i b = \Btil_i b'.
$$
\item When $A_1, A_2 \leq A_3 > A_4,A_5$ and $\beta_{\tau(i)}(b_1) \neq A_3$. In this case, it never happens that $\beta_i(\Btil_i b_1) = \beta_i(b_1)-2$, and hence,
\begin{align}
\begin{split}
\Btil_i b &= \begin{cases}
\frac{1}{\sqrt{2}}(\Btil_i b_1 \otimes (b_2 \otimes b_3) + b_1 \otimes \Ftil_i (b_2 \otimes b_3)) & \IF \max(A_1,A_2) = A_3, \\
\Btil_i b_1 \otimes (b_2 \otimes b_3) & \OW 
\end{cases} \\
&= \begin{cases}
\frac{1}{\sqrt{2}}(\Btil_i b_1 \otimes (b_2 \otimes b_3) + b_1 \otimes (b_2 \otimes \Ftil_i b_3)) & \IF A_2 < A_1 = A_3, \\
\frac{1}{\sqrt{2}}(\Btil_i b_1 \otimes (b_2 \otimes b_3) + b_1 \otimes (\Ftil_i b_2 \otimes b_3)) & \IF A_1 \leq A_2 = A_3, \\
\Btil_i b_1 \otimes (b_2 \otimes b_3) & \OW.
\end{cases}
\end{split} \nonumber
\end{align}

On the other hand, since $A_2 \leq A_3 > A_4$ and $\beta_{\tau(i)}(b_1) \neq A_3$, we have
$$
\beta_{\tau(i)}(b_1 \otimes b_2) = \max(A_3-1,A_2)-\wt_i(b_2) = \begin{cases}
B_i(b') & \IF A_2 = A_3, \\
B_i(b')-1 & \IF A_2 < A_3,
\end{cases}
$$
and hence,
$$
\Btil_i b' = \begin{cases}
\frac{1}{\sqrt{2}}(\Btil_i(b_1 \otimes b_2) \otimes b_3) & \IF A_2 = A_3 = A_5+1, \\
&\AND \beta_i(\Btil_i(b_1 \otimes b_2)) = \beta_i(b_1 \otimes b_2)-2, \\
\frac{1}{\sqrt{2}}(\Btil_i(b_1 \otimes b_2) \otimes b_3 + (b_1 \otimes b_2) \otimes \Ftil_i b_3) & \IF A_1 = A_3 > A_2, \\
\Btil_i(b_1 \otimes b_2) \otimes b_3 & \OW.
\end{cases}
$$
Since $\beta_{\tau(i)}(b_1) \neq A_3$, it never happens that $\beta_i(\Btil_i b_1) = \beta_i(b_1)-2$. Hence, by identities \eqref{calc}, it never happens that $\beta_i(\Btil_i(b_1 \otimes b_2)) = \beta_i(b_1 \otimes b_2)-2$.

Now, consider the case when $A_1 > A_2$. In this case, we have $A_2 < A_3$, and hence
\begin{align}
\begin{split}
\Btil_i b' &= \begin{cases}
\frac{1}{\sqrt{2}}((\Btil_i b_1 \otimes b_2) \otimes b_3 + (b_1 \otimes b_2) \otimes \Ftil_i b_3) & \IF A_1 = A_3, \\
\Btil_i(b_1 \otimes b_2) \otimes b_3 & \IF A_1 < A_3
\end{cases} \\
&= \begin{cases}
\frac{1}{\sqrt{2}}((\Btil_i b_1 \otimes b_2) \otimes b_3 + (b_1 \otimes b_2) \otimes \Ftil_i b_3) & \IF A_1 = A_3, \\
(\Btil_i b_1 \otimes b_2) \otimes b_3 & \IF A_1 < A_3.
\end{cases}
\end{split} \nonumber
\end{align}

Next, consider the case when $A_1 \leq A_2$. In this case, we have
\begin{align}
\begin{split}
\Btil_i b' &= \Btil_i(b_1 \otimes b_2) \otimes b_3 \\
&= \begin{cases}
\frac{1}{\sqrt{2}}((\Btil_i b_1 \otimes b_2) \otimes b_3 + (b_1 \otimes \Ftil_i b_2) \otimes b_3) & \IF A_2 = A_3, \\
(\Btil_i b_1 \otimes b_2) \otimes b_3 & \IF A_2 < A_3.
\end{cases}
\end{split} \nonumber
\end{align}

In each case, we have
$$
\Btil_i b = \Btil_i b'.
$$
\item When $A_1,A_2,A_3 \leq A_4 > A_5$ and $A_1 > A_2$. In this case, we have
\begin{align}
\begin{split}
&\Etil_{\tau(i)}(b_2 \otimes b_3) = \Etil_{\tau(i)}b_2 \otimes b_3, \\
&\vphi_i(\Etil_{\tau(i)}b_2 \otimes b_3) = \vphi_i(b_3)+\wt_i(b_2)-1 = \vphi_i(b_2 \otimes b_3)-1 = A_1-1, \\
&\Ftil_i(b_2 \otimes b_3) = b_2 \otimes \Ftil_i b_3, \\
&\vphi_{\tau(i)}(b_2 \otimes \Ftil_i b_3) = \vphi_{\tau(i)}(b_2) = \vphi_{\tau(i)}(b_2 \otimes b_3) = A_4,
\end{split} \nonumber
\end{align}
and hence,
\begin{align}
\begin{split}
\Btil_i b &= \begin{cases}
\frac{1}{\sqrt{2}} b_1 \otimes (\Etil_{\tau(i)}b_2 \otimes b_3) & \IF \beta_{\tau(i)}(b_1) = A_3 = A_4, \\
\frac{1}{\sqrt{2}}(b_1 \otimes (\Etil_{\tau(i)}b_2 \otimes b_3) + b_1 \otimes (b_2 \otimes \Ftil_i b_3)) & \IF A_1 = A_4 > \beta_{\tau(i)}(b_1), \\
b_1 \otimes (\Etil_{\tau(i)}b_2 \otimes b_3) & \OW
\end{cases} \\
&= \begin{cases}
\frac{1}{\sqrt{2}} b_1 \otimes (\Etil_{\tau(i)}b_2 \otimes b_3) & \IF \beta_{\tau(i)}(b_1) = A_3 = A_4, \\
\frac{1}{\sqrt{2}}(b_1 \otimes (\Etil_{\tau(i)}b_2 \otimes b_3) + b_1 \otimes (b_2 \otimes \Ftil_i b_3)) & \IF A_1 = A_4 > A_3, \OR \\
&\ A_1 = A_4 = A_3 \neq \beta_{\tau(i)}(b_1), \\
b_1 \otimes (\Etil_{\tau(i)}b_2 \otimes b_3) & \OW
\end{cases}
\end{split} \nonumber
\end{align}

On the other hand, we have
$$
\beta_{\tau(i)}(b_1 \otimes b_2) = \begin{cases}
B_i(b')-1 & \IF A_4 > A_3,A_2, \OR \\
&\ A_4 \leq A_3 > A_2 \AND \beta_{\tau(i)}(b_1) \neq A_3, \\
B_i(b') & \OW,
\end{cases}
$$
and hence,
\begin{align}
\begin{split}
\Btil_i b' &= \begin{cases}
\frac{1}{\sqrt{2}}\Btil_i(b_1 \otimes b_2) \otimes b_3 & \IF A_4 = A_5+1, \\
& \AND \beta_i(\Btil_i(b_1 \otimes b_2)) = \beta_i(b_1 \otimes b_2)-2, \\
\frac{1}{\sqrt{2}}(\Btil_i(b_1 \otimes b_2) \otimes b_3 + (b_1 \otimes b_2) \otimes \Ftil_i b_3) & \IF A_1 = A_4 > A_3, \OR \\
&\ A_1 = A_4 = A_3 \neq \beta_{\tau(i)}(b_1), \\
\Btil_i(b_1 \otimes b_2) \otimes b_3 & \OW.
\end{cases}
\end{split} \nonumber
\end{align}
Since $A_2,A_3 \leq A_4$, we have $\Btil_i(b_1 \otimes b_2) \in \clB_1 \otimes \clB_2$ if and only if $\Btil_i(b_1 \otimes b_2) = b_1 \otimes \Etil_{\tau(i)} b_2 \neq 0$. In this case it never happens that $\beta_i(b_1 \otimes \Etil_{\tau(i)}b_2) = \beta_i(b_1 \otimes b_2) -2$ since we have.
\begin{align}
\begin{split}
&F_i(b_1 \otimes \Etil_{\tau(i)}b_2) = \vphi_i(\Etil_{\tau(i)}b_2) \leq A_2, \\
&B_i(b_1 \otimes \Etil_{\tau(i)}b_2) = A_3, \\
&E_i(b_1 \otimes \Etil_{\tau(i)}b_2) = A_4+1.
\end{split} \nonumber
\end{align}
Hence, we obtain
\begin{align}
\begin{split}
\Btil_i b' &= \begin{cases}
\frac{1}{\sqrt{2}}(\Btil_i(b_1 \otimes b_2) \otimes b_3 + (b_1 \otimes b_2) \otimes \Ftil_i b_3) & \IF A_1 = A_4 > A_3, \OR \\
&\ A_1 = A_4 = A_3 \neq \beta_{\tau(i)}(b_1), \\
\Btil_i(b_1 \otimes b_2) \otimes b_3 & \OW
\end{cases} \\
&= \begin{cases}
\frac{1}{\sqrt{2}}((b_1 \otimes \Etil_{\tau(i)}b_2) \otimes b_3 + (b_1 \otimes b_2) \otimes \Ftil_i b_3) & \IF A_1 = A_4 > A_3, \OR \\
&\ A_1 = A_4 = A_3 \neq \beta_{\tau(i)}(b_1), \\
\frac{1}{\sqrt{2}}(b_1 \otimes \Etil_{\tau(i)}b_2) \otimes b_3 & \IF A_3 = A_4 = \beta_{\tau(i)}(b_1), \\
(b_1 \otimes \Etil_{\tau(i)}b_2) \otimes b_3 & \OW.
\end{cases}
\end{split} \nonumber
\end{align}
Thus, we conclude
$$
\Btil_i b = \Btil_i b'.
$$
\item When $A_1,A_2,A_3 \leq A_4 > A_5$ and $A_1 \leq A_2$. In this case, we have
\begin{align}
\begin{split}
&\Etil_{\tau(i)}(b_2 \otimes b_3) = \Etil_{\tau(i)}b_2 \otimes b_3, \\
&\vphi_i(\Etil_{\tau(i)}b_2 \otimes b_3) = \vphi_i(\Etil_{\tau(i)} b_2) = \begin{cases}
\vphi_i(b_2 \otimes b_3) & \IF \vphi_i(\Etil_{\tau(i)} b_2) = \vphi_i(b_2), \\
\vphi_i(b_2 \otimes b_3)-1 & \IF \vphi_i(\Etil_{\tau(i)} b_2) = \vphi_i(b_2)-1,
\end{cases} \\
&\Ftil_i(b_2 \otimes b_3) = \Ftil_i b_2 \otimes b_3, \\
&\vphi_{\tau(i)}(\Ftil_i b_2 \otimes b_3) = \vphi_{\tau(i)}(\Ftil_i b_2) = \begin{cases}
\vphi_{\tau(i)}(b_2 \otimes b_3)+1 & \IF \vphi_{\tau(i)}(\Ftil_i b_2) = \vphi_{\tau(i)}(b_2)+1, \\
\vphi_{\tau(i)}(b_2 \otimes b_3) & \IF \vphi_{\tau(i)}(\Ftil_i b_2) = \vphi_{\tau(i)}(b_2).
 \end{cases}
\end{split} \nonumber
\end{align}
and hence,
\begin{align}
\begin{split}
\Btil_i b &= \begin{cases}
\frac{1}{\sqrt{2}} b_1 \otimes (\Etil_{\tau(i)}b_2 \otimes b_3) & \IF A_2 = A_4 \AND \vphi_i(\Etil_{\tau(i)} b_2) = \vphi_i(b_2), \\
& \OR A_3 = A_4 = \beta_{\tau(i)}(b_1) \AND \vphi_i(\Etil_{\tau(i)} b_2) < A_4, \\
\frac{1}{\sqrt{2}}(b_1 \otimes (\Etil_{\tau(i)}b_2 \otimes b_3) & \IF A_2 = A_4 > \beta_{\tau(i)}(b_1), \\
 + b_1 \otimes (\Ftil_i b_2 \otimes b_3)) &\AND \vphi_i(\Etil_{\tau(i)} b_2) = \vphi_i(b_2)-1, \\
b_1 \otimes (\Etil_{\tau(i)}b_2 \otimes b_3) & \OW.
\end{cases}
\end{split} \nonumber
\end{align}

On the other hand, we have
\begin{align}
\begin{split}
\Btil_i b' &= \begin{cases}
\frac{1}{\sqrt{2}}\Btil_i(b_1 \otimes b_2) \otimes b_3 & \IF A_4 = A_5+1, \\
& \AND \beta_i(\Btil_i(b_1 \otimes b_2)) = \beta_i(b_1 \otimes b_2)-2, \\
\frac{1}{\sqrt{2}}(\Btil_i(b_1 \otimes b_2) \otimes b_3 + (b_1 \otimes b_2) \otimes \Ftil_i b_3) & \IF A_1 = A_4 \AND \beta_{\tau(i)}(b_1 \otimes b_2) \neq B_i(b'), \\
\Btil_i(b_1 \otimes b_2) \otimes b_3 & \OW.
\end{cases}
\end{split} \nonumber
\end{align}
As in the previous case, it never happens that $\beta_i(\Btil_i(b_1 \otimes b_2)) = \beta_i(b_1 \otimes b_2)-2$. Also, since $A_1 \leq A_2$, the condition that $A_1 = A_4$ implies $A_2 = A_4$, which, in turn, shows that $\beta_{\tau(i)}(b_1 \otimes b_2) = B_i(b')$.
Hence, we obtain
\begin{align}
\begin{split}
\Btil_i b' &= \Btil_i(b_1 \otimes b_2) \otimes b_3 \\
&= \begin{cases}
\frac{1}{\sqrt{2}}(b_1 \otimes \Etil_{\tau(i)} b_2) \otimes b_3 & \IF A_2 = A_4 \AND \vphi_i(\Etil_{\tau(i)} b_2) = \vphi_i(b_2), \\
&\OR A_3 = A_4 = \beta_{\tau(i)}(b_1) \AND \vphi_i(\Etil_{\tau(i)} b_2) < A_4, \\
\frac{1}{\sqrt{2}}((b_1 \otimes \Etil_{\tau(i)} b_2) \otimes b_3 & \IF A_2 = A_4 > \beta_{\tau(i)}(b_1), \\
 + (b_1 \otimes \Ftil_i b_2) \otimes b_3) &\AND \vphi_i(\Etil_{\tau(i)} b_2) = \vphi_i(b_2)-1, \\
(b_1 \otimes \Etil_{\tau(i)} b_2) \otimes b_3 & \OW.
\end{cases} 
\end{split} \nonumber
\end{align}
Thus, we conclude
$$
\Btil_i b = \Btil_i b'.
$$
\item When $A_1,A_2,A_3,A_4 \leq A_5$ and $A_1 > A_2$. In this case, we have
\begin{align}
\begin{split}
&\Etil_{\tau(i)}(b_2 \otimes b_3) = b_2 \otimes \Etil_{\tau(i)} b_3, \\
&\vphi_i(b_2 \otimes \Etil_{\tau(i)} b_3) = \vphi_i(\Etil_{\tau(i)}b_3)+\wt_i(b_2) = \begin{cases}
\vphi_i(b_2 \otimes b_3) & \IF \vphi_i(\Etil_{\tau(i)} b_3) = \vphi_i(b_3), \\
\vphi_i(b_2 \otimes b_3)-1 & \IF \vphi_i(\Etil_{\tau(i)} b_3) = \vphi_i(b_3)-1,
 \end{cases} \\
&\Ftil_i(b_2 \otimes b_3) = b_2 \otimes \Ftil_i b_3, \\
&\vphi_{\tau(i)}(b_2 \otimes \Ftil_i b_3) = \vphi_{\tau(i)}(\Ftil_i b_3) = \begin{cases}
\vphi_{\tau(i)}(b_2 \otimes b_3)+1 & \IF \vphi_{\tau(i)}(\Ftil_i b_3) = \vphi_{\tau(i)}(b_3)+1, \\
\vphi_{\tau(i)}(b_2 \otimes b_3) & \IF \vphi_{\tau(i)}(\Ftil_i b_3) = \vphi_{\tau(i)}(b_3),
 \end{cases}
\end{split} \nonumber
\end{align}
and hence,
$$
\Btil_i b = \begin{cases}
\frac{1}{\sqrt{2}} b_1 \otimes (b_2 \otimes \Etil_{\tau(i)} b_3) & \IF A_1 = A_5 \AND \vphi_i(\Etil_{\tau(i)} b_3) = \vphi_i(b_3), \\
& \OR A_3 = A_5 = \beta_{\tau(i)}(b_1) \AND \vphi_i(\Etil_{\tau(i)} b_3)+\wt_i(b_2) < A_5, \\
\frac{1}{\sqrt{2}}(b_1 \otimes (b_2 \otimes \Etil_{\tau(i)} b_3) & \IF A_1 = A_5 > \beta_{\tau(i)}(b_1), \\
+ b_1 \otimes (b_2 \otimes \Ftil_i b_3)) &\AND \vphi_i(\Etil_{\tau(i)} b_3) = \vphi_i(b_3)-1, \\
b_1 \otimes (b_2 \otimes \Etil_{\tau(i)} b_3) & \OW,
\end{cases}
$$

On the other hand, we have
$$
\Btil_i b' = \begin{cases}
\frac{1}{\sqrt{2}} (b_1 \otimes b_2) \otimes \Etil_{\tau(i)} b_3 & \IF A_1 = A_5 \AND \vphi_i(\Etil_{\tau(i)} b_3) = \vphi_i(b_3), \\
& \OR A_5 = \max(A_2,A_3,A_4),\ \beta_{\tau(i)}(b_1 \otimes b_2) = B_i(b'), \\
& \AND \vphi_i(\Etil_{\tau(i)} b_3) < A_5-\wt_i(b_2), \\
\frac{1}{\sqrt{2}}((b_1 \otimes b_2) \otimes \Etil_{\tau(i)} b_3 & \IF A_1 = A_5,\ \beta_{\tau(i)}(b_1 \otimes b_2) < A_5-\wt_i(b_2), \\
+ (b_1 \otimes b_2) \otimes \Ftil_i b_3) &\AND \vphi_i(\Etil_{\tau(i)} b_3) = \vphi_i(b_3)-1, \\
(b_1 \otimes b_2) \otimes \Etil_{\tau(i)} b_3 & \OW.
\end{cases}
$$
Since we have $A_2 < A_1$, we have $A_5 = \max(A_2,A_3,A_4)$ and $\beta_{\tau(i)}(b_1 \otimes b_2) = B_i(b')$ if and only if $A_3 = A_5$ and $\beta_{\tau(i)}(b_1) = A_3$. Also, we have $\beta_{\tau(i)}(b_1 \otimes b_2) < A_5-\wt_i(b_2)$ if and only if either $A_3 < A_5$, or $A_3 = A_5$ and $\beta_{\tau(i)}(b_1 \otimes b_2) \neq B_i(b')$. This condition is equivalent to that either $A_3 < A_5$, or $A_3 = A_5$ and $\beta_{\tau(i)}(b_1) < A_3$. This is, in turn, equivalent to $\beta_{\tau(i)}(b_1) < A_5$. Therefore, we obtain
$$
\Btil_i b' = \begin{cases}
\frac{1}{\sqrt{2}} (b_1 \otimes b_2) \otimes \Etil_{\tau(i)} b_3 & \IF A_1 = A_5 \AND \vphi_i(\Etil_{\tau(i)} b_3) = \vphi_i(b_3), \\
& \OR A_3 = A_5 = \beta_{\tau(i)}(b_1) \AND \vphi_i(\Etil_{\tau(i)} b_3) < A_5-\wt_i(b_2), \\
\frac{1}{\sqrt{2}}((b_1 \otimes b_2) \otimes \Etil_{\tau(i)} b_3 & \IF A_1 = A_5 > \beta_{\tau(i)}(b_1), \\
+  (b_1 \otimes b_2) \otimes \Ftil_i b_3) &\AND \vphi_i(\Etil_{\tau(i)} b_3) = \vphi_i(b_3)-1, \\
(b_1 \otimes b_2) \otimes \Etil_{\tau(i)} b_3 & \OW.
\end{cases}
$$
Thus, we conclude
$$
\Btil_i b = \Btil_i b'.
$$
\item When $A_1,A_2,A_3,A_4 \leq A_5$ and $A_1 \leq A_2$. In this case, we have
\begin{align}
\begin{split}
&\Etil_{\tau(i)}(b_2 \otimes b_3) = b_2 \otimes \Etil_{\tau(i)} b_3, \\
&\vphi_i(b_2 \otimes \Etil_{\tau(i)} b_3) = \vphi_i(b_2) = \vphi_i(b_2 \otimes b_3), \\
&\Ftil_i(b_2 \otimes b_3) = \Ftil_i b_2 \otimes b_3, \\
&\vphi_{\tau(i)}(\Ftil_i b_2 \otimes b_3) = \vphi_{\tau(i)}(b_3)+1 = \vphi_{\tau(i)}(b_2 \otimes b_3)+1,
\end{split} \nonumber
\end{align}
and hence,
$$
\Btil_i b = \begin{cases}
\frac{1}{\sqrt{2}} b_1 \otimes (b_2 \otimes \Etil_{\tau(i)} b_3) & \IF A_2 = A_5, \OR \\
&\ A_3 = A_5 = \beta_{\tau(i)}(b_1) \AND A_2 < A_5, \\
b_1 \otimes (b_2 \otimes \Etil_{\tau(i)} b_3) & \OW. 
 \end{cases}
$$

On the other hand, since $A_1 = A_5$ implies $A_2 = A_5$, and hence, $\beta_{\tau(i)}(b_1 \otimes b_2) = B_i(b')$, we have
$$
\Btil_i b' = \begin{cases}
\frac{1}{\sqrt{2}}(b_1 \otimes b_2) \otimes \Etil_{\tau(i)} b_3 & \IF A_1 = A_5 \AND \vphi_i(\Etil_{\tau(i)}b_3) = \vphi_i(b_3), \OR \\
 & \ A_5 = \max(A_2,A_3,A_4),\ \beta_{\tau(i)}(b_1 \otimes b_2) = B_i(b'), \\
 & \AND \vphi_i(\Etil_i b_3) < A_5-\wt_i(b_2), \\
(b_1 \otimes b_2) \otimes \Etil_{\tau(i)} b_3 & \OW. 
 \end{cases}
$$
Now, consider the case when $A_1 = A_5$. In this case, we automatically have $A_2 = A_5$ and $\beta_{\tau(i)}(b_1 \otimes b_2) = B_i(b')$. Then, regardless of the value $\vphi_i(\Etil_{\tau(i)} b_3)$, we obtain
$$
\Btil_i b' = \frac{1}{\sqrt{2}}(b_1 \otimes b_2) \otimes \Etil_{\tau(i)} b_3 = \Btil_i b.
$$

Next, consider the case when $A_1 < A_5$. In this case, we always have $\vphi_i(\Etil_{\tau(i)} b_3) \leq A_1-\wt_i(b_2) < A_5-\wt_i(b_2)$. Therefore, we obtain
\begin{align}
\begin{split}
\Btil_i b' &= \begin{cases}
\frac{1}{\sqrt{2}}(b_1 \otimes b_2) \otimes \Etil_{\tau(i)} b_3 & \IF A_5 = \max(A_2,A_3,A_4) \AND \beta_{\tau(i)}(b_1 \otimes b_2) = B_i(b'), \\
(b_1 \otimes b_2) \otimes \Etil_{\tau(i)} b_3 & \OW 
 \end{cases} \\&= \begin{cases}
\frac{1}{\sqrt{2}}(b_1 \otimes b_2) \otimes \Etil_{\tau(i)} b_3 & \IF A_2 = A_5, \OR \\
&\ A_2 < A_3 = A_5 \AND \beta_{\tau(i)}(b_1) = A_3, \\
(b_1 \otimes b_2) \otimes \Etil_{\tau(i)} b_3 & \OW.
 \end{cases}
\end{split} \nonumber
\end{align}
Thus, we conclude
$$
\Btil_i b = \Btil_i b'.
$$

\end{enumerate}

Now, we have exhausted all the cases. Hence, the proof completes.

\subsection*{Declarations}
Data sharing not applicable to this article as no datasets were generated or analyzed during the current study.
The author has no conflicts of interest to declare.

\end{document}